\documentclass{article}
\usepackage[OT2,OT1]{fontenc}
\usepackage{amsfonts,amsmath, amssymb,latexsym,mathrsfs,comment,url}
\usepackage[all,cmtip]{xy}
\usepackage{mathtools}

\usepackage{float}
\def\Z{{\mathbb Z}}
 \DeclareFontFamily{U}{wncy}{}

\def\GL{{\rm GL}}
\def\SO{{\rm SO}}
\def\PGL{{\rm PGL}}

\def\Inv{{\rm Inv}}

\def\Aut{{\rm Aut}}
\def\irr{{\rm irr}}
\def\gen{{\rm gen}}

\def\Vol{{\rm Vol}}
\def\R{{\mathbb R}}

\def\Q{{\mathbb Q}}
\def\H{{\mathcal H}}

\def\Z{{\mathbb Z}}

\def\Q{{\mathbb Q}}

\def\H{{\mathcal H}}

\usepackage{enumitem}
\def \h{{\rm ht}}

\newtheorem{theorem}{Theorem}[section]

\newtheorem{cor}[theorem]{Corollary}
\newtheorem{lemma}[theorem]{Lemma}

\newtheorem{remark}[theorem]{Remark}
\newtheorem{proposition}[theorem]{Proposition}
\newenvironment{proof}{\noindent {\bf Proof:}}{$\Box$ \vspace{2 ex}}

\DeclareFontFamily{U}{wncy}{}
 \DeclareFontShape{U}{wncy}{m}{n}{<->wncyr10}{}
\DeclareSymbolFont{mcy}{U}{wncy}{m}{n}
 \DeclareMathSymbol{\Sh}{\mathord}{mcy}{"58} 

\newenvironment{customproof}[1][]
{\par\noindent{\textbf{Proof of #1:}}\quad}
{$\Box$ \vspace{2 ex}}
\def \Bo{{\mathcal{O}^{(i)}(n,t,X)}}
\def \Bod{{\mathcal{O}^{(i)}_{\delta}(n,t,X)}}
\def \Bods{{\mathcal{O}^{(i)}_{\delta,S}(n,t,X)}}
\def \Bodi{{\mathcal{O}}_{\delta}(n,t,X)}
\def \Bodid{{\mathcal{O}}^{(i)}_{\delta}(n,t,X)}
\def \SV{{SV}_{\delta}}

\def\JR{\frac{27R^2}{48a^2H}+\frac{H^2}{48a^2}}
\def\JP{\frac{27R^2}{48a^2P}+\frac{P^2}{48a^2}}
\def \dw{{t^{-2}dud^\times t  }}
\def \f{ax^4+bx^3y+cx^2y^2+dxy^3+ey^4}
\def \gcd{{\rm gcd}}
\def \det{{\rm det}}
\usepackage{xcolor}
\usepackage{environ}
\definecolor{darkpurple}{rgb}{0.5, 0.0, 0.5}
\newif\ifshowcontent
\showcontentfalse

\definecolor{lightpurple}{rgb}{0.75, 0.6, 0.85}
\usepackage{xcolor,soul}
\sethlcolor{lightpurple}
\usepackage{amscd,amssymb,amsmath}
\usepackage{enumitem}
\usepackage{color}

\author{Fatemehzahra Janbazi}
\title{Boundedness of average rank of elliptic curves ordered by the coefficients}
\begin{document}
\maketitle
\begin{abstract}
We study the average rank of elliptic curves $E_{A,B} : y^2 = x^3 + Ax + B$ over $\mathbb{Q}$, ordered by the height function $h(E_{A,B}) := \max(|A|, |B|)$. Understanding this average rank requires estimating the number of irreducible integral binary quartic forms under the action of $\mathrm{GL}_2(\mathbb{Z})$, where the invariants $I$ and $J$ are bounded by $X$. A key challenge in this estimation arises from working within regions of the quartic form space that expand non-uniformly, with volume and projection of the same order. To address this, we develop a new technique for counting integral points in these regions, refining existing methods and overcoming the limitations of Davenport’s lemma. This leads to a bound on the average size of the 2-Selmer group, yielding an upper bound of 1.5 for the average rank of elliptic curves ordered by $h$.
\end{abstract}

\section{Introduction} 
Any elliptic curve $E$ over $\mathbb{Q}$ can be uniquely expressed in the form  $E_{A,B}: y^2 = x^3 + Ax + B$, where for all primes $p$, if $p^4 \mid A$, then $p^6 \nmid B$.   
It has been conjectured in the work of Gelfand~\cite{golfand} and Katz--Sarnak~\cite{katz} that 50\% of elliptic curves have rank~0 and 50\% have rank~1.
 These densities are expected to hold regardless of whether elliptic curves are ordered by conductor, discriminant, or height. 

A standard way to order elliptic curves is using the naive height $H(E_{A,B})$, defined as $\max\{4|A|^3, 27B^2\}$.
Bhargava and Shankar showed in \cite{main} that when elliptic curves are ordered by $H$, the average rank is at most $1.5$. This result follows from their computation of the average size of the $2$-Selmer group, which they proved to be $3$ under this ordering. Their result aligns with heuristics of Poonen--Rains~\cite{PoonenRains}, which predicts  the average size of $\text{Sel}_2$ should be $3$.  

In this paper, we consider the coefficient height function for elliptic curves, given by  
\begin{equation*}
h(E_{A,B}) \vcentcolon= \max\{|A|, |B|\}.
\end{equation*}  
Ordering elliptic curves by naive height behaves similarly to ordering them by discriminant. Indeed, curves with naive height less than $X$ satisfy $|A| \ll X^{1/3}$ and $|B| \ll X^{1/2}$, which corresponds to having discriminant less than $X$. On the other hand, elliptic curves with coefficients bounded in absolute value by $X$ have discriminants bounded by $O(X^3)$, and those with discriminant less than $X$ form a zero-density subfamily within this set. This induces new challenges in understanding the distribution of elliptic curves when ordered by coefficient height. In particular, it has been observed in \cite{thin} that for certain thin families, the distribution of the $2$-Selmer group does not always follow the predictions of Poonen--Rains \cite{PoonenRains}. 

We analyze the distribution of the $2$-Selmer group and the average rank of elliptic curves when ordered by $h$. In particular, we establish the following results:

\begin{theorem}\label{2-Selmer average}
    When elliptic curves over $\mathbb{Q}$ are ordered by the height $h$, the average size of the $2$-Selmer group is at most $3$.
\end{theorem}

Since the rank of an elliptic curve is at most half the size of its $2$-Selmer group, we obtain the following immediate corollary:

\begin{cor}
    When elliptic curves over $\mathbb{Q}$ are ordered by the height $h$, their average rank is at most $1.5$.
\end{cor}

Although we establish an upper bound, we expect that the average size of the $2$-Selmer group in this setting actually matches the predicted value of $3$; see Remark \ref{uniformity tail estimate}. 

Let $V_{\mathbb{R}}$ denote the space of binary quartic forms with real coefficients, and let $V_{\mathbb{Z}}$ denote the space of binary quartic forms with integer coefficients. The group $\mathrm{GL}_2(\mathbb{Z})$ acts on $V_{\mathbb{Z}}$ via linear substitutions of variables. The correspondence between the $2$-Selmer group of elliptic curves and $\GL_2(\mathbb{Z})$-orbits of integral binary quartic forms was first introduced by Birch and Swinnerton-Dyer \cite{birch1963notes}, and later developed by Cremona \cite{cremona2008algorithms}. For binary quartic forms $f(x, y) = ax^4 + bx^3y + cx^2y^2 + dxy^3 + ey^4$, the ring of invariants for this action is generated by two fundamental quantities, $I(f)$ and $J(f)$, which are polynomial functions of the coefficients of the form. The explicit definitions of these invariants will be provided in Section \ref{sec2}.

It is classically known that for fixed invariants $I$ and $J$, there are only finitely many $\GL_2(\Z)$-orbits of integral binary quartic forms with those invariants, provided that $I$ and $J$ are not both zero. To analyze their distribution, we introduce a height function $h(f)$ for binary quartic forms, defined as follows:

\begin{equation*}
    h(f) \vcentcolon= h(I, J) \vcentcolon= \max\{|I|, |J|\}.
\end{equation*}
This height function is closely related to the coefficient height on elliptic curves through the parametrization. By introducing new ideas, we shift the main challenges in handling this thin family to problems of equidistribution and non-Archimedean analysis, as detailed in the method of proofs section. Combined with Bhargava’s averaging method, this perspective allows us to establish the following result:

\begin{theorem}\label{binary quartic average}
    Let $h^{(i)}_{I,J}$ denote the number of $\GL_2(\Z)$-equivalence classes of irreducible binary quartic forms with $4 - 2i$ real roots and invariants $I$ and $J$. Then:
    \begin{itemize}  
        \item [\textnormal{(a)}] $\displaystyle\sum_{\substack{h(I,J) < X}} h^{(0)}_{I,J} = \frac{\zeta(2)}{27}X^{2} + o(X^2)$
        \item [\textnormal{(b)}] $\displaystyle\sum_{\substack{h(I,J) < X}} h^{(1)}_{I,J} = \frac{2\zeta(2)}{27}X^{2} + o(X^2)$
        \item [\textnormal{(c)}] $\displaystyle\sum_{\substack{h(I,J) < X}} h^{(2)}_{I,J} = \frac{2\zeta(2)}{27}X^{2} + o(X^2)$
\end{itemize} 
\end{theorem} 

Finally, we show that the number of equivalence classes of binary quartic forms per eligible $(I, J)$ as stated in \cite{main}, with a given number of real roots, is constant on average.
\begin{theorem}
    Let $h^{(i)}_{I,J}$ denote the number of $\GL_2(\Z)$-equivalence classes of irreducible binary quartic forms with $4 - 2i$ real roots and invariants $I$ and $J$. Let $n_0 = 4$, $n_1 = 2$, and $n_2 = 2$. Then, for $i = 0, 1, 2$, we have:

    \begin{equation*}
        \lim_{X \to \infty} \frac{\displaystyle\sum_{h(I,J) < X} h^{(i)}_{I,J}}{\displaystyle\sum_{\substack{(I,J) \;\mathrm{eligible}\\ (-1)^i \Delta(I,J) > 0 \\ h(I,J) < X }} 1} = \frac{2\zeta(2)}{n_i}.
\end{equation*}
\end{theorem}
\subsection{Methods of Proofs}
Theorem \ref{binary quartic average} serves as the main technical ingredient in the proof of Theorem \ref{2-Selmer average}. As in Bhargava--Shankar \cite{main}, we begin with the same initial setup. The argument proceeds by applying Bhargava’s averaging method to count the number of orbits. In particular, we need to analyze the integral
\begin{equation}\label{TheIntegral}
    \int_{(tnk)\in \mathcal{F}_{\text{PGL}_2}} \#\big\{f \in V_{\Z,\irr}^{(i)} \cap X^{1/2}(tn)G_0 \cdot L^{(i)} : |J(f)| < X \big\} t^{-2} d^{\times}t  dndk,
\end{equation}
where $G_0$ is a compact, semialgebraic, and left $K$-invariant subset of $\PGL_2(\R)$, and $\mathcal{F}_{\text{PGL}_2}$ denotes Gauss’s fundamental domain for $\PGL_2(\Z) \backslash \PGL_2(\R)$. Additionally, $L^{(i)}$, as defined in \cite{main}, serves as the fundamental set for the action of $\GL_2(\R)$ on $V_{\R}$. The scaling by $X^{1/2}$ arises from the assumption that $|I| < X$; however, once this condition is imposed, the further condition $|J| < X$ must additionally be enforced.

The naive height is a degree 6 homogeneous function in the coefficients of the form $f$. Due to this homogeneity, the primary difficulty in counting the orbits in \cite{main} arises when the ball becomes distorted for large $t$, whereas at $t = 1$, the counting remains well-controlled, allowing the authors to address this issue by cutting off the cusp. This homogeneity also ensures that standard lattice point counting methods from geometry of numbers can be applied effectively. However, the new height function is not homogeneous on $V_{\R}$, and the difficulty in counting appears even in the main ball at $t = 1$. 

Specifically, if we consider a scaled ball by $X^{1/2}$ in the space as it appears in above integral:
\begin{equation*}
     X^{1/2}(nG_0\cdot L^{(i)}),
\end{equation*}
and attempt to count the number of irreducible integral points with $J$-invariant less than $X$, fixing the first four coefficients $(a, b, c, d)$ of the form constrains the last coefficient to an interval of size $O(1)$. This introduces an Archimedean difficulty in counting lattice points, as the volume of the space in which we aim to count lattice points and the volume of its projection onto the first four coefficients are of the same order. Consequently, the approach in \cite{main} does not extend to this setting, making an alternative method necessary.

To analyze the distribution of integral quartic forms within a bounded region scaled by a factor of $X^{1/2}$ under the constraint $|J| < X$, we first fiber over the coefficients $(a, b, c)$ and study the distribution of $(d, e)$. Instead of working directly in the coefficient space, we describe quartic forms in terms of their semi-invariants, as introduced in Section \ref{sec2}. Specifically, we analyze the distribution of the semi-invariants $(I(f), R(f))$, where $I$ is the primary invariant of degree 2 and $R$ is a semi-invariant of degree 3, given by  
\begin{equation*}
    R(f) \vcentcolon= b^3 + 8a^2d - 4abc.
\end{equation*}
Thus, we replace $(d, e)$ with $(I, R)$ over each fiber $(a, b, c)$. This reformulation enables us to impose the constraint $|J| < X$ by utilizing the syzygy relation among the semi-invariants:
\begin{equation*}
    P^3 - 48a^2PI + 64a^3J = -27R^2,
\end{equation*}
where $P$ is given by $P = 8ac - 3b^2$. To incorporate the height restriction in this new framework, we observe that when $P \neq 0$, the syzygy expresses $I$ linearly in terms of $R$ and $J$. This enables us to translate the condition $|J| < X$ into a constraint on the $(I, R)$-space via the relation
\begin{equation*}
\big|I-(\JP)\big|\ll \bigg|\frac{aX}{P}\bigg|.
\end{equation*}
This reformulation makes the application of the height function, our archimedean input, easier but comes at the cost of introducing new non-Archimedean difficulties. For example, if $R$ is also fixed in the fiber over $(a, b, c)$, the permissible range of $I$ is naturally constrained to an interval of size $O(X^{1/2})$. However, instead of counting all possible values of $I$, we restrict our attention to those satisfying a congruence condition modulo $12a$. More precisely, for quartic forms with fixed coefficients $a, b, c$ and semi-invariant $R$, the value of $I$ is uniquely determined modulo $12a$.  Thus, we need to determine the number of possible values of $I$ over each fiber $(a, b, c)$ with fixed $R$, within an interval of size $O(X^{1/2})$, while also satisfies congruent condition whose modulus has the same size.

This observation reduces the Archimedean difficulty in each fiber to a non-Archimedean problem, allowing us to apply equidistribution methods.  To carry this out, we apply Poisson summation over each fiber in the $(R, I)$-space. To make this argument precise, we fix $(a, b, c)$ with $a \neq 0$ and consider the map
\begin{equation*}
    \upsilon_{(a,b,c)}\colon  \R^{2} \rightarrow \R^2, \quad (d,e) \mapsto (I(f),R(f)),
\end{equation*}
where $f=(a,b,c,d,e)$. This map defines a bijection of $\R^2$ to $\R^2$.

Now, assume that $(a,b,c)$ is an element of $\Z^3$ with $a\neq0$. The map $\upsilon_{(a,b,c)}$ sends $\Z^2$ to  
\begin{equation*}
   \Lambda_{(a,b,c)} \vcentcolon = \left\{ (I,R) \in \mathbb{Z}^2 :
   \begin{array}{ll}
        R &\equiv b^3 - 4abc \pmod{8a^2}, \\
        I &\equiv -3b(R - b^3 + 4abc)(8a^2)^{-1}  + c^2 \pmod{12a}
   \end{array} 
   \right\}.
\end{equation*}
This shows that not all values of $(I,R)$ appear for integral quartic forms. In particular, as mentioned, over each fiber, fixing the semi-invariant $R$ determines $I$ modulo $12a$, while $R$ itself is fixed modulo $8a^2$.

During the counting process, by leveraging both the geometric properties of our fibers and the underlying lattice structure, we can control the error more efficiently. These aspects are formalized in Lemma \ref{h-Fourier-Coeff} and Proposition \ref{Finte-part-coeff}.

This perspective, based on the use of semi-invariants, can also be applied to the family of  binary $n$-ic forms. The existence of such semi-invariants and their associated syzygies is ensured by classical results in invariant theory. Moreover, all invariants can be expressed as rational functions of these semi-invariants, suggesting that a similar approach may be effective in more general settings for counting orbits, particularly in cases where classical lattice point techniques fall short.

To complete the proof of Theorem~\ref{2-Selmer average}, we need to apply a sieve to the family of elliptic curves with coefficient height less than $X$. As this sieve involves difficulties, we require the following Theorem

\begin{theorem}\label{IJ-Sieve-introduction}
For a prime $p$, let $\mathcal{W}_p$ represent elliptic curves $E$ over $\Q$ such that $p^2$ divides its discriminant. For any $Q \geq 5$ and $\epsilon > 0$, we have:
\[
\# \bigcup_{p > Q} \left\{ E_{A, B} \in \mathcal{W}_p : h(E_{A,B}) < X \right\} = O_{\epsilon}\left( \frac{X^{2+\epsilon}}{Q} + X^{7/4+\epsilon} \right).
\]
\end{theorem}
To address this, we represent irreducible cubic polynomials as reducible quartic polynomials with a unique root over $\mathbb{P}^1_{\mathbb{Q}}$. Using the averaging method, we fiber over the root of the form to organize the counting process. This approach is particularly effective since the unique root is intrinsically linked to the semi-invariants of the form, allowing us to impose the condition $|J| < X$ in a systematic way.

This paper is organized as follows. In Section \ref{sec2}, we introduce the smooth averaging method and define the semi-invariants of the form. We then restrict the cusp and refine the counting process within the averaging method to minimize errors in the subsequent Fourier analysis. Additionally, we define a new height function that facilitates Fourier calculations while closely approximating the original height function, ensuring that the error remains within an acceptable range.

Section \ref{sec3} is dedicated to bounding the Fourier coefficients, leading to the proof of Theorem \ref{binary quartic average}. In Section \ref{sec4}, we show that the volume of the fibers over $(a,b,c)$ approximates the volume of the ball, compute the main volume contribution, and complete the proof of Theorem \ref{binary quartic average}.

Finally, in Section \ref{sec5}, we establish the connection between binary quartic forms and the $2$-Selmer group of elliptic curves. We derive the necessary uniformity estimates for counting elliptic curves of coefficient height less than $X$, thereby completing the proof of Theorem \ref{2-Selmer average}.

\section*{Acknowledgment } 
It is pleasure to thank Arul Shankar for suggesting the problem and for many helpful conversations. I am also grateful to Jacob Tsimerman and Ila Varma for useful discussions and comments.

\section{Counting classes of integral quartic forms with semi-invariants}\label{sec2}

Let $V_{\R}$ denote the space of binary quartic forms with real coefficients. A form in this space, $f(x,y)$, is written as
\begin{equation*}
f(x,y) = ax^4 + bx^3y + cx^2y^2 + dxy^3 + ey^4,
\end{equation*}
where $a, b, c, d, e \in \R$. When these coefficients are integers, the form is called \textit{integral}.

The group $\GL_2(\R)$ acts on $V_{\R}$ via linear substitutions. For $\gamma \in \GL_2(\R)$, the action is defined as
\begin{equation*}
\gamma \cdot f(x,y) = f((x,y) \cdot \gamma),
\end{equation*}
which gives a well-defined left action, satisfying $\gamma_1 \cdot (\gamma_2 \cdot f) = (\gamma_1 \gamma_2) \cdot f$.

The relative invariants $I$ and $J$ of this action are given by:
\begin{equation}\label{I}
I(f) \vcentcolon= I(a,b,c,d,e) = 12ae - 3bd + c^2,
\end{equation}
\begin{equation}\label{J}
J(f) \vcentcolon= J(a,b,c,d,e) = 72ace + 9bcd - 27ad^2 - 27eb^2 - 2c^3.
\end{equation}
Since $I$ and $J$ are homogeneous of degrees 4 and 6, respectively, they transform under $\gamma$ according to
\begin{equation*}
I(\gamma \cdot f) = \det(\gamma)^4 I(f), \quad J(\gamma \cdot f) = \det(\gamma)^6 J(f).
\end{equation*}
The discriminant $\Delta(f)$ is given by:
\begin{equation*}
\Delta(I,J) = \frac{4I^3 - J^2}{27}.
\end{equation*}
We now introduce the semi-invariants of this action in addition to \( I \) and \( J \), as discussed in \cite{cremona1}. These are defined by:
\begin{align}
    a(f) &\vcentcolon= a, \nonumber \\
    H(f) &\vcentcolon= 8ac - 3b^2,  \label{H} \\
    R(f) &\vcentcolon= b^3 + 8a^2d - 4abc, \label{R} \\
    S(f) &\vcentcolon= \frac{1}{3}(H^2 - 16a^2I). \nonumber
\end{align}

\begin{remark}
    We will use notation $H$ for the given semi-invariant as defined in \cite{cremona1} instead of $P$ as in the introduction. We will not use naive height as symbol $H$. 
\end{remark}
In this section, we consider a binary quartic form $f$ with semi-invariants $a(f)$, $H(f)$, $R(f)$, $I(f)$, and $J(f)$. To simplify notation, we will write these as $a$, $H$, $R$, $I$, and $J$, respectively, with the understanding that they always refer to the semi-invariants of $f$.

With this notation, these semi-invariants of quartic forms satisfy the syzygy equation:
\begin{equation}\label{syzygy}
    H^3 - 48 a^2 H I + 64 a^3 J = -27 R^2.
\end{equation}
When $H \neq 0$, then $J$ can be expressed as:
\begin{equation}\label{JR}
    \frac{4}{3} \frac{a}{H} J = I - \left( \frac{27 R^2}{48 a^2 H} + \frac{H^2}{48 a^2} \right).
\end{equation}
For any constant $C > 0$, we define the height $h_C(f)$ of a binary quartic form $f$ as:
\begin{equation*}
    h_C(f) \vcentcolon= \max \{ |I|, |J| / C \}.
\end{equation*}
Since the height $h_C$ is defined with a constant $C$, this general formulation provides flexibility in applications. In particular, choosing $C = 36$ establishes a connection between the height of binary quartic forms and the $2$-Selmer group of elliptic curves when ordered by the height $h_e$ on their coefficients. Crucially, the choice of $C$ does not affect the structure of the proof, allowing us to proceed with a general $C$.

The action of $\GL_2(\mathbb{Z})$ on $V_{\mathbb{Z}}$, the space of integral quartic forms, is closed. Our goal is to count the number of orbits in $V_{\mathbb{Z}}$ with height bounded by $X$. Specifically, we seek to determine the number of $\GL_2(\mathbb{Z})$-equivalence classes of binary quartic forms with height at most $X$ and a specified number of real roots.

For $i = 0, 1, 2$, let $V_\mathbb{R}^{(i)}$ be the set of binary quartic forms with nonzero discriminant, where each form has $i$ pairs of complex conjugate roots and $4 - 2i$ real roots in $\mathbb{P}^1_\mathbb{C}$.

We are interested in counting irreducible $\GL_2(\mathbb{Z})$-classes within a $\GL_2(\mathbb{Z})$-invariant subset $S \subset V_{\mathbb{Z}}$, subject to the height condition $h_C(f) < X$. Let $N_C(S; X)$ denote the number of such classes. The following theorem, a restatement of Theorem \ref{binary quartic average}, provides the asymptotic count of $\GL_2(\mathbb{Z})$-classes of irreducible binary quartic forms with bounded height.

\begin{theorem}\label{The-Main-counting-Theo} We have
\begin{itemize}
\item[\textnormal{(a)}] 
$N_{C}(V_{\Z}^{(0)};X)=\frac{
C\zeta(2)}{27}X^2+o(X^2)$
\item[\textnormal{(b)}] $N_{C}(V_{\Z}^{(1)};X)=\frac{2C\zeta(2)}{27}X^2+o(X^2)$
\item[\textnormal{(c)}] $N_{C}(V_{\Z}^{(2+)};X)=\frac{C\zeta(2)}{27}X^2+o(X^2)$
\item[\textnormal{(d)}] $N_{C}(V_{\Z}^{(2-)};X)=\frac{C\zeta(2)}{27}X^2+o(X^2)$
\end{itemize}
\end{theorem}
%%%%%%%%%%%%%%%%%%%%%%%%%%%%%%%%%%%%%%%%%%%%%%%%%%%%%%%%%%%%%%%%%%%%%%%%%%%%%%%%%%%%%%%%%%%%%%%%%%%%%%%%%%%%%%%%%%%%%%%%%%%%%%%%%%%%%%%%%%%%%%%%%%%
\subsection{Smoothing the averaging}
From now on, assume that $C$ is fixed, and we will denote $h_C$ by $h$. Consider the twisted action of $\GL_2(\R)$ on $V_{\R}$, which is defined for any $\gamma \in \GL_2(\R)$ and $f \in V_{\R}$ by
\begin{equation*}
\gamma \cdot f(x,y) \coloneqq \frac{f((x,y)\gamma)}{\det(\gamma)^2}.
\end{equation*}
This action induces an action of $\mathrm{PGL}_2(\mathbb{R})$ on $V_{\mathbb{R}}$. The orbits of $\mathrm{GL}_2(\mathbb{Z})$ on $V_{\mathbb{Z}}$ and $\mathrm{PGL}_2(\mathbb{Z})$ on $V_{\mathbb{Z}}$ coincide. Therefore, counting orbits of $\mathrm{GL}_2(\mathbb{Z})$ with bounded height is equivalent to counting orbits of $\mathrm{PGL}_2(\mathbb{Z})$ with  bounded height. To count the number of $\mathrm{PGL}_2(\mathbb{Z})$ orbits, we apply Bhargava's averaging method, as described in \cite{avg1,avg2}. Given the need for Poisson summation in this process, we use the smooth averaging method developed in \cite{moment}.

Let $L^{(i)}$ be the fundamental sets as defined in \cite{main}, and set $R^{(i)} = \mathbb{R}^{+} \cdot L^{(i)}$. For any $(I, J) \in \mathbb{R}^2$ with $\Delta(I, J) > 0$, the sets $R^{(0)}$, $R^{(2+)}$, and $R^{(2-)}$ each contain exactly one point with invariants $(I, J)$. Likewise, for $(I, J) \in \mathbb{R}^2$ with $\Delta(I, J) < 0$, then $R^{(1)}$ contains exactly one such point. Therefore, $R^{(i)}$ serves as a fundamental set for the action of $\mathrm{PGL}_2(\mathbb{R})$ on $V_{\mathbb{R}}^{(i)}$.

Define $\mathcal{I}^{(i)}$ as the set of pairs $(I, J) \in \mathbb{R}^2$ such that $(-1)^i \Delta(I, J) > 0$. We introduce a map $\kappa^{(i)}: \mathcal{I}^{(i)} \to R^{(i)}$. This map $\kappa^{(i)}$ assigns to each $(I, J) \in \mathcal{I}^{(i)}$ the unique quartic form $f \in R^{(i)}$ with invariants $(I, J)$. Any pair $(I, J) \in \mathcal{I}^{(i)}$ can be uniquely written as $(I, J) = (\lambda^2 I_0, \lambda^3 J_0)$, where $(I_0, J_0)$ has a naive height of one and $\lambda \in \mathbb{R}$. The map $\kappa^{(i)}$ then sends $(I, J)$ to the element $\lambda \cdot \kappa^{(i)}(I_0, J_0)$, which is the unique quartic form in $R^{(i)}$ with invariants $(I, J)$. This construction defines a section for the invariant map from $V_{\mathbb{R}}^{(i)}$ to $\mathbb{R}^2$, mapping each quartic form $f$ to its invariants $(I(f), J(f))$.

Let $\mathcal{J}^{(i)}$ be the subset of $(I, J) \in \mathbb{R}^2$ such that $(-1)^{i}\Delta(I, J) > 0$, $|I| < 1$, and $|J| < 2$. The map $\kappa^{(i)}$ sends $\mathcal{J}^{(i)}$ to $\mathbb{R}^{+}_{<1} \cdot L^{(i)}$.

Fix a compactly supported $K$-invariant smooth function $\theta$ on $\PGL_2(\mathbb{R})$. Let $\omega^{(i)}$ be a smooth function on $\mathbb{R}^2$, compactly supported in $\mathcal{J}^{(i)}$. For a quartic form $f$, we denote $\omega^{(i)}(I(f), J(f))$ simply as $\omega^{(i)}(f)$ to simplify notation. Using this, we define the function $\Psi[\omega^{(i)}]: V_{\mathbb{R}} \to \mathbb{R}$ as follows:

\begin{equation*}
\Psi[\omega^{(i)}](f) \vcentcolon=  \left(\sum _{\substack{(g,(I,J))\in \PGL_2(\R) \times \mathcal{J}^{(i)}\\  g \cdot \kappa^{(i)}(I,J) = f }} \theta(g) \right) \cdot \omega^{(i)}(f).
\end{equation*}

The function $\Psi[\omega^{(i)}]$ is non-zero only for $f$ in $\mathrm{PGL}_2(\mathbb{R}) \left(\mathbb{R}^{+}_{<1} \cdot L^{(i)}\right)$. Note that the above sum has a fixed length and is finite for any $f$ in $\mathrm{PGL}
_2(\mathbb{R}) \left(\mathbb{R}^{+}_{<1} \cdot L^{(i)}\right)$. Moreover, $\Psi[\omega^{(i)}]$ is also a smooth and compactly supported function on $V_{\mathbb{R}}$.

The height function $\h_X$ is defined as the characteristic function of the set $\{f \in V_{\mathbb{R}} : h(f) < X\}$, where $h(f)$ denotes the height of $f$. Specifically, we define:

\begin{equation*} \h_X(f) \vcentcolon= \h_{X}(I,J) \vcentcolon= \chi_{(-1,1)}\left(\frac{I}{X}\right) \chi_{(-1,1)}\left(\frac{J}{CX}\right), \end{equation*} where $\chi_{(-1,1)}$ denotes the characteristic function of $(-1,1)$ in $\R$.

Gauss' fundamental domain for the action of $\PGL_2(\mathbb{Z})$ on $\PGL_2(\mathbb{R})$, denoted by $\mathcal{F}_{\PGL_2}$, consists of elements of the form $ntk$, where $n$ ranges over $N(t)$, $t$ belongs to $A'$, and $k$ belongs to $K$.

\begin{equation*}
N(t) \vcentcolon= \left\{ \begin{pmatrix} 1 & 0\\ n & 1\end{pmatrix}: n \in i(t) \right\}
, \,
A'= \left\{\begin{pmatrix} t^{-1} & \\  & t\end{pmatrix}: t \geq  \frac{\sqrt[4]{3}}{\sqrt{2}} \right\},
\end{equation*} 
here $i(t)$ is a subset of $[-1/2, 1/2]$ that depends on $t$, and it is the entire interval $[-1/2, 1/2]$ when $t > 1$. The group $K$ represents $\SO_2(\mathbb{R})$. Let the Haar measure $d\gamma$ on $\PGL_2(\mathbb{R})$ be defined by $d\gamma = \dw$ with $dk$ normalized such that the volume of $K$ is one.

Following the averaging method developed in \cite{moment}, we introduce the function $\sigma^{(i)}(I, J)$, defined as:
\begin{equation*}
\sigma^{(i)}(I, J) \vcentcolon= \h_X(I, J) \cdot \omega^{(i)}(X^{-1}I, X^{-3/2}J),
\end{equation*}
where $\h_X$ represents the height function and $\omega^{(i)}$ is as defined before.

For any subset $L \subset V_{\mathbb{Z}}$, let $L_{\irr} \subset L$ denote the subset of irreducible forms. When $L$ is a $\PGL_2(\mathbb{Z})$-invariant subset of $V_{\mathbb{Z}}$, we consider the counting function $N(\omega^{(i)}, L; X)$, defined by:
\begin{equation*}
N(\omega^{(i)}, L; X) \vcentcolon= \sum_{f \in \frac{L_{\irr} \cap V_{\Z}^{(i)}}{\PGL_2(\mathbb{Z})}} \frac{1}{\#\Aut_{\mathbb{Z}}(f)} \cdot \sigma^{(i)}(I(f), J(f)),
\end{equation*}
where the sum is taken over representatives $f$ of irreducible $\PGL_2(\mathbb{Z})$-orbits in $L\cap V_{\Z}^{(i)}$, and $\Aut_{\mathbb{Z}}(f)$ represents the $\PGL_2(\mathbb{Z})$-stabilizers of $f$.

Let $\Vol(\theta)$ denote the integral $\int_{g \in \PGL_2(\mathbb{R})} \theta(g)  dg$. Let $n_i$ denote the cardinality of the stabilizer of quartic forms in $V_{\mathbb{R}}^{(i)}$ under the action of $\PGL_2(\mathbb{R})$. This is well-defined, as the stabilizer of a quartic form under $\PGL_2(\mathbb{R})$ depends only on the sign of the discriminant.  As calculated in \cite{main}, $n_{0}$, $n_{2+}$, and $n_{2-}$ equal $4$, and $n_1 = 2$. Using the smooth averaging method, we arrive at the following proposition.

\begin{proposition}
For $L$ a $\PGL_2(Z)$-invariant subset of $V_{Z}^{(i)}$, and $X$ large enough, then we have:
\begin{equation}\label{Bahargava-avaraging-smooth}
    N(\omega^{(i)}, L;X) = \frac{1}{n_i \Vol(\theta)}\int_{\gamma \in \mathcal{F}_{\PGL_2}}  \Big( \sum_{f \in L_{\irr}} (\gamma \cdot \Psi[\omega^{(i)}])(X^{-1/2} \cdot f) \cdot \h_{X}(f)\Big) \, d\gamma, 
\end{equation}
where $\big(\gamma \cdot \Psi[\omega^{(i)}]\big)(f)$ is defined to be $\Psi[\omega^{(i)}](\gamma^{-1} \cdot f)$.
\end{proposition}
\begin{proof}
Following the approach in \cite{moment}, we define the function $\mathcal{S}(\sigma^{(i)}, \theta)$ as follows:
\begin{equation*}
    \mathcal{S}(\sigma^{(i)}, \theta)(f) =   \sum _{\substack{(g,(I,J)) \in \PGL_2(\mathbb{R}) \times \mathcal{I}^{(i)} \\  g \cdot \kappa^{(i)}(I, J) = f }} \theta(g) \, \sigma^{(i)}(f).
\end{equation*}
 For any $\gamma \in \mathcal{F}_{\PGL_2}$, we define the action as follows:
\begin{equation*}
\big(\gamma \cdot \mathcal{S}(\sigma^{(i)}, \theta)\big)(f) \coloneqq \mathcal{S}(\sigma^{(i)}, \theta)(\gamma^{-1} \cdot f).
\end{equation*}
Then, Theorem 5.1 in \cite{moment} implies that for any $\PGL_2(\mathbb{Z})$-invariant subset $L$ of integral binary quartic forms the following holds:

\begin{equation*}
\begin{split}
\sum_{f \in \frac{L_{\irr} \cap V_{\Z}^{(i)}}{\PGL_2(\mathbb{Z})}} \frac{1}{\#\Aut_{\mathbb{Z}}(f)} \cdot \sigma^{(i)}(I(f),J(f)) &=  \sum_{f \in \frac{L_{\irr} \cap V_{\Z}^{(i)}}{\PGL_2(\mathbb{Z})}} \frac{1}{\#\Aut_{\mathbb{Z}}(f)} \cdot \omega^{(i)}(X^{-1/2} f) \cdot \h_{X}(f)\\
&=\frac{1}{n_i \, \Vol(\theta)}\int_{\gamma \in \mathcal{F}_{\PGL_2}}  \Big( \sum_{f \in L_{\irr}}  \big(\gamma \cdot \mathcal{S}(\sigma^{(i)}, \theta)\big)(f) \Big) \, d\gamma,
\end{split}
\end{equation*}
For any $f \in R^{(i)}$ with height less than $X$, we can express $f$ as $\lambda \cdot f_0$ for some $f_0 \in L^{(i)}$. Since every element of $L^{(i)}$ satisfies either $|I(f_0)| = 1$ or $|J(f_0)| = 2$, it follows that $\lambda < X^{1/2}$ whenever $X > (C/2)^2$. Hence, $X^{-1/2} \cdot f$ lies in $\mathbb{R}^{+}_{<1} \cdot L^{(i)}$.

Similarly, for any quartic form in $\PGL_2(\mathbb{R}) \cdot R^{(i)}$ with height less than $X$, the rescaled form $X^{-1/2} \cdot f$ is $\PGL_2(\mathbb{R})$-equivalent to a form in $\mathbb{R}^{+}_{<1} \cdot L^{(i)}$, meaning its invariants lie within the region $\mathcal{J}^{(i)}$.

All above and the fact that the height function is invariant under the action of $\PGL_2(\mathbb{R})$, implies:
\begin{equation*}
  \sum_{f \in L_{\irr}} \big(\gamma \cdot \mathcal{S}(\sigma^{(i)}, \theta)\big)(f) = \sum_{f \in L_{\irr}} \big(\gamma \cdot \Psi[\omega^{(i)}]\big)(X^{-1/2} \cdot f) \cdot \h_X(f).
\end{equation*}
This completes the proof.
\end{proof}

Define $\chi^{(i)}$ as the characteristic function of $\mathcal{J}^{(i)}$. A form $f \in V_{\mathbb{Z}}$ is said to be \textit{generic} if it is irreducible and its stabilizer in $\PGL_2(\mathbb{Q})$ is trivial. We denote the set of generic points in a subset $S \subset V_{\mathbb{Z}}$ by $S_{\gen}$.

Our goal is to compute $N(\omega^{(i)}, V_{\mathbb{Z}, \gen}^{(i)}; X)$. To achieve this, we let $\omega^{(i)}$ approach $\chi^{(i)}$ in the limit, ultimately obtaining $N(V_{\mathbb{Z}, \gen}^{(i)}; X)$. This approach works because the function $\chi^{(i)}(X^{-1/2} f) \cdot \h_X(f)$ evaluates to one if $f$ has height less than $X$ and satisfies $(-1)^i \Delta(f) > 0$, and to zero otherwise.

Let $\omega^{(i)}$ be fixed from now on, and for simplicity, we will denote $\Psi[\omega^{(i)}]$ by $\Psi^{(i)}$. 
To compute the integral in equation \eqref{Bahargava-avaraging-smooth}, it is crucial to first understand $\Bo$, which is defined as:

\begin{equation}\label{def Bo} \Bo \vcentcolon = \sum_{f \in V_{\mathbb{Z}, a \neq 0}} (nt \cdot \Psi^{(i)})(X^{-1/2} \cdot f) \cdot \h_{X}(f), \end{equation}
where $V_{\mathbb{Z}, a \neq 0}$ denotes the set of integral quartic forms with a non-zero $a$ coefficient.

We focus on understanding $\Bo$ because Lemma \ref{reducibility lemma} implies the following integral expression for $N(\omega^{(i)}, V_{\mathbb{Z}, \gen}^{(i)}; X)$:

\begin{align*}\label{integral} N(\omega^{(i)}, V_{\mathbb{Z}, \gen}^{(i)}; X)
&= \frac{1}{n_i \Vol(\theta)} \int_{t = \frac{\sqrt[4]{3}}{\sqrt{2}}} \int_{N'(t)} \Big( \sum_{f \in V_{\mathbb{Z}, a \neq 0}} (nt \cdot \Psi[\chi^{(i)}])(X^{-1/2} \cdot f)  \cdot \h_{X}(f) \Big)\, \dw \\ & + O_{\epsilon}(X^{7/4+\epsilon}). \end{align*}

By definition, the function $\Psi^{(i)}$ has bounded support in $V_{\mathbb{R}}$ since both $\theta$ and $\mathbb{R}^{+}_{<1} \cdot L^{(i)}$ are bounded. Moreover, there exists a constant $M$ such that for any $n \in N(t)$, the support of $n \cdot \Psi$ remains bounded by $M$. This ensures that $n \cdot \Psi$ is nonzero only for $f \in V_{\mathbb{R}}$ whose coefficients have absolute values bounded by $M$.

The following lemma allows us to focus on $\Bo$ for values of $t$ as small as $X^{\delta_t}$, where $\delta_t$ is a positive constant.

\begin{lemma}\label{eBound}
Assume $(a, b, c, d)$ are integers with $a \neq 0$ and $8ac \neq b^2$. Then
\begin{equation*}
\# \Big\{e \in \mathbb{Z} : |I(a, b, c, d, e)| < X \text{ and } |J(a, b, c, d, e)| < X\Big\} = O\left(\frac{X}{|H|}\right),
\end{equation*}
where $H = 8ac - 3b^2$.
\end{lemma}
\begin{proof}
Given that $a$, $b$, $c$, and $d$ are fixed, the quantities $H$ and $R$, as defined in equations \eqref{H} and \eqref{R}, are also determined. Consequently, equation \eqref{JR} implies that $I$ lies within the interval

\begin{equation*}
\Big(\alpha(a,b,c,d)-\frac{4}{3}\left|\frac{aC}{H}\right|X, \alpha(a,b,c,d)+\frac{4}{3}\left|\frac{aC}{H}\right|X\Big)
\end{equation*}
where $\alpha(a,b,c,d)$ is
\begin{equation*}
\JR.
\end{equation*}
By looking at equation\eqref{I}, then $I$ must satisfy the congruence $I \equiv -3bd + c^2 \pmod{12a}$. Consequently, the number of possible values for $e$ corresponds to the number of values for $I$ under these conditions, which is bounded by $O\left(\frac{X}{|H|}\right)$. 
\end{proof}

This lemma highlights the difficulty in estimating $\Bo$. When a height restriction of less than $X$ is imposed in $V_{\mathbb{R}}$, the variable $e$ varies within an interval of size $O(X/|H|)$ for fixed $(a, b, c, d)$. This variation complicates the lattice point count, particularly when $H$ is large and close to $X$.

By losing the uniformly expanded region for describing the space of quartic forms with height less than $X$, our understanding of this space becomes limited. To address this, instead of counting quartic forms based on their coefficients, we analyze them through their semi-invariants, which provide a clearer structure via the syzygy equation on the height function. To implement this approach, we introduce the following map:
\begin{align*}
    \upsilon\colon &V_{\R} \longrightarrow V_{\R}^s, \\
    &(a, b, c, d, e) \mapsto (a, b, c, R, I),
\end{align*}
where $V_{\mathbb{R}}^s$ is a five-dimensional vector space over $\mathbb{R}$, where $R$ and $I$ are semi-invariant polynomials associated with $f = (a, b, c, d, e)$. This new space, referred to as the \textit{semi-invariant space}, provides a more structured perspective on binary quartic forms. Each element in this space is called a \textit{semi-form}, denoted by $f_s = (a, b, c, R, I)$. When all components $(a, b, c, R, I)$ are integers, we obtain an \textit{integral semi-form}, and the space of such forms is denoted by $V_{\mathbb{Z}}^s$.

\begin{lemma}
The function $\upsilon$, when restricted to $V_{\mathbb{R}, a \neq 0}$, establishes a bijection with $V_{\mathbb{R}, a \neq 0}^s$. Its inverse is given by:
 \begin{align*}
         \upsilon^{-1}\colon & V_{\R,a\neq 0}^{s} \longrightarrow V_{\mathbb{R},a\neq0} \\
          &(a,b,c,R,I) \mapsto \Big(a,b,c,d(a,b,c,R),\frac{I+3bd(a,b,c,R)-c^2}{12a}\Big)  
   \end{align*}
where 
\begin{equation}\label{d}
  d(a,b,c,R) \vcentcolon = \frac{R-b^3+4abc}{8a^2}.
\end{equation}
Define $H(a,b,c)=8ac-3b^2$ and  
define equivalent height on the semi-invariant space as follows:
\begin{equation}\label{height-semi}
h_C(a,b,c,R,I)\vcentcolon=h(a,b,c,R,I) \vcentcolon=
\begin{cases}
\max\left(|I|,  \left| \frac{3HI}{4aC} - \Gamma \right| \right) & \text{if } H \neq 0 \\
\max\left(|I|, C^{-1}\left| \frac{27R^2}{64a^3} \right| \right) & \text{if } H = 0
\end{cases}
\end{equation}
where $\Gamma$ is defined by
\begin{equation*}
    \frac{4aC}{3H}\Gamma(a,b,c,R)  =\JR.
\end{equation*}
 The defined height has the property that $h(\upsilon(a,b,c,d,e))=h(a,b,c,d,e)$.
\end{lemma}
\begin{proof}
   The first part is obvious, and the second part follows directly from equation \eqref{syzygy}.
 \end{proof}

With a slight abuse of notation, we use $h$ to denote the height both on $V_{\mathbb{R}}$ and $V_{\R}^{s}$. The lemma above shows that studying binary quartic forms with height $h$ is equivalent to studying semi-forms in the semi-invariant space with the same height.

In the semi-invariant space, the structure of the height function becomes more transparent. Fixing $(a, b, c, R)$ with $b^2 \neq 8ac$ allows us to determine that $I$ lies within a specific interval. To better understand $\Bo$, we analyze it within this framework, which requires studying the image of $V_{\mathbb{Z}, a \neq 0}$ under its mapping to the semi-invariant space.

Define for $a\ne0$ the function $d(a,b,c,R)=(R-b^3+4abc)/(8a^2)$ as before. For any $(a,b,c)\in \mathbb{Z}^3$ with $a\neq 0$ define $\Lambda_{(a,b,c)}$ a set in $\mathbb{Z}^2$:
\begin{equation*}
   \Lambda_{(a,b,c)} \vcentcolon = \left\{ (R,I)\in \mathbb{Z}^2 :
   \begin{array}{ll}
        R &\equiv b^3 - 4abc \pmod{8a^2} \\
        I &\equiv -3b d(a,b,c,R) + c^2 \pmod{12a}
   \end{array} 
   \right\}.
\end{equation*}
Define $\Lambda$:

\begin{equation*}
    \Lambda \vcentcolon = \left\{(a,b,c,R,I) \in V_{\Z, a\neq 0}^{s} :  (R,I) \in \Lambda_{(a,b,c)}\right\}.
\end{equation*}

 \begin{lemma}    
The map $\upsilon$ establishes a bijection between $V_{\mathbb{Z}, a \neq 0}$ and $\Lambda$.
 \end{lemma}
 \begin{proof}
Assume $f = (a, b, c, d, e) \in V_{\mathbb{Z}, a \neq 0}$. Referring to the definitions of the semi-invariants $R$ and $I$, we observe the following congruences:
\begin{equation*}
    R(f) \equiv b^3 - 4abc \pmod{8a^2} \quad \text{and} \quad I(f) \equiv -3b d(a, b, c, R) + c^2 \pmod{12a}.
\end{equation*}
Conversely, the map $\upsilon^{-1}$ provides a way to reconstruct elements in $V_{\mathbb{Z}, a \neq 0}$ by mapping elements of $\Lambda$ back to $V_{\mathbb{Z}, a \neq 0}$. This completes the proof.
\end{proof}

Based on the aforementioned results, counting the number of $V_{\mathbb{Z},a\neq 0}$ points with a height less than $X$ in a region of $V_{\mathbb{R}}$ is equivalent to count the number of $\Lambda$ points with a height less than $X$ in its image under the map $\upsilon$. Additionally, we know height restriction and $\Lambda$ in the semi-invariant space well, which enables us to use Poisson summation for the counting. 
These observations imply

\begin{equation}\label{Bo}
\Bo = \sum_{(a,b,c,R,I) \in \Lambda} \psi^{(i)}_{n}\left(\frac{at^4}{\lambda},\frac{bt^2}{\lambda},\frac{c}{\lambda},\frac{Rt^6}{\lambda^3},\frac{I}{\lambda^2}\right) \h_X(a,b,c,R,I).
\end{equation}
Here, $\lambda$ is defined as $X^{1/2}$, and $\h_X$ denotes the characteristic function of the height being less than $X$ in the semi-invariant space. Additionally, $\psi^{(i)}_{n}(f_s)$ is given by $n \cdot \Psi^{(i)}(\upsilon^{-1}(f_s))$. If $\upsilon(a,b,c,d,e) = (a,b,c,R,I)$ and $a\neq0$, then

\begin{equation*}
\upsilon^{-1}\left(\frac{at^4}{\lambda}, \frac{bt^2}{\lambda}, \frac{c}{\lambda}, \frac{Rt^6}{\lambda^3}, \frac{I}{\lambda^2}\right)
\end{equation*}
is equal to act by  $\lambda^{-1} \begin{pmatrix} t & \\ & t^{-1} \end{pmatrix}$ on the form $(a,b,c,d,e)$. This explains why the equation \eqref{Bo} holds.

Finally, assume that $M$ is sufficiently large to encompass the bounded support of both $n \cdot \Psi^{(i)}$ and $\psi^{(i)}_{n}$ on $V_{\R, a \neq 0}$ and $V_{\R, a \neq 0}^s$ respectively. From this point onward, we fix $i$ and omit it for brevity.

%%%%%%%%%%%%%%%%%%%%%%%%%%%%%%%%%%%%%%%%%%%%%%%%%%%%%%%%%%%%%%%%%%%%%%%%%%%%%%%%%%%%%%%%%%%%%%%%%%%%%%%%%%%%%%

\subsection{Cutting off the cusp and restriction of the integral}
For any positive $\delta$, define the set $B_{X,\delta}$ as:
\begin{equation*}
B_{X, \delta} \vcentcolon= \left\{ (a, b, c) \in \mathbb{Z}^3 : X^{1/2 - \delta_a} \leq |a|, |c| \leq MX^{1/2}, \, MX^{1/2 - \delta_b} \leq |b| \leq MX^{1/2}, \,H(a, b, c) \neq 0 \right\}, \end{equation*} where $\delta_a = 6\delta$ and $\delta_b = 2\delta$.

In this section, we refine our approach by cutting off the cusp and restricting the summation. These modifications are essential for controlling the error terms in the Poisson summation. Specifically, we establish:

\begin{proposition}\label{boundedScounting}
Let $\delta$ be a positive number less than $1/8$. Let $S$ be any $\PGL_2(\Z)$-invariant subset of $V_{\Z}$ and $dw = \dw$, then $N(\omega^{(i)},S_{\gen},X)$ is:

\begin{equation*}
     N(\omega^{(i)},S_{\gen},X) = \frac{1}{n_i \Vol(\theta)} \int_{\frac{\sqrt[4]{3}}{\sqrt{2}}}^{X^{\delta}} \int_{N'(t)}  \Bods\, dw + 
         O_{\epsilon}(X^{2-2\delta+\epsilon}). 
\end{equation*}
Here $\Bods$ is defined by

\begin{equation}
\Bods \vcentcolon= \sum_{\substack{f \in S_{\delta}}} (nt\cdot \Psi^{(i)})(X^{-1/2} \cdot f) \cdot \h_{X}(f),
\end{equation}
where $S_{\delta}$ is the set of integral forms $f = (a,b,c,d,e)$ in $S$ such that $(a,b,c) \in B_{X,\delta}$.
\end{proposition}
If $S = V_{\Z}$, then we denote $\Bods$ by $\Bodid$.

The function $n \cdot \Psi^{(i)}$ has a compact support bounded by $M$, independent of $n$. Consequently, if $nt \cdot \Psi(\lambda^{-1} f)$ is nonzero, then $\big|t^4 a / \lambda \big|$ must also be bounded by $M$. Since $|a|$ is greater than one by the generic assumption, this bound implies that $t^4 < M\lambda$.

For any subset $B$ of $V_{\mathbb{R}}$, define for any nonzero $t$
\begin{equation*}
(\lambda, t) \cdot B \vcentcolon= \left\{ \left( \frac{\lambda}{t^4} a, \frac{\lambda}{t^2}  b, \lambda c, \lambda t^2 d, \lambda t^4 e \right) \in V_{\mathbb{R}} : (a, b, c, d, e) \in B \right\}.
\end{equation*}
The next lemma restricts our attention to integral forms with nonzero $H$ semi-invariants, allowing us to apply Lemma \ref{eBound} to estimate the number of integral forms within a bounded region where the height is less than $X$.

\begin{lemma} \label{H=0}
For any $S$ subset of $V_{\Z}$ invariant under the action of $\PGL_2(\Z)$, then 
\begin{equation*}
 \frac{1}{n_i \Vol(\theta)} \int_{\frac{\sqrt[4]{3}}{\sqrt{2}}}^{M\lambda^{1/4}}\int_{N'(t)} 
      \Big( \sum_{\substack{f \in S_{\gen} \\ H(f) = 0}} (nt \cdot \Psi^{(i)})(X^{-1/2} \cdot f) \cdot \h_{X}(f) \Big)
    \, \dw = O_{\epsilon}(X^{7/4+\epsilon})
\end{equation*}
\end{lemma}

\begin{proof}
Since $nt \cdot \Psi$ has compact support bounded by a constant $M$ independent of $n$, then
\begin{equation*}
  \sum_{\substack{f \in V_{\mathbb{Z}, \gen} \\ H(f) = 0}} (nt \cdot \Psi^{(i)})(X^{-1/2} \cdot f)  =O\Big(\#\big\{\text{integral forms in } (\lambda, t) \cdot B \text{ with non-zero } a\ \text{and } H=0\big\}\Big).
\end{equation*}

If $\f$ is a binary quartic form in $(\lambda, t) \cdot B$, then the coefficients $|a|, |b|, |c|, |d|$, and $|e|$ are bounded by $M\lambda/t^4$, $M\lambda/t^2$, $M\lambda$, $M\lambda t^2$, and $M\lambda t^4$, respectively. Given the constraint $8ac = 3b^2$, once $b$ is fixed, the number of possible combinations of $a$ and $c$ that satisfy this equation is of the order $O_{\epsilon}(X^{\epsilon})$ for any $\epsilon$. Using the bounds on $d$ and $e$, we can further restrict the possible values of the form. Thus, we have
\begin{align*}
 &\int_{\frac{\sqrt[4]{3}}{\sqrt{2}}}^{M\lambda^{1/4}} \int_{N'(t)}  
    \# \{x \in (\lambda, t) \cdot B \cap V_{\mathbb{Z}, a \neq 0} : H(x) = 0 \} \, \dw \\
 \ll_{\epsilon}\quad  &\int_{\frac{\sqrt[4]{3}}{\sqrt{2}}}^{M\lambda^{1/4}} \int_{N'(t)} 
    X^{\epsilon} \cdot \lambda t^4 \cdot \lambda t^2 \cdot \frac{\lambda}{t^2} \, \dw \\
 \ll_{\epsilon} \quad &X^{7/4 + \epsilon}.
\end{align*}
This completes the proof
\end{proof}

The lemma above implies that we can restrict the summation in $\Bo$ to integral forms where their $H$ semi-invariant is non-zero.
 
Let $S$ be any subset of $V_{\mathbb{Z}}$. Define $S_r$ as the subset of elements in $S$ that are generic and satisfy $H \neq 0$. In the following propositions, we aim to show that it is sufficient to focus on small values of $t$.

\begin{proposition}\label{smallt}
For any subset $S$ of $V_{\mathbb{Z}}$ and any positive $\delta_t \leq 1/8$, the following is true.

\begin{equation*} 
\frac{1}{n_i \Vol(\theta)} \int_{X^{\delta_t}}^{M\lambda^{1/4}}\int_{N'(t)} 
     \Big( \sum_{\substack{f \in S_r}} (nt \cdot \Psi^{(i)})(X^{-1/2} \cdot f) \cdot \h_{X}(f) \Big)\dw  = O_{\epsilon}(X^{2-2\delta_t+\epsilon})
\end{equation*}
\end{proposition}
\begin{proof}
As indicated in Lemma \ref{H=0}, we need to bound the number of integral forms in $(\lambda, t) \cdot B$ that have height less than $X$ and non-zero $H$ semi-invariant. When $(a, b, c, d, e) \in (\lambda, t) \cdot B$, the coefficients $|a|$, $|b|$, $|c|$, $|d|$, and $|e|$ are bounded by $M\lambda / t^4$, $M\lambda / t^2$, $M\lambda$, $M\lambda t^2$, and $M\lambda t^4$, respectively.

Fixing $(a, b, c, d)$ under the condition $8ac \neq 3b^2$, the number of possible values for $e$ with height less than $X$ is bounded by $O(X / |H|)$, as described in Lemma \ref{eBound}. Next, we fix $d$ within the range $M\lambda t^2$ and compute the sum over $(a, b, c)$. To proceed, note that $H$ and $b$ are bounded by $O(X)$ and $O(\lambda / t^2)$, respectively. Once $H$ and $b$ are fixed, the values of $a$ and $c$ can be determined as factors of $H + 3b^2$. The number of ways to choose $a$ and $c$ in this manner is $O_{\epsilon}(X^{\epsilon})$ for any $\epsilon$. This implies that:
\begin{align*}
 \# \Big\{f \in (\lambda,t) \cdot B \cap V_{\Z, a\neq 0}: H(f) \neq 0, \, h(f)<X\Big\}  
 &\ll \sum_{0<|H|\ll X} \sum_{|b| \ll \lambda /t^2} \sum_{a\mid H+3b^2} \sum_{|d|\ll \lambda t^2}\frac{X}{|H|} \\
&\ll_{\epsilon} X^{\epsilon} \cdot \lambda t^2\cdot \frac{\lambda}{t^2} \sum_{0<H\ll X} \frac{X}{|H|} \\
&\ll_{\epsilon} 
X^{2+\epsilon}.
\end{align*}
As a result, we obtain:
\begin{align*} &\int_{X^{\delta_t}}^{M\lambda^{1/4}}\int_{N'(t)}   \Big( \sum_{\substack{f \in V_{\Z, \gen} \\ H(f) \neq 0}} (nt \cdot \Psi^{(i)})(X^{-1/2} \cdot f) \cdot \h_{X}(f) \Big) \, \dw \\ \ll_{\epsilon}  \quad
&\int_{X^{\delta_t}}^{M\lambda^{1/4}}\int_{N'(t)}   X^{2+\epsilon} \, \dw \\ \ll_{\epsilon} \quad
&X^{2+\epsilon}  \cdot X^{-2\delta_t}. \end{align*}
Since $S_{\gen}$ is a subset of $V_{\mathbb{Z}, \gen}$, the proposition follows from the above analysis.
\end{proof}

According to the propositions mentioned in this section, to estimate the number of $\PGL_2(\mathbb{Z})$ orbits using the averaging method, we only need to understand the following integral

\begin{equation*}
N(\omega^{(i)},V_{\Z, \gen}^{(i)};X) = \frac{1}{n_i \Vol(\theta)} \int_{\frac{\sqrt[4]{3}}{\sqrt{2}}}^{X^{\delta_t}}\int_{N'(t)}  
   \Bo \dw + O_{\epsilon}(X^{2-2\delta_t+\epsilon})
\end{equation*}.

The goal of the following lemma is to show that in the sum defining $\Bo$, the contribution from integral forms with $a$, $b$, and $c$ smaller than $X^{1/2-\delta}$ is negligible. Consequently, the dominant contribution to $\Bo$ comes from terms where $(a,b,c)\in B_{X,\delta}$.

\begin{lemma}\label{abig}
Let $\delta_a$ and $\delta_b$ be positive numbers satisfying $\delta_b \leq 1/8$ and $0<\epsilon_a = \delta_a - 2\delta_b < 1/4$. Let $S$ be any subset of $V_{\mathbb{Z}}$ that is invariant under $\PGL_2(\mathbb{Z})$, and let $dw = \dw$. Then:

\begin{itemize}
    \item[\textnormal{(a)}] If we restrict the summation to  integral forms with $|b|<X^{1/2-\delta_b}$, then:
        \begin{equation*}
      \frac{1}{n_i \Vol(\theta)} \int_{\frac{\sqrt[4]{3}}{\sqrt{2}}}^{X^{\delta_t}} \int_{N'(t)}  
       \Big( \sum_{\substack{f \in S_{r} \\ |b|<X^{1/2-\delta_b}}} (nt \cdot \Psi^{(i)})(X^{-1/2} \cdot f) \cdot \h_{X}(f) \Big)
      dw= O_{\epsilon}(X^{2+\epsilon-\delta_b})
     \end{equation*}   

     \item[\textnormal{(b)}]  If we consider points in the summation that  satisfying $|b| \geqslant X^{1/2-\delta_b}$ and either $|a|$ or $|c|$ less than $X^{1/2-\delta_a}$, then:
    \begin{equation*}
   \frac{1}{n_i \Vol(\theta)} \int_{\frac{\sqrt[4]{3}}{\sqrt{2}}}^{X^{\delta_t}} \int_{N'(t)}  
       \Big( \sum_{\substack{f \in S_r \\ |b| \geq X^{1/2-\delta_b}\\ |a| <X^{1/2-\delta_a}}} (nt \cdot \Psi^{(i)})(X^{-1/2} \cdot f) \cdot \h_{X}(f) \Big)
      dw= O(X^{2 -\epsilon_a})
 \end{equation*} 
    
    \begin{equation*}
   \frac{1}{n_i \Vol(\theta)} \int_{\frac{\sqrt[4]{3}}{\sqrt{2}}}^{X^{\delta_t}} \int_{N'(t)}  
       \Big( \sum_{\substack{f \in S_r \\ |b| \geq X^{1/2-\delta_b}\\ |c| <X^{1/2-\delta_a}}} (nt \cdot \Psi^{(i)})(X^{-1/2} \cdot f) \cdot \h_{X}(f) \Big)
      dw= O(X^{2 -\epsilon_a})
    \end{equation*} 
\end{itemize}
\end{lemma}
\begin{proof}
\begin{itemize}
  \item[\textnormal{(a)}] As in the proof of the proposition, our goal is to bound the summation over integral forms in $\Bo$ with $H \neq 0$ by applying Lemma \ref{eBound}. Using a similar counting technique as in previous arguments, it suffices to analyze the sum of $X/|H|$ over $(a, b, c)$ satisfying $|a| \ll X^{1/2}$, $|b| \ll X^{1/2-\delta_b}$, and $|c| \ll X^{1/2}$. Moreover, $|H|$ is bounded by $O(X)$.  

To refine the counting, we shift our focus to $(H, b)$, leading to the following estimate:

\begin{align*}
& \int_{\frac{\sqrt[4]{3}}{\sqrt{2}}}^{X^{\delta_t}} \int_{N'(t)}  
      \Big( \sum_{\substack{f \in V_{\mathbb{Z}, a\neq 0} \\ |b| \leq X^{1/2-\delta_b}\\ H(f) \neq 0}} (nt \cdot \Psi^{(i)})(X^{-1/2} \cdot f) \cdot \h_{X}(f) \Big)\dw \\ \ll_{\epsilon} \quad
    &\int_{\frac{\sqrt[4]{3}}{\sqrt{2}}}^{X^{\delta_t}} \int_{N'(t)}   X^{1+\epsilon}X^{1/2-\delta_b} \lambda t^2 \dw \\ \ll_{\epsilon} \quad
    &X^{2+\epsilon -\delta_b}
\end{align*}
\item[\textnormal{(b)}] In both cases, where either $|a|$ or $|c|$ is less than $X^{1/2-\delta_a}$, we can show that $|H| \gg X^{1-2\delta_b}$. Using the same counting technique as above, we bound $X/|H|$ by $X^{2\delta_b}$. Moreover, the number of possible choices for $(a, b, c)$ is bounded by $X^{1/2-\delta_a} \lambda^2 / t^2$.

\begin{align*}
&\int_{\frac{\sqrt[4]{3}}{\sqrt{2}}}^{X^{\delta_t}} \int_{N'(t)}  
       \Big( \sum_{\substack{f \in V_{\mathbb{Z}, a\neq 0} \\ |b| \geq X^{1/2-\delta_b}\\ |a| <X^{1/2-\delta_a}}} (nt \cdot \Psi^{(i)})(X^{-1/2} \cdot f) \cdot \h_{X}(f) \Big)
      dw \\ \ll \quad
    &\int_{\frac{\sqrt[4]{3}}{\sqrt{2}}}^{X^{\delta_t}} \int_{N'(t)}  X^{1/2-\delta_a} X^{2\delta_b}\frac{\lambda^{2}}{t^2}\lambda t^2 \,  \dw   
    \\ \ll \quad
    & X^{2+2\delta_b-\delta_a}
\end{align*}
\end{itemize}
This completes the proof.
\end{proof}

Proposition \ref{boundedScounting} follows directly from the preceding lemmas.

The following corollary provides the continuous version of the above lemmas. Let $V_{\delta}$ denote the set of elements $f = (a, b, c, d, e)$ in $V_{\mathbb{R}}$ such that $(\lfloor a \rfloor, \lfloor b \rfloor, \lfloor c \rfloor) \in B_{X,\delta}$.

\begin{cor}\label{VtoV_R}
    For any $\delta$ less than $1/8$, and let $\delta_t=\delta$. Then, we have:
    \begin{align*}
       &\int_{\frac{\sqrt[4]{3}}{\sqrt{2}}}^{X^{\delta_t}} \int_{N'(t)}   \Big( \int_{f \in V_{\delta}} (nt \cdot \Psi^{(i)})(X^{-1/2} \cdot f) \cdot \h_{X}(f) \, df\Big) \dw =\\
       &\int_{\frac{\sqrt[4]{3}}{\sqrt{2}}} \int_{N'(t)}   \Big(\int_{f \in V_{\R}} (nt \cdot \Psi^{(i)})(X^{-1/2} \cdot f) \cdot \h_{X}(f) \, df \Big)\dw
        + O_{\epsilon}(X^{2-2\delta+\epsilon}).
    \end{align*}
\end{cor}
\begin{proof}
The proof follows a similar approach to the previous lemmas.
\end{proof}

%%**********************************************************************************************************************************************************************************************
\subsection{Estimates on reducibility}\label{reducibility subsection}
The following lemma shows that when estimating $N(\omega^{(i)}, S_{\gen}; X)$ for a $\PGL_2(\mathbb{Z})$-invariant set $S \subseteq V_{\mathbb{Z}}$, it suffices to restrict the sum to elements in $V_{\mathbb{Z}, a \neq 0} \cap S$. This reduction simplifies the integral while introducing an error term of order $O_{\epsilon}(X^{3/2+\epsilon})$. Specifically, we have:
\begin{lemma}\label{reducibility lemma} For understanding $N(\omega^{(i)}, S_{\gen}; X)$ for any set $S \subseteq V_{\mathbb{Z}}$ that is closed under $\PGL_2(\mathbb{Z})$, it is sufficient to consider the sum inside the integral over $V_{\mathbb{Z}, a \neq 0} \cap S$. Specifically, we have: \begin{align*} &\frac{1}{n_i \Vol(\theta)} \int_{t = \frac{\sqrt[4]{3}}{\sqrt{2}}} \int_{N'(t)}  \left( \sum_{f \in S_{\gen}} (nt \cdot \Psi[\omega^{(i)}])(X^{-1/2} \cdot f) \right) \cdot h_{X}(f) \, \dw =\\ & \frac{1}{n_i \Vol(\theta)} \int_{t = \frac{\sqrt[4]{3}}{\sqrt{2}}} \int_{N'(t)}   \left( \sum_{f \in V_{\mathbb{Z}, a \neq 0} \cap S} (nt \cdot \Psi[\omega^{(i)}])(X^{-1/2} \cdot f) \right) \cdot h_{X}(f) , \dw + O_{\epsilon}(X^{7/4+\epsilon}). \end{align*} \end{lemma}

We proceed in two steps. First, we show that the sum over $V_{\mathbb{Z}, a \neq 0}$, restricted to reducible forms over $\mathbb{Q}$, is bounded by $O_{\epsilon}(X^{7/4+\epsilon})$. Next, we demonstrate that the sum over $V_{\mathbb{Z}, \irr}$, which corresponds to non-generic elements (i.e., forms with nontrivial stabilizers under $\PGL_2(\mathbb{Q})$), satisfies the same bound $O_{\epsilon}(X^{7/4+\epsilon})$.

The following lemma implies Lemma \ref{reducibility lemma} as $nt\cdot \Psi[\omega^{(i)}]$ has bounded compact support.

\begin{lemma}\label{lemma for reducibility lemma} For any $B$ bounded subset of $V_{\mathbb{R}}$, we have:
\begin{equation}\label{reducible}
\int_{t = \frac{\sqrt[4]{3}}{\sqrt{2}}} \int_{N'(t)}   \#\left\{ f \in (\lambda, t) \cdot B \cap V_{\mathbb{Z}, a \neq 0} : h(f) < X, f \text{ is reducible over } \mathbb{Q} \right\} \dw = O_{\epsilon}(X^{7/4+\epsilon}), \end{equation} and
\begin{equation} \label{generic}
\int_{t = \frac{\sqrt[4]{3}}{\sqrt{2}}} \int_{N'(t)} \#\left\{ f \in (\lambda, t) \cdot B \cap V_{\mathbb{Z}, \irr} : h(f) < X, f \text{ is not generic} \right\} \dw = O_{\epsilon}(X^{3/2+\epsilon}). \end{equation} \end{lemma}

The \textit{cubic resolvent} of a quartic form $f(x)$ is the cubic polynomial associated with the invariants $I$ and $J$, and it is defined as:
\begin{equation*}
    g(x) = x^3 - 3Ix + J.
\end{equation*}

Since non-generic irreducible quartic forms, as well as quartics that factor as the product of two nonlinear polynomials over $\mathbb{Q}$, have reducible cubic resolvents, the following lemma provides a key step in proving Lemma \ref{lemma for reducibility lemma}.

\begin{lemma}\label{h_IJ}
Let $I$ and $J$ be fixed and bounded by $X$. Denote by $h_{I,J}$ the number of equivalence classes of integral quartic forms with these invariants. If the cubic resolvent of these quartic forms, given by $g(x) = x^3 - 3Ix + J$, is reducible over $\mathbb{Q}$, then $h_{I,J}$ is bounded by $O_{\epsilon}(X^{1/2+\epsilon})$.
\end{lemma}
\begin{proof}
We aim to understand the following set:
\begin{equation*}
    \left\{f \in V_{\Z} : I(f) = I, J(f) = J \right\} / \PGL_2(\Z).
\end{equation*}
By the reduction theory of quartic forms, it suffices to study forms with $a = O(X^{1/2})$. Fixing $a$, $H$, $I$, and $J$ uniquely determines an equivalence class of quartic forms. This follows from the fact that, given these values, the roots of the quartic form can be determined up to a constant by solving the quartic equation using radicals. Moreover, the semi-invariants of quartic forms satisfy the syzygy equation, which further constrains the structure of these forms. Consequently, the number of such equivalence classes can be bounded by counting all possible pairs $(a, H)$.  

Fix $a$. We now seek the number of integral solutions to the following elliptic curve equation:
\begin{equation*}
   H^3 - 48a^2IH + 64a^3J = -27R^2.
\end{equation*}
Rewriting this equation in a more convenient form, we obtain:
\begin{equation*}
    H^3 - 3(12a)^2IH - (12a)^3J = R^2.
\end{equation*}
Applying the results from Theorems 5.1 and 5.2 in \cite{elliptic-integral}, we conclude that the number of integral solutions to this equation is bounded by $O_{\epsilon}(X^{\epsilon} h_2(K)^{0.28})$, where $K$ is the étale extension $\mathbb{Q}[x]/(x^3 - 3Ix + J)$. Since $x^3 - 3Ix + J$ is assumed to be reducible, the 2-torsion subgroup is small, leading to a bound of $O_{\epsilon}(X^{\epsilon})$. This completes the proof.
\end{proof}

\begin{customproof}[Lemma \ref{lemma for reducibility lemma}] To prove equations \eqref{reducible} and \eqref{generic}, we divide the analysis into two cases:
\begin{itemize}
    \item Counting quartic forms with at least one linear factor.
    \item Counting quartic forms with a reducible cubic resolvent.
\end{itemize}
For the first case, we apply Lemma \ref{nonzero-a-bounding}, which provides the necessary bounds. In the second case, irreducible non-generic forms and quartic forms that factor as the product of two irreducible quadratics both have a reducible cubic resolvent. Applying Lemma \ref{reducible cubic resolvent count}, we find that the number of reducible cubic polynomials with height less than $X$ is bounded by $O_{\epsilon}(X^{1+\epsilon})$. Finally, Lemma \ref{h_IJ} establishes the bound $O_{\epsilon}(X^{3/2+\epsilon})$ for the number of quartic classes with a reducible cubic resolvent.
\end{customproof}

%%**********************************************************************************************************************************************************************************************
\subsection{Counting lattice points with height restriction in semi-invariant space}
Define $\Lambda_{X, \delta}$ to be the set of all semi-forms $(a, b, c, R, I)$ in $\Lambda$ such that $(a, b, c) \in B_{X, \delta}$. According to Corollary \ref{boundedScounting}, the task of estimating integral orbits of quartic forms with height less than $X$ can be simplified to calculating the following summation:

\begin{equation}\label{Semi-equivalent-question} \sum_{(a, b, c, R, I) \in \Lambda_{X, \delta}} \psi_{n} \left( \frac{at^4}{\lambda}, \frac{bt^2}{\lambda}, \frac{c}{\lambda}, \frac{Rt^6}{\lambda^3}, \frac{I}{\lambda^2} \right) \h_X(a, b, c, R, I), \end{equation} where $\psi_{n}$ and $\h_X(a, b, c, R, I)$ are defined as before.

To estimate the summation in equation \eqref{Semi-equivalent-question}, we analyze the fibers of $(a, b, c)$ and apply Poisson summation to each fiber. First, we fix $(a, b, c) \in B_{X, \delta}$ under the previously outlined assumptions. With this choice of $(a, b, c)$, we then estimate the following summation over $(R, I)$:

\begin{equation*} \sum_{(R, I) \in \Lambda_{(a, b, c)}} \psi_{n} \left( \frac{at^4}{\lambda}, \frac{bt^2}{\lambda}, \frac{c}{\lambda}, \frac{Rt^6}{\lambda^3}, \frac{I}{\lambda^2} \right) \h_X(a, b, c, R, I). \end{equation*}
For any $(a, b, c) \in \mathbb{R}^3$ with $a \neq 0$, we define $R_0(a, b, c)$ as follows:

\begin{itemize}
    \item  When $H > 0$: If  
    \begin{equation}\label{R0Hp}
        X - \frac{H^2}{48a^2} > \frac{4C}{3} \left| \frac{aX}{H} \right|,
    \end{equation}
    then  
    \begin{equation*}
        R_0(a, b, c) = \sqrt{\frac{48a^2 H}{27} \left( X - \frac{H^2}{48a^2} \right)}.
    \end{equation*}
    Otherwise, we set  
    \begin{equation*}
        R_0(a, b, c) = 1.
    \end{equation*}

    \item  When $H < 0$:
    \begin{equation}\label{RoHn}
        R_0(a, b, c) = \sqrt{\left| \frac{48a^2 H}{27} \left( X + \frac{H^2}{48a^2} \right) \right|}.
    \end{equation}
\end{itemize}
For any $(a, b, c)$ with $H \neq 0$ and $a \neq 0$, let $\mathcal{H}_{(a, b, c)}$ be the set of elements in the fiber of $\Lambda_{(a, b, c)}$ whose height is less than $X$. This set is given by:
\begin{equation}\label{H_def} 
    \mathcal{H}_{(a,b,c)} \vcentcolon = \left\{(R,I) :  
    \begin{array}{ll} 
        & |I| < X, \\  
        & |I - (\JR)| < \frac{4C}{3} \left| \frac{a}{H} \right| X 
    \end{array}  
    \right\}.  
\end{equation}
The function $\h_{X}(a, b, c, R, I)$ determines whether $(R, I)$ belongs to this set. Specifically, we have:
\begin{equation*}
    \h_{X}(a, b, c, R, I) = \chi_{{(a, b, c)}}(R, I),
\end{equation*}
where $\chi_{{(a, b, c)}}$ is the characteristic function of $\mathcal{H}_{(a, b, c)}$.

We introduce the set $\mathcal{H}'_{(a, b, c)}$, which provides a simplified parametrization compared to $\mathcal{H}_{(a, b, c)}$ and is specifically designed to facilitate Fourier calculations. We claim that $\mathcal{H}_{(a, b, c)}$ and $\mathcal{H}'_{(a, b, c)}$ contribute equivalently to the overall count. The symmetric difference between these sets contains only a negligible number of points in $\Lambda_{(a, b, c)}$ and has small volume.  

We define $\mathcal{H}'_{(a, b, c)}$ as follows:
\begin{equation}\label{HH_def} \H'_{(a,b,c)} \vcentcolon= \left\{(R,I) : \begin{array}{ll} & |R| < R_0(a,b,c) \\ & |I - (\JR)| < \frac{4C}{3}|\frac{a}{H}|X \end{array} \right\}. \end{equation}
The goal of this section is twofold. First, we show that $\mathcal{H}_{(a, b, c)}$ can be replaced by $\mathcal{H}'_{(a, b, c)}$ in the height function, as their contributions are nearly identical. Second, we demonstrate that the $H$-invariant, defined as $H(a, b, c) = 3b^2 - 8ac$, can be assumed to be close to $X$, as previously discussed. Specifically, we establish the following proposition:

\begin{proposition}\label{height_and_H_change_proposition}
Let $\delta$ be a positive constant less than $1/28$, and define $\delta_H = 28\delta$. Now, let $B'_{X,\delta}$ denote the subset of $B_{X, \delta}$ consisting of all $(a, b, c)$ satisfying $|H(a, b, c)| \geq X^{1 - \delta_H}$. Then, we have:
\begin{equation*}
\Bodi = \sum_{(a, b, c) \in B'_{X, \delta}} \sum_{(R, I) \in \Lambda_{(a, b, c)}} \psi_{n} \left( \frac{a t^4}{\lambda}, \frac{b t^2}{\lambda}, \frac{c}{\lambda}, \frac{R t^6}{\lambda^3}, \frac{I}{\lambda^2} \right) \chi'_{(a, b, c)}(R, I) + O_{\epsilon}(X^{2 - 2\delta + \epsilon}),
\end{equation*}
where $\chi'_{(a, b, c)}$ denotes the characteristic function of $\H'_{(a, b, c)}$, defined as
\begin{equation*}
   \chi'_{(a, b, c)}(R, I) = \chi_{(-1,1)}\left(\frac{R}{R_0}\right)\,\chi_{(-1,1)}\left(\frac{3HI}{4aCX} - X^{-1} \Gamma(a, b, c, R) \right),
\end{equation*}
where $\chi_{(-1,1)}$ is the characteristic function of the interval $(-1,1)$, and $\Gamma$ is defined as before.
\end{proposition}
To prove this proposition, we divide the argument into a series of lemmas within this subsection. Each lemma addresses a specific aspect of the proof, allowing us to construct the full argument step by step. Consequently, a direct proof of the proposition is not provided, as the lemmas collectively establish the result.

%\begin{lemma}\label{R_0 is big} Let $\delta$ be positive and less than $1/28$. Let $\delta_R$ and $\delta_H$ be $14\delta$ and $28\delta$, respectively. Define $B'{X,\delta}$ to be the set of all $(a,b,c) \in B_{X,\delta}$ such that $|H(a,b,c)| \geqslant X^{1-\delta_H}$ and $|R_0(a,b,c)| \geq X^{3/2-\delta_R}$. Then:
%\begin{equation*} \begin{split} \sum_{(a,b,c)\in B_{X,\delta} \backslash B'{X,\delta}} \sum_{(R,I) \in \Lambda_{(a,b,c)}} \psi_ {n,k}\left(\frac{at^4}{\lambda},\frac{bt^2}{\lambda},\frac{c}{\lambda},\frac{Rt^6}{\lambda^3},\frac{I}{\lambda^2}\right) \chi'{{(a,b,c)}}(R,I) &= O_{\epsilon}(X^{2+\epsilon+2\delta_a-\delta_R}+X^{2+\epsilon+2\delta_a-\delta_H/2}) \ &= O_{\epsilon}(X^{2-2\delta+\epsilon}), \end{split}
%\end{equation*} where $\chi'{(a,b,c)}$ is the characteristic function of $H'{(a,b,c)}$. \end{lemma}%

Let define for any $W$ subset of semi-invariant space, and $\lambda,t$ positive numbers:
\begin{equation}\label{expan-S}
    (\lambda,t)\cdot W \vcentcolon =\left\{(\frac{\lambda}{t^4}a,\frac{\lambda}{t^2}b, \lambda c,\frac{\lambda^3}{t^6}R,\lambda^2I): (a,b,c,R,I) \in W\right\}.
\end{equation} 
Let $W_{(a,b,c)}$ represent the fiber of $W$ over $(a, b, c)$.  For any subset $U \subset \mathbb{R}^2$ and any positive $\lambda$ and $t$, we define:
\begin{equation}\label{expan-S-fibe}
  (\lambda,t)\cdot U \vcentcolon=\left\{(\frac{\lambda^{3}}{t^6}R,\lambda^2I):(R,I)\in U\right\}. 
\end{equation}
We observe that the fiber of $(\lambda, t) \cdot W$ over a fixed $(a, b, c)$ is equivalent to $(\lambda, t) \cdot W'$, where $W'$ denotes the fiber of $W$ over $(at^4/\lambda, bt^2/\lambda, c/\lambda)$.

Let $\lambda$ and $t$ be fixed, where $\lambda = X^{1/2}$ and $t < X^{\delta_t}$. 
\begin{lemma}\label{HtoH'}
    For any \( (a,b,c) \in B_{X,\delta} \), then:
    \begin{equation*}
        \#\{f_s \in (\H_{(a,b,c)} \triangle \H'_{(a,b,c)}) \cap \Lambda_{(a,b,c)}\} = O\left(\frac{X^{3/2}}{|a|^{1/2}|H|}\right)
    \end{equation*}
    and 
\begin{equation*}
    \frac{1}{|a|^3}\Vol(\H_{(a,b,c)} \triangle \H'_{(a,b,c)}) = O\left(\frac{X^{3/2}}{|a|^{1/2}|H|}\right)
\end{equation*}

\end{lemma}
\begin{proof}
Let $H = 8ac - 3b^2$. When $X - \frac{H^2}{48a^2} > \frac{4C}{3}\left|\frac{aX}{H}\right|$, we define $R_{-1}$ and $R_{1}$ such that:
\begin{equation*}
-\frac{4C}{3}\left|\frac{aX}{H}\right| + \frac{27R_{-1}^2}{48a^2H} + \frac{H^2}{48a^2} = X,
\end{equation*}
and
\begin{equation*}
\frac{4C}{3}\left|\frac{aX}{H}\right| + \frac{27R_{1}^2}{48a^2H} + \frac{H^2}{48a^2} = X.
\end{equation*}
The symmetric difference between the sets $\mathcal{H}_{(a,b,c)}$ and $\mathcal{H}'_{(a,b,c)}$ is confined to where $|R|$ lies in the range $[R_{1}, R_{-1}]$. For each fixed $R$, the value of $I$ is restricted to an interval of size $O\left(\left| \frac{aX}{H} \right|\right)$.  

The value of $R$ is fixed modulo $8a^2$ in $\Lambda_{(a,b,c)}$, and for each fixed $R$, the value of $I$ is also determined modulo $12a$. Consequently, we have:
\begin{equation*}
\#\big\{f_s \in (\H_{(a,b,c)} \triangle \H'_{(a,b,c)}) \cap \Lambda_{(a,b,c)}\big\} \ll \left|\frac{X}{H}\right| \left(\frac{|R_{1} - R_{-1}|}{a^2} + 1\right).
\end{equation*}
Thus, we only need to bound $|R_{1} - R_{-1}|$. By our assumption and the definitions, $R_{-1} \gg a^{3/2} X^{1/2}$ and $|R_{1}^2 - R_{-1}^2| = O(|a|^3 X)$, which implies:
\begin{equation*}
|R_{1} - R_{-1}| = O(X^{1/2} |a|^{3/2}).
\end{equation*}
Hence, we obtain the desired bound. The same reasoning applies when $H$ is negative.
\end{proof}

\begin{lemma}\label{S-Counting-Reduced} 
 For any $n$ and $t$ as assumed before, we have:
\begin{equation*}
\begin{split}
\Bodi &=
\sum_{(a, b, c) \in B_{X, \delta}} \sum_{(R, I) \in \Lambda_{(a, b, c)}} \psi_{n}\left(\frac{at^4}{\lambda}, \frac{bt^2}{\lambda}, \frac{c}{\lambda}, \frac{Rt^6}{\lambda^3}, \frac{I}{\lambda^2}\right) \chi_{{(a, b, c)}}(R, I) \\
&= \sum_{(a, b, c) \in B_{X, \delta}} \sum_{(R, I) \in \Lambda_{(a, b, c)}} \psi_{n}\left(\frac{at^4}{\lambda}, \frac{bt^2}{\lambda}, \frac{c}{\lambda}, \frac{Rt^6}{\lambda^3}, \frac{I}{\lambda^2}\right) \chi'_{{(a, b, c)}}(R, I) 
+ O_{\epsilon}(X^{7/4 + \delta_a/2 + \epsilon}),
\end{split}
\end{equation*}
here $\chi_{{(a, b, c)}}$ and $\chi'_{{(a, b, c)}}$ denote the characteristic functions of $\H_{(a, b, c)}$ and $\H'_{(a, b, c)}$, respectively.
\end{lemma}
\begin{proof}
By applying Lemma \ref{HtoH'}, it suffices to show that the sum over $B_{X,\delta}$ of $\frac{X^{3/2}}{|a|^{1/2}|H|}$ is $O_{\epsilon}(X^{7/4+\delta_a/2+\epsilon})$.  

First, using the bound $|a| > X^{1/2 - \delta_a}$, we reduce the sum over $\frac{1}{|H|}$ by a factor of $X^{3/2 - 1/4 + \delta_a/2}$. Next, for fixed $H$ and $b$, the values of $a$ and $c$ are determined as factors of $H + 3b^2$, leading to the following result:
\begin{equation*}
     \sum_{(a,b,c) \in B_{X,\delta}} \frac{X^{3/2}}{|a|^{1/2}|H|}  \ll
    \sum_{(a,b,c) \in B_{X,\delta}} \frac{X^{3/2-1/4+\delta_a/2}}{|H|}  \ll_{\epsilon}
    X^{3/2-1/4+\delta_a/2}\cdot X^{1/2+\epsilon}.    
\end{equation*} 
This completes the proof.
\end{proof}

In the following lemma, our objective is to restrict the counting to $(a, b, c)$ satisfying $|H| > X^{1-\delta_H}$, where $\delta_H$ is a positive constant.

\begin{lemma}\label{H-Is-Big}
Let $\delta$ be a positive constant less than $1/28$, and define $\delta_H = 28\delta$. Now, define $B'_{X,\delta}$ to be the set of all $(a,b,c) \in B_{X,\delta}$ such that $|H(a,b,c)| \geqslant X^{1-\delta_H}$. Then:
\begin{equation*}
\begin{split}
    \sum_{(a,b,c)\in B_{X,\delta} \backslash B'_{X,\delta}} \sum_{(R,I) \in \Lambda_{(a,b,c)}} \psi _{n,k}\left(\frac{at^4}{\lambda},\frac{bt^2}{\lambda},\frac{c}{\lambda},\frac{Rt^6}{\lambda^3},\frac{I}{\lambda^2}\right) \chi'_{{(a,b,c)}}(R,I) &= O_{\epsilon}(X^{2+\epsilon+2\delta_a-\delta_H/2}) \\ &= O_{\epsilon}(X^{2-2\delta+\epsilon}),
\end{split}  
\end{equation*}
where $\chi'_{(a,b,c)}$ is the characteristic function of $\H'_{(a,b,c)}$.
\end{lemma}

\begin{proof}
We are interested in bounding the number of points in $\Lambda_{(a, b, c)}$ that fall within $\H'_{(a, b, c)}$ for each $(a, b, c) \in B_{X, \delta}$, excluding those in $B'_{X, \delta}$. If $(R, I) \in \Lambda_{(a, b, c)} \cap \H'_{(a, b, c)}$, then $R$ is fixed (congruent to $8a^2$) and lies within an interval of size $O(R_0)$.

For any $(a, b, c) \in B_{X, \delta}$, since $H = O(X)$ and $|a| \geq X^{1/2 - \delta_a}$, we have the following estimate for $R_0(a, b, c)$: \begin{equation*} |R_0(a, b, c)| \ll \sqrt{\Big|a^2 H \Big( X + \frac{H^2}{48a^2} \Big)\Big|} \ll X^{1/2 + \delta_a} |a| |H|^{1/2}. \end{equation*}

Additionally, fixing $R$ implies that $I$ takes on a fixed value modulo $12a$ and lies within an interval of size $O\left( \frac{|a|X}{|H|} \right)$. Thus, the total number of points in $\Lambda_{(a, b, c)} \cap \H'_{(a, b, c)}$ for $(a, b, c) \in B_{X, \delta} \backslash B'_{X, \delta}$ is bounded by:
\begin{align*}
    \sum_{(a, b, c) \in B_{X, \delta} \backslash B'_{X, \delta}} \#\big\{ (R, I) \in \Lambda_{(a, b, c)} \cap \H'_{(a, b, c)} \big\} 
    &\ll \sum_{(a, b, c) \in B_{X, \delta} \backslash B'_{X, \delta}} \frac{X^{1/2 + \delta_a} |a| |H|^{1/2}}{|a|^2} \cdot \frac{X}{|H|} \\
    &\ll \sum_{(a, b, c) \in B_{X, \delta} \backslash B'_{X, \delta}} \frac{X^{1+2\delta_a}} {|H|^{1/2}}.
\end{align*}
Then, as before, we concentrate on $H$ and $b$ being fixed, and we find $a$ and $c$ as factors of $H + 3b^2$. We also have the additional assumption that $|H|$ is less than $X^{1 - \delta_H}$, which implies the bound $X^{2 + 2\delta_a - \delta_H/2 + \epsilon}$.
\end{proof}

\begin{remark}
The bound 
\begin{equation*} 
|R_0(a, b, c)| \ll  X^{1/2 + \delta_a} |a| |H|^{1/2} 
\end{equation*}
is used frequently throughout the proofs and serves as an upper estimate for the parameter $R_0(a, b, c)$ when $(a,b,c) \in B_{X,\delta}$.
\end{remark}
The following corollary provides the continuous version of the lemmas in this section. Define $\SV'$ as the set of semi-forms $f_s = (a, b, c, R, I)$ such that $(\lfloor a \rfloor, \lfloor b \rfloor, \lfloor c \rfloor) \in B'_{X,\delta}$. Similarly, let $\SV$ denote the set of semi-forms $f_s = (a, b, c, R, I)$ for which $(\lfloor a \rfloor, \lfloor b \rfloor, \lfloor c \rfloor) \in B_{X,\delta}$.

\begin{cor}\label{SV'toSV}
Let $\delta$ be positive and less than $1/28$. Then:
\begin{equation*}
    \int_{f_s \in \SV'} \frac{1}{|a(f_s)^3|}\,  \psi_{n}((\lambda,t)\cdot f_s) \cdot \h_X'(f_s) \, df_s = \int_{f_s \in \SV}  \frac{1}{|a(f_s)^3|}\, \psi_{n}((\lambda,t)\cdot f_s) \cdot \h_X(f_s) \, df_s + O_{\epsilon}(X^{2-2\delta+\epsilon}),
\end{equation*}
where the operation $(\lambda,t) \cdot (a,b,c,R,I)$ is defined by
\begin{equation*}
    (\lambda,t) \cdot (a,b,c,R,I) \vcentcolon= \Big( \frac{at^4}{\lambda}, \frac{bt^2}{\lambda}, \frac{c}{\lambda}, \frac{Rt^6}{\lambda^3}, \frac{I}{\lambda^2} \Big).
\end{equation*}
The term $\h_X'$ represents the new height in the semi-invariant space, given by
\begin{equation*} 
       \h_X'(a,b,c,R,I) \vcentcolon= \chi_{(-1,1)}\Big(\frac{R}{R_0(a,b,c)}\Big) \chi_{(-1,1)}\Big(\frac{3HI}{4aCX}-X^{-1}\Gamma(a,b,c,R)\Big),
\end{equation*}
where $R_0(a,b,c)$ and $\Gamma(a,b,c,R)$ are defined similarly as before.  
\end{cor}
\begin{proof}
The proof follows similarly to the proofs of Lemma \ref{HtoH'} and Lemma \ref{H-Is-Big}.
\end{proof}

%%%%%%%%%%%%%%%%%%%%%%%%%%%%%%%%%%%%%%%%%%%%%%%%%%%%%%%%%%%%%%%%%%%%%%%%%%%%%%%%%%%%%%%%%%%%%%%%%%%%%%%%%%%%%%%%%%%%
 \section{Equidistribution of integral quartic forms with the bounded height}\label{sec3}

In this section, our objective is to estimate the summation in Proposition \ref{height_and_H_change_proposition} using Poisson summation. Our primary goal is to demonstrate the uniform distribution of the lattice points $\Lambda$ within the space of semi-forms, subject to the height restriction $X$.

We introduce a smooth function $\phi$, which we refer to as a \textit{good $\delta_s$-approximation} to the characteristic function $\chi_{(-1,1)}$. This concept will be formalized later in this section. Using this function, we define the smooth approximation $h_{(a,b,c)}$ to the characteristic function $\chi'_{(a,b,c)}$ as: \begin{equation*} 
h_{(a,b,c)}(R, I) \vcentcolon= \phi\Big(\frac{R}{R_0}\Big) \phi\Big(\frac{3HI}{4aCX} - X^{-1}\Gamma(a,b,c,R)\Big), \end{equation*} where $R_0$ and $\Gamma(a,b,c,R)$ are as previously defined. Using this, we construct the function $\theta$, which will be central to our study:
\begin{equation*} 
\theta(a, b, c, R, I) \vcentcolon= \psi_{n}\left(\frac{at^4}{\lambda}, \frac{bt^2}{\lambda}, \frac{c}{\lambda}, \frac{Rt^6}{\lambda^3}, \frac{I}{\lambda^2}\right) \cdot h_{(a,b,c)}(R,I),
\end{equation*}
which we refer to as a \textit{good $\delta_s$-approximation}. Given fixed parameters $(a, b, c) \in B'_{X, \delta}$, let $\theta_{(a, b, c)}$ represent the function $\theta$ with $(a, b, c)$ held constant. We define the summation $N_{(a, b, c)}(\theta,X)$ as: \begin{equation*} N_{(a, b, c)}(\theta, X) \vcentcolon= \sum_{(R, I) \in \Lambda_{(a, b, c)}} \theta_{(a, b, c)}(R, I), \end{equation*} where $\Lambda_{(a, b, c)}$ denotes the lattice points in the fiber over $(a, b, c)$, as defined previously. Our main result, Theorem \ref{Base-Poisson-Counting}, is as follows:

\begin{theorem}\label{Base-Poisson-Counting}
For any fixed $(a, b, c) \in B'_{X, \delta}$, and for any good $\delta_s$-approximation $\theta$, we have:
\begin{equation*}
N_{(a, b, c)}(\theta, X) = \frac{1}{|8a^2 \cdot 12a|} \int \int \theta_{(a, b, c)}(R, I) \, dR \, dI + O(X^{1/4 + 12\delta_s+50\delta}).
\end{equation*}
\end{theorem}

We establish this theorem by applying the Poisson summation formula, showing that the Fourier coefficients of $\theta$, except at zero, contribute negligibly. Since $\theta$ is the product of two functions, $\psi_{n}$ and $h_{(a,b,c)}$, its Fourier transform is given by the convolution of the Fourier transforms of these components. This decomposition allows us to analyze the resulting Fourier integral in four distinct regions, where we identify the dominant contributions in each.

Additionally, Lemma \ref{h-Fourier-Coeff} provides essential bounds on the Fourier transform of $h_{(a,b,c)}$, which are derived from the underlying geometry of the space. A key aspect of our approach is the structure of the lattice $\Lambda_{(a, b, c)}$, described in Proposition \ref{Finte-part-coeff}, which plays a crucial role in controlling the error terms.

Because our summation occurs over the lattice $\Lambda_{(a, b, c)}$, rather than over the standard lattice $\mathbb{Z}^2$, we must adapt the twisted Poisson summation formula accordingly. To do this, we first provide a local description of the lattice $\Lambda_{(a, b, c)}$.

\begin{lemma}\label{Lambda(a,b,c)} For any $(a, b, c) \in \mathbb{Z}^3$ with $a \neq 0$, let $\zeta = b^3 - 4abc$ and define 
\begin{equation*} \bar{\Lambda}_{(a, b, c)} \vcentcolon= \left\{ \left(\zeta + 8a^2k,  c^2-3kb\right) \in \frac{\mathbb{Z}}{8a^2 \cdot 12a \, \mathbb{Z}} \times \frac{\mathbb{Z}}{12a \, \mathbb{Z}} : k = 0, 1, \dots, 12|a| - 1 \right\}. \end{equation*} 
Then, $(R, I) \in \Lambda_{(a, b, c)}$ if and only if $(R \pmod{8a^2 \cdot 12a}, I \pmod{12a}) \in \bar{\Lambda}_{(a, b, c)}$. 
\end{lemma}
\begin{proof} 
 To demonstrate this, consider the expressions 
$\left( \zeta + 8a^2 k + 8a^2 \cdot 12a \cdot m,  c^2 -3kb+ 12a \cdot n \right)$
for integers $m$ and $n$. These expressions show that 
$R = \zeta + 8a^2 k + 8a^2 \cdot 12a \cdot m$ satisfies 
$R \equiv \zeta \pmod{8a^2}$, and
    \begin{equation*}
        I \equiv -3kb + c^2 \pmod{12a}.
    \end{equation*}
    Consequently, $(R \pmod{8a^2 \cdot 12a}, I \pmod{12a}) \in \bar{\Lambda}_{(a, b, c)}$ implies $(R, I) \in \Lambda_{(a, b, c)}$. Conversely, if $(R, I) \in \Lambda_{(a, b, c)}$, then we obtain unique values of $k$ and $m$ satisfying the congruence conditions as described in the lemma.
\end{proof}

For any smooth function $f$ with compact support on $\mathbb{R}^2$, we define its Fourier transform by
\begin{equation*}
    \widehat{f}(\alpha, \beta) \vcentcolon= \int_{\mathbb{R}^2} f(x, y) e^{-2 \pi i (x \alpha + y \beta)} \, dx \, dy.
\end{equation*}
For such smooth functions, the Poisson summation formula on $\mathbb{R}^2$ states that
\begin{equation*}
    \sum_{(m, n) \in \mathbb{Z}^2} f(m, n) = \sum_{(\alpha, \beta) \in \mathbb{Z}^2} \widehat{f}(\alpha, \beta).
\end{equation*}

\begin{proposition}\label{Poisson} 
Let $f$ be a smooth function in the $(R, I)$-space with compact support, and let $(a, b, c) \in \mathbb{Z}^3$ with $a \neq 0$. Let $\psi_{\Lambda_{(a, b, c)}}$ denote the characteristic function of $\Lambda_{(a, b, c)}$. Then we have: \begin{equation*} 
\sum_{(m, n) \in \mathbb{Z}^2} \psi_{\Lambda_{(a, b, c)}}(m, n) f(m, n) = \sum_{(\alpha, \beta) \in \mathbb{Z}^2} \widehat{\psi}_{\Lambda_{(a, b, c)}}(\alpha, \beta) \, \widehat{f} \left( \frac{\alpha}{8a^2 \cdot 12a}, \frac{\beta}{12a} \right), \end{equation*}
where $\widehat{\psi}_{\Lambda_{(a, b, c)}}(\alpha, \beta)$ is defined as \begin{equation*} 
\widehat{\psi}_{\Lambda_{(a, b, c)}}(\alpha, \beta) = \frac{1}{8a^2 \cdot 12a \cdot 12a} \sum_{k=0}^{12|a| - 1} e^{2 \pi i \alpha \xi_{\alpha,k}(a, b, c)} e^{2 \pi i \beta \xi_{\beta,k}(a, b, c)}, 
\end{equation*} 
with
\begin{equation*} \xi_{\alpha,k}(a, b, c) \vcentcolon= \frac{\zeta + 8a^2 k}{8a^2 \cdot 12a}, \quad \xi_{\beta,k}(a, b, c) \vcentcolon= \frac{c^2 - 3bk}{12a}, \end{equation*} and $\zeta = b^3 - 4abc$. \end{proposition}
\begin{proof}
By applying Lemma \ref{Lambda(a,b,c)}, we obtain:
    \begin{equation*}
        \sum_{(m, n) \in \mathbb{Z}^2} \psi_{\Lambda_{(a, b, c)}}(m, n) f(m, n) = \sum_{(m, n) \in \mathbb{Z}^2} \sum_{k=0}^{12|a| - 1} f(\zeta + 8a^2k + 8a^2 \cdot 12a \cdot m\, ,\,  c^2-3kb + 12a \cdot n).
    \end{equation*}
    Now, let us define
    \begin{equation*}
        f_k(R, I) \vcentcolon= f(\zeta + 8a^2k + 8a^2 \cdot 12a \cdot R \, ,\, c^2-3bk  + 12a \cdot I),
    \end{equation*}
    which allows us to rewrite the sum as
    \begin{align*}
        \sum_{(m, n) \in \mathbb{Z}^2} \sum_{k=0}^{12|a| - 1} f(\zeta + 8a^2k + 8a^2 \cdot 12a \cdot m, -3kb + c^2 + 12a \cdot n) 
        &= \sum_{(m, n) \in \mathbb{Z}^2} \sum_{k=0}^{12|a| - 1} f_{k}(m, n) \\
        &= \sum_{k=0}^{12|a| - 1} \sum_{(m, n) \in \mathbb{Z}^2} f_{k}(m, n).
    \end{align*}
    Applying the Poisson summation formula to each $f_k$, we get:
    \begin{equation*}
        \sum_{k=0}^{12|a| - 1} \sum_{(m, n) \in \mathbb{Z}^2} f_{k}(m, n) = \sum_{k=0}^{12|a| - 1} \sum_{(\alpha, \beta) \in \mathbb{Z}^2} \widehat{f}_{k}(\alpha, \beta).
    \end{equation*}
   The Fourier transform $\widehat{f}_k$ is related to the original transform $\widehat{f}$ by the following expression:
    \begin{equation*}
        \widehat{f}_k (\alpha, \beta) = \frac{1}{8a^2 \cdot 12a \cdot 12a} e^{2 \pi i \alpha \xi_{\alpha,k}(a, b, c)} e^{2 \pi i \beta \xi_{\beta,k}(a, b, c)} \, \widehat{f}\big(\frac{\alpha}{8a^2 \cdot 12a},\frac{\beta}{12a}\big)
        \end{equation*}
This completes the proof.
\end{proof}
%%%%%%%%%%%%%%%%%%%%%%%%%%%%%%%%%%%%%%%%%%%%%%%%%%%%%%%%%%%%%%%%%%%%%%%%%%%%%%%%%%%%%%%%%%%%%%%%%%%%%%%%%%%%%%%%%%%%%%%%%%%%%%%%%%%%%%%%%%%%%%%%%%%
\subsection{Bounds\ on Fourier coefficients of the height function}
Let $(a,b,c) \in B'_{X,\delta}$ be fixed for this section. For brevity, we will denote $R_0$ instead of $R_0(a,b,c)$. Let $\phi_\epsilon$ denote a smooth approximation of $\chi_{(-1,1)}$. We shall refer to it as a \textit{good $\epsilon$-approximation} if:

\begin{enumerate}
    \item the set of elements such that $|\phi_{\epsilon}(x)-\chi_{(-1,1)}(x)| > 4\epsilon$ is inside the union of two intervals with a size of $O(\epsilon)$.
    
    \item the real-valued function $\phi_{\epsilon}$ is smooth.
    
    \item For any $m\gg1$ and $\alpha$, then:
    \begin{equation}\label{Tm-property}
        \int \phi_{\epsilon}(x)e^{-imx^2}e^{-i\alpha x}dx=O\left(\frac{1+|\alpha/m|}{\epsilon \sqrt{|m|}}\right)
    \end{equation}
    \item The $k$-th derivatives of $\phi_\epsilon$ are bounded by $O(\epsilon^{-2k})$.
\end{enumerate}

\begin{lemma}
    For any $\epsilon$ less than one, there exist \textit{good $\epsilon$-approximations} $\phi_{1,\epsilon}$ and $\phi_{2,\epsilon}$ such that for any $x$:
    \begin{equation*}
        \phi_{1,\epsilon}(x) \leq \chi_{(-1,1)}(x) \leq \phi_{2,\epsilon}(x)
    \end{equation*}
\end{lemma}
\begin{proof}  
To construct $\phi_{1, \epsilon}$, a smooth approximation of the characteristic function $\chi_{(-1,1)}$, we begin with a function $\varphi_{1, \epsilon}$ designed to approximate $\chi_{(-1,1)}$ within the interval $(-1+\epsilon, 1-\epsilon)$. The function $\varphi_{1, \epsilon}$ is defined as follows. It takes the constant value $1 - \epsilon$ on the subinterval $(-1 + 2\epsilon, 1 - 2\epsilon)$. In the transition regions, it decreases linearly from $1 - \epsilon$ to $0$ over the intervals $(-1 + \epsilon, -1 + 2\epsilon)$ and $(1 - 2\epsilon, 1 - \epsilon)$. Outside the interval $(-1 + \epsilon, 1 - \epsilon)$, the function is identically zero.

Let $w$ be a smooth, compactly supported function on $[1/2, 1]$ that is positive and satisfies $\int w(y) \, dy = 1$. For $\gamma > 0$, we define the scaled function $w_{\gamma}$ as
\begin{equation*}
w_{\gamma}(y) = \frac{1}{\gamma} w\left(\frac{y}{\gamma}\right).
\end{equation*}
Then, we construct $\phi_{1, \epsilon}$ by convolving $w_{\epsilon^2 / 2}$ with $\varphi_{1, \epsilon}$:
\begin{equation*}
\phi_{1, \epsilon} = w_{\epsilon^2 / 2} * \varphi_{1, \epsilon}.
\end{equation*}
Our objective is to verify that $\phi_{1, \epsilon}$ meets the necessary conditions to be an $\epsilon$-good approximation of $\chi_{(-1,1)}$.
\begin{enumerate}
    \item  Support and Proximity to $\chi_{(-1,1)}$: By definition, $\phi_{1, \epsilon}$ is zero outside the interval $(-1, 1)$, as $\varphi_{1, \epsilon}$ itself vanishes outside $(-1 + \epsilon, 1 - \epsilon)$. For values of $x$ in $(-1 + 2\epsilon + \epsilon^2 / 2, 1 - 2\epsilon + \epsilon^2 / 4)$ and for any $y \in [1/2, 1]$, the function $\varphi_{1, \epsilon}(x - y \epsilon^2 / 2)$ remains constant at $1 - \epsilon$. This leads to the following calculation:

\begin{equation*}
\phi_{1, \epsilon}(x) = w_{\epsilon^2 / 2} * \varphi_{1, \epsilon}(x) = \int_{\mathbb{R}} \varphi_{1, \epsilon}(x - y \epsilon^2 / 2) w(y) \, dy = (1 - \epsilon) \int_{\mathbb{R}} w(y) \, dy = 1 - \epsilon. 
\end{equation*}
This confirms that $\phi_{1, \epsilon}(x) = 1 - \epsilon$ within the specified interval, ensuring that $|\phi_{1, \epsilon}(x) - \chi_{(-1,1)}(x)| \leq \epsilon$ everywhere except within two intervals of size $O(\epsilon)$.

\item Smoothness: Since $w_{\epsilon^2 / 2}$ is a smooth function, the convolution $\phi_{1, \epsilon} = w_{\epsilon^2 / 2} * \varphi_{1, \epsilon}$ is also smooth, fulfilling the smoothness requirement.

\item Fourier Decay Property: To bound the Fourier transform integral \begin{equation*} \int \phi_{1, \epsilon}(x) e^{-imx^2} e^{-i\alpha x} \, dx = \int w(y) \int \varphi_{1, \epsilon}(x - y \epsilon^2 / 2) e^{-imx^2} e^{-i\alpha x} \, dx \, dy, \end{equation*} 
we focus on bounding the integral of $ \varphi_{1, \epsilon}(x - y \epsilon^2 / 2) e^{-imx^2} e^{-i\alpha x} $ for fixed $ y $.

Firstly, in intervasl where $ \varphi_{1, \epsilon}(x-y\epsilon^2/2) $ is constant, we only need to bound an integral of the form
\begin{equation*} 
\left| \int_a^b e^{-imx^2} e^{-i\alpha x} \, dx \right|, 
\end{equation*} for fixed endpoints $a$ and $b$. By performing a change of variable, we transform this to an integral of the form
\begin{equation*} 
\frac{1}{\sqrt{|m|}}\left|\int_{a'}^{b'} e^{-ix^2} \, dx \right|, 
\end{equation*}
where $a'$ and $b'$ are fixed real numbers. This integral is $O(1)$, as Fresnel integral $\int_{o}^{x} e^{-it^2} \, dt$ is a bounded function of $x$.

Secondly, In regions where $\varphi_{1, \epsilon}$ varies linearly, we approximate $\varphi_{1, \epsilon}(x - y \epsilon^2 / 2)$ by a linear function, leading to an integral of the form 

\begin{equation*} \left| \int_{a}^{b} x e^{-imx^2} e^{-i\alpha x} \, dx \right|, \end{equation*} 
where $a$ and $b$ are fixed. Changing variables, this integral becomes \begin{equation*} \left| \int_{a'}^{b'} \frac{x - \frac{\alpha}{2\sqrt{|m|}}}{|m|} e^{-ix^2} , dx \right|, \end{equation*} where \begin{equation*} a' = a\sqrt{|m|} + \frac{\alpha}{2\sqrt{|m|}} \quad \text{and} \quad b' = b\sqrt{|m|} + \frac{\alpha}{2\sqrt{|m|}}. \end{equation*} This gives the bound \begin{equation*} O\left(\frac{1 + |\alpha| / |m|}{\sqrt{|m|}} \right). 
\end{equation*}

Finally, note that the slope of the linear regions in $\varphi_{1, \epsilon}$ is $O(\epsilon^{-1})$ by definition. This ensures that any contributions from the linear parts are controlled and satisfy the Fourier decay condition.

\item Derivative Bound: The derivative of $\phi_{1, \epsilon}$ is given by
    \begin{equation*}
    \frac{d}{dx} \phi_{1, \epsilon}(x) = \frac{d}{dx} \left( w_{\epsilon^2 / 2} * \varphi_{1, \epsilon} \right)(x) = \int_{\mathbb{R}} \frac{d}{dx} w_{\epsilon^2 / 2}(y) \cdot \varphi_{1, \epsilon}(x - y) \, dy.
    \end{equation*}
  Since $w_{\epsilon^2 / 2}$ is smooth and compactly supported, its derivative is also smooth and compactly supported. Consequently, $\frac{d}{dx} \phi_{1, \epsilon}(x)$ is well-behaved and satisfies the bound $O(\epsilon^{-2})$, thereby meeting the final requirement.

\end{enumerate}
Thus, $\phi_{1, \epsilon}$ satisfies all conditions of an $\epsilon$-good approximation of $\chi_{(-1,1)}$.

Demonstrating that $|\phi_{1,\epsilon}-\varphi_{1,\epsilon}|$ is less than $\epsilon/2$ implies that $\phi_{1,\epsilon}$ is smaller than $\chi_{(-1,1)}$. Hence, we aim to demonstrate:   

\begin{equation*}
\left|\int w(y) \left(\varphi_{1,\epsilon}(x-\epsilon^2/2y) - \varphi_{1,\epsilon}(x)\right) dy\right| \leq \frac{\epsilon}{2}.
\end{equation*}
This holds true for any $x$ because $\phi_{1,\epsilon}$ is defined with lines having a maximum derivative of $1/\epsilon$.

This conclusion implies the lemma for $\phi_{1,\epsilon}$. Moreover, similar reasoning and construction apply to the other function $\phi_{2,\epsilon}$.
\end{proof}

\begin{lemma}\label{h-Fourier-Coeff}
Let $\phi$ be a good $\epsilon$-approximation of $\chi_{(-1,1)}$. We define the smooth approximation of $\chi'_{(a, b, c)}$ by
\begin{equation}\label{h-def}
h(R, I) \vcentcolon= \phi\left(\frac{R}{R_0}\right) \phi\left(\frac{3HI}{4aCX} - X^{-1}\Gamma(a,b,c,R)\right),
\end{equation}
where $\Gamma(a,b,c,R)$ is defined by 
\begin{equation*}
     \frac{4aC}{3H}\Gamma(a,b,c,R)  =\JR.
\end{equation*}
We have the following bounds:
\begin{itemize} 
  \item[\textnormal{(a)}]  For any $|\beta R_0^2| \gg |a^2H|$ and $|a^2 H \alpha| < \epsilon^{-4} |\beta R_0|$, we have
    \begin{equation*}
    \widehat{h}(\alpha, \beta) = O \left( \mathrm{Vol}\left(\H'_{(a, b, c)}\right) \cdot \epsilon^{-6} \left|\frac{\beta R_0^2}{a^2 H}\right|^{-1/2} \right).
    \end{equation*}
    
     \item[\textnormal{(b)}] For any $\alpha$ and $\beta$, we have
    \begin{equation*}
    \widehat{h}(\alpha, \beta) = O \left( \mathrm{Vol}\left(\H'_{(a, b, c)}\right) \cdot \frac{1 + |R_0^2 / (X a^3)|^K}{\epsilon^{4K} (1 + |R_0 \alpha|)^K} \cdot \frac{1}{\epsilon^{2M} (1 + |a X \beta / H|)^M} \right).
    \end{equation*}
    
    \item[\textnormal{(c)}] For any $\beta$ and non-zero $\alpha$, we have
    \begin{equation*}
    \widehat{h}(\alpha, \beta) = O \left( \mathrm{Vol}\left(\H'_{(a, b, c)}\right) \cdot \frac{1 + |\beta R_0^2 / (a^2 H)|^K}{\epsilon^{2K} (R_0 |\alpha|)^K} \right).
    \end{equation*}
\end{itemize}
\end{lemma}

\begin{proof}
To prove part (a), we define the function $T_m$ as follows:
\begin{equation*}
T_m(x) \vcentcolon= \phi(x)e^{-imx^2}.
\end{equation*}
We begin by evaluating $\widehat{h}(\alpha, \beta)$ using the definition of $h(R, I)$:
\begin{equation*}
h(R, I) \vcentcolon= \phi\left(\frac{R}{R_0}\right) \phi\left(\left(I - \frac{27R^2}{48a^2H} - \frac{H^2}{48a^2}\right) \cdot \frac{H}{4aCX}\right),
\end{equation*}
and thus,
\begin{equation*}
\widehat{h}(\alpha, \beta) = \int_{R} \int_{I} h(R, I) e^{-2\pi i (\alpha R + \beta I)} \, dR \, dI.
\end{equation*}
Substituting $h(R, I)$ and scaling $R$ by $R_0$ and $I$ by $\frac{4aCX}{3H}$, we have:
\begin{equation*}
\widehat{h}(\alpha, \beta) = \frac{R_0 \cdot 4aCX}{3H} \cdot e^{-2\pi i \frac{\beta H^2}{48a^2}}\int_{R}\int_{I} \phi(R) \phi(I) e^{-2\pi i \beta \left(\frac{4aCX}{3H} I + \frac{27R_0^2 R^2}{48a^2H}\right)} e^{-2\pi i \alpha R R_0} \, dI \, dR.
\end{equation*}
This allows us to express
\begin{equation*}
\Big|\widehat{h}(\alpha, \beta)\Big| = \Big|\frac{1}{2} \cdot \text{Vol}(\H'_{(a, b, c)}) \cdot \widehat{\phi}\left(\beta \cdot \frac{4aCX}{3H}\right) \cdot \widehat{T}_m(\alpha R_0) \Big|,
\end{equation*}
where $m = c \cdot \frac{\beta R_0^2}{a^2 H}$ for some constant $c$. Since $\phi$ is an $\epsilon$-good approximation, it satisfies the Fourier decay property:
\begin{equation*}
\widehat{T}_m(\alpha R_0) = O\left(\frac{1 + \frac{|\alpha R_0|}{|m|}}{\epsilon\sqrt{|m|}}\right).
\end{equation*}
Using the assumption $\frac{|\alpha R_0|}{|m|} = O(\epsilon^{-4})$, we obtain the desired bound in part (a).

Since $\phi$ is smooth and compactly supported, with its $M$-th derivatives bounded by $O(\epsilon^{-2M})$, its Fourier transform $\widehat{\phi}(\beta)$ exhibits rapid decay. By applying integration by parts, differentiating $M$ times introduces a factor of $(i\beta)^{-M}$, leading to the bound:
\begin{equation*}
|\widehat{\phi}(\beta)| = O\left(\epsilon^{-2M}(1 + |\beta|)^{-M}\right).
\end{equation*}
To bound $\widehat{T}_m(\alpha)$, we consider the function $T_m(x) = \phi(x)e^{-imx^2}$. Since $\phi$ is compactly supported, we can estimate the $K$-th derivative of $T_m$ using the product rule. The $K$-th derivative consists of terms arising from both $\phi$ and $e^{-imx^2}$, where derivatives of $\phi$ are bounded by $O(\epsilon^{-2K})$, and derivatives of $e^{-imx^2}$ contribute factors of $|m|$. Consequently, the $K$-th derivative of $T_m$ is bounded by:
\begin{equation*}
O\left((1 + |m|^K) \epsilon^{-2K}\right).
\end{equation*}
Applying integration by parts to $\widehat{T}_m(\alpha)$ yields the bound:
\begin{equation*}
\widehat{T}_m(\alpha) = O\left((1 + |m|^K) \epsilon^{-2K} (1 + |\alpha|)^{-K}\right).
\end{equation*}
For any $\beta$ satisfying $|\beta R_0^2|\ll |a^2H|$, we estimate $\widehat{T}_m(\alpha)$ and $\widehat{\phi}(\beta)$ by $O(\epsilon^{-2K}(1 + |\alpha|)^{-K})$ and $O(\epsilon^{-2M}(1 + |\beta|)^{-M})$, respectively, leading to the desired bound.  
In the case where $|\beta R_0^2|\gg|a^2H|$, we consider:

\begin{equation*}
\widehat{T}_m(R_0\alpha) \widehat{\phi}\left(\beta \frac{4aCX}{H}\right) = O\left(\left|\frac{\beta R_0^2}{a^2H}\right|^K \epsilon^{-2K} (1 + |R_0\alpha|)^{-K} \cdot \epsilon^{-2M-2K}\left(\frac{\beta aX}{H}\right)^{-M-K}\right).
\end{equation*}
Then, since
\begin{equation*}
\left|\frac{\beta aX}{H}\right|^{-K} \cdot \left|\frac{\beta R_0^2}{a^2H}\right|^K = \left|\frac{R_0^2}{a^3X}\right|^K,
\end{equation*}
we obtain the desired bound.This completes the proof of parts (b) and (c).
\end{proof}

Define the function $\theta$  :
\begin{equation}\label{theta-def} \theta_j(a, b, c, R, I) \coloneqq \psi_{n}\left(\frac{at^4}{\lambda}, \frac{bt^2}{\lambda}, \frac{c}{\lambda}, \frac{Rt^6}{\lambda^3}, \frac{I}{\lambda^2}\right) \cdot \phi_{j, \epsilon}\left(\frac{R}{R_0}\right) \phi_{j, \epsilon}\left(\frac{3HI}{4aCX} - X^{-1}\Gamma(a, b, c, R)\right), \end{equation}
where $\epsilon$ is defined as $X^{-\delta_s}$ for some fixed positive constant $\delta_s$. For brevity, we will refer to $\theta_j$ simply as $\theta$ throughout this section. We say that $\theta$, as given in equation \eqref{theta-def}, is a \textit{good $\delta_s$-approximation}.  

For integers $(\alpha, \beta)$, we define $\epsilon_{(\alpha, \beta)}$ to be $1$ if $\alpha \equiv 3b\beta \pmod{12a}$, and $0$ otherwise. The following proposition demonstrates how working with the semi-invariant structure and, in particular, the $\Lambda$ lattice set, helps control the error and establish equidistribution.
\begin{proposition}\label{Finte-part-coeff}
For any $(a,b,c)\in \mathbb{Z}^3$ with $a \neq 0$, then:
\begin{equation*}
|\widehat \psi_{\Lambda(a,b,c)}(\alpha,\beta)| = \frac{\epsilon_{(\alpha,\beta)}}{|8a^2 \cdot 12a|}.
\end{equation*}
\end{proposition}
\begin{proof}
To compute $8a^2 \cdot 12a \cdot 12a \cdot |\widehat{\psi}_{\Lambda(a, b, c)}(\alpha, \beta)|$, we start by expanding the expression as follows:
\begin{equation*}
8a^2 \cdot 12a \cdot 12a \cdot |\widehat{\psi}_{\Lambda(a, b, c)}(\alpha, \beta)| = \left|\sum_{k=0}^{12|a|-1} e^{2\pi i \alpha \left(\frac{\zeta + 8a^2 k}{8a^2 \cdot 12a}\right)} e^{2\pi i \beta \left(\frac{c^2 - 3b k}{12a}\right)}\right|,
\end{equation*}
where $\zeta$ is $b^3-4abc$.

This sum simplifies by decomposing each exponential term and factoring out constants that are independent of $k$, yielding:
\begin{equation*}
\left|\sum_{k=0}^{12|a|-1} e^{2\pi i k \left(\frac{\alpha - 3b \beta}{12a}\right)}\right|.
\end{equation*}
This is a geometric series, which equals $12|a|$ if $\alpha \equiv 3b \beta \pmod{12a}$ (so each term is $1$), and zero otherwise. This completes the proof.
\end{proof}

\subsection{Proof of the equidistribution theorem} 
Applying twisted Poisson summation, our goal reduces to demonstrating the bound:
\begin{equation}\label{poissonerror}
 \sum_{(\alpha,\beta)\in \mathbb{Z}^2  \backslash \{(0,0)\}} \widehat \psi_{\Lambda(a,b,c)}(\alpha,\beta)\widehat\theta_{(a,b,c)}\left(\frac{\alpha}{8a^2 \cdot 12a},\frac{\beta}{12a}\right) = O\left(X^{1/4+\delta^*}\right)
\end{equation}
where $\delta^*$ is given by $12\delta_s + 50\delta$. To establish this result, we will present a series of lemmas, followed by corollaries that bound specific terms in equation \eqref{poissonerror}.

The following lemma establishes bounds on the Fourier coefficients of $\theta_{(a,b,c)}$ when either $\alpha$ or $\beta$ is sufficiently large.

\begin{lemma}\label{theta-fourier-coeff}
    Let consider $\beta\neq 0$ and $M$ and $K$ non-negative integers, then:
    \begin{align*}
        |\widehat \theta_{(a,b,c)}(\alpha,\beta)| \ll 
        &|R_0\cdot\frac{aX}{H}| \cdot \frac{1}{(1+\lambda^{3}/t^6|\alpha|)^K} \cdot \frac{1}{(1+\lambda^2|\beta|)^M}+\\
        &  |R_0\lambda^2| \cdot\frac{1}{(1+\lambda^{3}/t^6|\alpha|)^K}\cdot \frac{X^{2\delta_s(M+2)}}{(1+ |\beta aX/H|)^M}+\\
         &  |\frac{\lambda^{3}aX}{t^6H}| \cdot\frac{((R_0^2/(a^3X))^{K+2}+1) X^{4(K+2)\delta_s}}{(1+R_0|\alpha|)^K} \cdot\frac{1}{(1+ \lambda^2|\beta|)^M}+\\
          &|R_0\cdot\frac{aX}{H}| \cdot\frac{((R_0^2/(a^3X))^{K}+1) X^{4K\delta_s}}{(1+R_0|\alpha|)^K} \cdot\frac{X^{2M\delta_s}}{(1+ |\beta aX/H|)^M}.   
\end{align*} 
\end{lemma}
\begin{proof} For brevity, we define $g(R, I)$ as: \begin{equation*} g(R,I) \coloneqq \psi_{n} \left(\frac{a t^4}{\lambda}, \frac{b t^2}{\lambda}, \frac{c}{\lambda}, \frac{R t^6}{\lambda^3}, \frac{I}{\lambda^2}\right). \end{equation*} 
To prove the bound on $\widehat{g}(\alpha, \beta)$, we observe that $g(R, I)$ depends on $R$ and $I$ through scaled arguments of $\psi_n$. Since $\psi_n$ is smooth and compactly supported, its Fourier transform exhibits rapid decay. Specifically, the scaling factors $\frac{\lambda^3}{t^6}$ in $R$ and $\lambda^2$ in $I$ lead to the estimate:
\begin{equation*} \widehat{g}(\alpha, \beta) = O\left(\frac{\lambda^5}{t^6} \cdot \frac{1}{\left(1 + \frac{\lambda^3}{t^6} |\alpha|\right)^K} \cdot \frac{1}{\left(1 + \lambda^2 |\beta|\right)^M}\right), 
\end{equation*}
where $K$ and $M$ are positive constants, and the implied constant depends on bounded support of $\psi_n$.

Let $h$ be the smooth approximation of $\mathcal{H}'_{(a, b, c)}$ as defined in Lemma \ref{h-Fourier-Coeff}. Since $\theta$ is the product of $h$ and $g$, the convolution theorem states that its Fourier transform, $\widehat{\theta}$, is given by the convolution of their respective Fourier transforms $\widehat{h}$ and $\widehat{g}$: 
\begin{equation*} \widehat{\theta}(\alpha, \beta) = (\widehat{h} * \widehat{g})(\alpha, \beta). \end{equation*} 
Thus, we can bound $|\widehat{\theta}(\alpha, \beta)|$ by 
\begin{equation*} \big|\widehat{\theta}(\alpha, \beta)\big| \ll \Big|\int \int \widehat{g}(x, y) \, \widehat{h}(\alpha - x, \beta - y) \, dx \, dy\Big|, \end{equation*} 
where the integral form expresses the convolution explicitly.

We begin by decomposing the integral into four parts and bounding each term separately. First, we consider the portion of the integral over the region where $|x| > |\alpha|/2$ and $|y| > |\beta|/2$. In this region, our goal is to establish an upper bound for the integral.
\begin{equation*}
\int_{|x| > |\alpha|/2} \int_{|y| > |\beta|/2} \big|\widehat{g}(x, y) \widehat{h}(\alpha - x, \beta - y)\big| \, dx \, dy.
\end{equation*}

Since $|x| > |\alpha|/2$ and $|y| > |\beta|/2$, we can directly apply the bounds established for $\widehat{g}$ and $\widehat{h}$ in this region. Specifically, we bound $\widehat{g}(x, y)$ by
\begin{equation*}
\big|\widehat{g}(x, y) \big| \ll \frac{\lambda^5}{t^6} \cdot \frac{1}{\left(1 + \frac{\lambda^3}{t^6} |x|\right)^{K+2}} \cdot \frac{1}{(1 + \lambda^2 |y|)^{M+2}},
\end{equation*}
as derived earlier. For $\widehat{h}(\alpha - x, \beta - y)$, we apply a simple volume bound, $|\widehat{h}(\alpha - x, \beta - y)|\ll \text{Vol}\big(\H'_{(a,b,c)}\big)$, since $h$ is a smooth approximation of $\mathcal{H}'_{(a, b, c)}$. Combining these bounds, we obtain:
\begin{align*}
&\int_{|x| > |\alpha|/2} \int_{|y| > |\beta|/2} \big|\widehat{g}(x, y) \widehat{h}(\alpha - x, \beta - y)\big| \, dx \, dy \ll \\
& \int_{|x| > |\alpha|/2} \int_{|y| > |\beta|/2} \Big| \frac{\lambda^5}{t^6} \frac{1}{\left(1 + \frac{\lambda^3}{t^6} |x|\right)^{K+2}} \frac{1}{(1 + \lambda^2 |y|)^{M+2}} \cdot \text{Vol}\big(\H'_{(a,b,c)}\big) \Big| dx \, dy  \ll \\ 
&\big|R_0\cdot\frac{aX}{H}\big| \cdot \frac{1}{(1+\lambda^{3}/t^6|\alpha|)^K} \cdot \frac{1}{(1+\lambda^2|\beta|)^M}.
\end{align*}
This establishes the bound for the first term in the summation of the lemma.

Next, we consider the portion of the integral where $|x| > |\alpha|/2$ but $|y| \leq |\beta|/2$. This condition implies that $|\beta - y| \geq |\beta|/2$. For $\widehat{g}(x, y)$, we apply the same bound as in the previous region, emphasizing its decay in $x$ and setting $M = 0$. Specifically, we have  
\begin{equation*}
\Big|\widehat{g}(x, y)\Big| \ll \frac{\lambda^5}{t^6} \cdot \frac{1}{\left(1 + \frac{\lambda^3}{t^6} |x|\right)^{K+2}}.
\end{equation*}
For $\widehat{h}(\alpha - x, \beta - y)$, we use the results from Lemma \ref{h-Fourier-Coeff}, which show that $\widehat{h}(\alpha, \beta)$ can be expressed as a product of three terms: the volume term $\text{Vol}\big(\mathcal{H}'_{(a,b,c)}\big)$, the Fourier transform $\widehat{\phi}$, and the Fourier transform $\widehat{T}_m$. Specifically, in this region, we apply the Fourier decay estimate for $\widehat{\phi}((\beta - y) a X/ H)$ from the proof of Lemma \ref{h-Fourier-Coeff}. Since $\phi$ is smooth and compactly supported as a good $\epsilon$-approximation, we obtain  
\begin{equation*}
\Big|\widehat{\phi}((\beta - y) a X / H)\Big| \ll \frac{X^{2\delta_s (M+2)}}{(1 + |(\beta - y) a X / H|)^{M+2}}.
\end{equation*}
Combining these bounds, we obtain the following for the integral:
\begin{align*}
&\int_{|x| > |\alpha|/2} \int_{|y| \leq |\beta|/2} \big|\widehat{g}(x, y) \widehat{h}(\alpha - x, \beta - y)\big| \, dx \, dy \\ \ll \quad
&\int_{|x| > |\alpha|/2} \int_{|y| \leq |\beta|/2} \frac{\lambda^5}{t^6} \frac{1}{\left(1 + \frac{\lambda^3}{t^6} |x|\right)^{K+2}} \cdot \text{Vol}(\H'_{(a,b,c)}) \cdot \frac{X^{2\delta_s (M+2)}}{(1 + |(\beta - y) a X / H|)^{M+2}} \, dx \, dy \\ \ll\quad
& |R_0\lambda^2| \cdot\frac{1}{(1+\lambda^{3}/t^6|\alpha|)^K}\cdot \frac{X^{2\delta_s(M+2)}}{(1+ |\beta aX/H|)^M},
\end{align*}
which provides the second bound in the lemma.

Now, consider the integral over $|x| \leq |\alpha|/2$ and $|y| > |\beta|/2$. This assumption implies that $|\alpha - x| \geq |\alpha|/2$. For $\widehat{g}(x, y)$, we apply the bound focusing on the decay in $y$ and setting $K = 0$. Specifically, we have
\begin{equation*}
\big|\widehat{g}(x, y)\big| \ll \frac{\lambda^5}{t^6} \cdot \frac{1}{(1 + \lambda^2 |y|)^{M+2}}.
\end{equation*}
For $\widehat{h}(\alpha - x, \beta - y)$, we apply part (b) of Lemma \ref{h-Fourier-Coeff}. we have:
\begin{equation*}
\widehat{h}(\alpha - x, \beta - y) \ll \text{Vol}(\H'_{(a,b,c)}) \cdot \frac{X^{4 \delta_s (K+2) }\Big(\big(\frac{R_0}{a^3 X}\big)^{K+2} + 1\Big)}{(1 + |(\alpha - x) R_0|)^{K+2}}.
\end{equation*}
Combining these bounds, we obtain
\begin{align*}
&\int_{|x| \leq |\alpha|/2} \int_{|y| > |\beta|/2} \big|\widehat{g}(x, y) \widehat{h}(\alpha - x, \beta - y)\big| \, dx \, dy \\ \ll \quad
&\int_{|x| \leq |\alpha|/2} \int_{|y| > |\beta|/2}   \frac{\lambda^5}{t^6} \frac{1}{(1 + \lambda^2 |y|)^{M+2}} \cdot \text{Vol}(\H'_{(a,b,c)}) \cdot \frac{X^{4 \delta_s (K+2) }\Big(\big(\frac{R_0}{a^3 X}\big)^{K+2} + 1\Big)}{(1 + |(\alpha - x) R_0|)^{K+2}}  dx \, dy \\ \ll \quad
&\big|\frac{\lambda^{3}aX}{t^6H}\big| \cdot\frac{((R_0/(a^3X))^{K+2}+1) X^{4(K+2)\delta_s}}{(1+R_0|\alpha|)^K} \cdot\frac{1}{(1+ \lambda^2|\beta|)^M}
\end{align*}
This gives the bound for the third part of the sum in the lemma.

The final part of the integral is over the region where $|x| \leq |\alpha|/2$ and $|y| \leq |\beta|/2$. This condition implies that $|\alpha - x| \geq |\alpha|/2$ and $|\beta - y| \geq |\beta|/2$. For $\widehat{g}(x, y)$, we use the decay bound with $K = 2$ and $M = 2$, yielding:
\begin{equation*}
\big|\widehat{g}(x, y)\big| \ll \frac{\lambda^5}{t^6} \cdot \frac{1}{\left(1 + \frac{\lambda^3}{t^6} |x|\right)^{2}} \cdot \frac{1}{(1 + \lambda^2 |y|)^{2}}.
\end{equation*}
For $\widehat{h}(\alpha - x, \beta - y)$, we apply part (b) of Lemma \ref{h-Fourier-Coeff}, then we have:

\begin{equation*}
\big|\widehat{h}(\alpha - x, \beta - y)\big| \ll \text{Vol}(\H'_{(a,b,c)}) \cdot \frac{X^{4K\delta_s}(1 + |R_0^2 / (X a^3)|^{K})}{ (1 + |R_0 \alpha|)^{K}} \cdot \frac{X^{2M\delta_s }}{ (1 + |a X \beta / H|)^{M}}.
\end{equation*}
Combining these bounds, we have
\begin{align*}
\int_{|x| \leq |\alpha|/2} \int_{|y| \leq |\beta|/2} \big|\widehat{g}(x, y) \widehat{h}(\alpha - x, \beta - y)\big| \, dx \, dy \ll \int_{|x| \leq |\alpha|/2} \int_{|y| \leq |\beta|/2} & \frac{\lambda^5}{t^6} \frac{1}{\left(1 + \frac{\lambda^3}{t^6} |x|\right)^{2}} \frac{1}{(1 + \lambda^2 |y|)^{2}} \\
& \cdot \text{Vol}(\H'_{(a,b,c)}) \cdot \frac{X^{4K\delta_s}(1 + |R_0^2 / (X a^3)|^{K})}{(1 + |R_0 \alpha|)^{K}} \\
& \cdot \frac{X^{2M\delta_s}}{ (1 + |a X \beta / H|)^{M}} \, dx \, dy.
\end{align*}
This completes the proof of the theorem by providing the bound for the final part of the integral.
\end{proof}
\begin{cor}\label{beta}
Assume $\delta_e$ sufficiently small and positive. Consider the sum over all $(\alpha, \beta)$ with $|\beta| > X^{2\delta_s + \delta_e}$ in Equation \eqref{poissonerror}. Then:
\begin{equation*}
    \sum_{\substack{(\alpha, \beta) \in \mathbb{Z}^2 \backslash \{(0, 0)\} \\ |\beta| > X^{2\delta_s + \delta_e}}} \Big| \widehat{\psi}_{\Lambda(a, b, c)}(\alpha, \beta) \, \widehat{\theta}_{(a, b, c)}\left(\frac{\alpha}{8a^2 \cdot 12a}, \frac{\beta}{12a}\right) \Big| = O(X^{-3}).
\end{equation*}    
\end{cor}
\begin{proof}
By Proposition \ref{Finte-part-coeff}, we know that $\widehat{\psi}_{\Lambda(a, b, c)}(\alpha, \beta)$ exhibits sufficient decay. Thus, it suffices to establish the bound:
\begin{equation*}
    \sum_{\substack{(\alpha, \beta) \in \mathbb{Z}^2 \backslash \{(0, 0)\} \\ |\beta| > X^{2\delta_s + \delta_e}}} \Big| \frac{1}{a^3} \, \widehat{\theta}_{(a, b, c)}\left(\frac{\alpha}{8a^2 \cdot 12a}, \frac{\beta}{12a}\right) \Big| = O(X^{-3}).
\end{equation*}
We now apply Lemma \ref{theta-fourier-coeff} with $K = 2$ and demonstrate that each term in $\widehat{\theta}_{(a, b, c)}(\alpha, \beta)$ can be sufficiently bounded by choosing $M$ large enough. To achieve this, we ensure that both $|\beta| \lambda^2 / a$ and $X^{-2\delta_s} |\beta X / H|$ are large, allowing us to control the terms through decay factors of the form $(1 + \lambda^2 |\beta/a|)^M$ and $(1 + |\beta| X / H)^M$. Specifically, we observe that:
\begin{itemize}
    \item The term $|\beta \lambda^2 / a|$ satisfies $|\beta \lambda^2 / a| > X^{1/2 - \delta_a}$, ensuring that the factor $(1 + \lambda^2 |\beta/a|)^M$ achieves sufficient decay.
    \item Similarly, the bound $X^{-2\delta_s} |\beta| X / H \gg X^{\delta_e}$ ensures that the decay factor $X^{-2M\delta_s}(1 + |\beta X / H|)^M$ effectively suppresses contributions as $M$ increases.
\end{itemize}
With these conditions in place, each term in $\widehat{\theta}_{(a, b, c)}(\alpha, \beta)$ decays sufficiently. By selecting $M$ appropriately, we confirm that the total sum is bounded by $O(X^{-3})$, completing the proof of the corollary.
\end{proof}
\begin{cor}\label{alpha}
Assume $\delta_e$ sufficiently small and positive. Consider the sum over all $(\alpha, \beta)$ with $|\alpha| > aX^{4\delta_s + \delta_a+\delta_e}$ in Equation \eqref{poissonerror}. Then:
\begin{equation*}
      \sum_{\substack{(\alpha, \beta) \in \mathbb{Z}^2 \backslash \{(0, 0)\} \\ |\alpha| > a X^{4\delta_s + \delta_a + \delta_e}}} 
    \Big| \widehat{\psi}_{\Lambda(a, b, c)}(\alpha, \beta) \, \widehat{\theta}_{(a, b, c)}\left(\frac{\alpha}{8a^2 \cdot 12a}, \frac{\beta}{12a}\right) \Big|= O(X^{-3}).
\end{equation*}
\end{cor}
\begin{proof}
 We aim to bound the following sum:
\begin{equation*}
    \sum_{\substack{(\alpha, \beta) \in \mathbb{Z}^2 \backslash \{(0, 0)\} \\ |\alpha| > a X^{4\delta_s + \delta_a + \delta_e}}} 
    \left| \frac{1}{a^3} \widehat{\theta}_{(a, b, c)}\left(\frac{\alpha}{8a^2 \cdot 12a}, \frac{\beta}{12a}\right) \right| = O(X^{-3}).
\end{equation*}
Similar to Corollary \ref{beta}, we apply Lemma \ref{theta-fourier-coeff} with $M = 2$ and select $K$ sufficiently large. According to the lemma, we need to ensure that both
\begin{equation*}
    \frac{\lambda^3}{a^3 t^6} |\alpha| \quad \text{and} \quad \left( \frac{R_0^2}{a^3 X} \, X^{4\delta_s} \right)^{-1} \frac{R_0 \alpha}{a^3}
\end{equation*}
are sufficiently large to control the error by choosing $K$ large enough. Specifically, we have:

\begin{itemize}
    \item $\frac{\lambda^3}{a^3 t^6} |\alpha|$ is greater than $X^{1/2 - 2\delta_a - 6\delta_t + 4\delta_s}$, allowing us to control this part of the sum effectively with decay in $|\alpha|$.

    \item Using the fact that $R_0$ is bounded by $|X^{1/2 + \delta_a} a |H|^{1/2}|$, we find that the second term, $\left( \frac{R_0^2}{a^3 X} \, X^{4\delta_s} \right)^{-1} \frac{R_0 \alpha}{a^3}$, is greater than $X^{\delta_e}$, which provides the necessary decay for this term.
\end{itemize}
With both conditions satisfied, we can choose $K$ sufficiently large to guarantee that each term in the summation decays as required. Consequently, the total sum is bounded by $O(X^{-3})$, completing the proof of the corollary.   
\end{proof}

The following lemma provides a useful bound for controlling the error when $\alpha$ is small and
$\beta$ is non-zero.
\begin{lemma}\label{theta-fourier-coeff-zero} For any $\beta$ such that $|\beta R_0^2| \gg |a^2 H|$, we have \begin{equation*}
\widehat{\theta}_{(a, b, c)}(\alpha, \beta) = O\left(\frac{\mathrm{Vol}\big(\H'_{(a,b,c)}\big)}{(1 + \lambda^2 |\beta|)^M} + X^{6\delta_s} \, \mathrm{Vol}\big(\H'_{(a,b,c)}\big) \left|\frac{\beta R_0^2}{a^2 H}\right|^{-1/2}\right). \end{equation*} 
\end{lemma}

\begin{proof} Similar to the previous lemma, we decompose the integral into three parts and bound each term separately. For the region where $|y| > |\beta|/2$, we obtain:

\begin{align*} 
\int \int_{|y| > |\beta/2|} \big|\widehat{g}(x, y) \, \widehat{h}(\alpha - x, \beta - y)\big| \, dx \, dy &\ll \int \int_{|y| > |\beta/2|} \Big| \frac{\lambda^5}{t^6} \frac{1}{\left(1 + \frac{\lambda^3}{t^6} |x|\right)^{2}} \frac{1}{(1 + \lambda^2 |y|)^{M+2}} \cdot \mathrm{Vol}\big(\H'_{(a,b,c)}\big) \Big| dx \, dy. \end{align*}
This integral is bounded above by  
\begin{equation*}  
O \left(\frac{\mathrm{Vol}\big(\mathcal{H}'_{(a,b,c)}\big)}{(1 + \lambda^2 |\beta|)^M}\right).  
\end{equation*}  
To proceed with the proof, we now consider the portion of the integral where $|y| \leq |\beta|/2$, further dividing it into two subcases based on the relative magnitudes of $|(\alpha - x)a^2 H|$ and $|R_0 (\beta - y) X^{4\delta_s}|$.

\medskip
\noindent\textbf{Case 1:} $|y| \leq |\beta/2|$ and $|(\alpha - x)a^2 H| \ll |R_0 (\beta - y) X^{4\delta_s}|$.

In this region, we observe that $|\beta - y| \geq |\beta|/2$. Moreover, since $|\beta R_0^2| \gg |a^2 H|$, the dominant term controlling $\widehat{h}$ depends on $|\beta - y|$ rather than $|\alpha - x|$. Thus, we apply part (a) of Lemma \ref{h-Fourier-Coeff}, which yields:
\begin{equation*}
\big|\widehat{h}(\alpha - x, \beta - y)\big| \ll 
\text{Vol}(\H'_{(a,b,c)}) \cdot X^{6\delta_s} \left|\frac{\beta R_0^2}{a^2 H}\right|^{-1/2}.
\end{equation*}
Combining this bound with the decay for $\widehat{g}(x, y)$, we have:
\begin{align*}
    &\int_{|y| \leq |\beta/2|} \int_{|(\alpha - x)a^2 H| \ll |R_0 (\beta - y) X^{4\delta_s}|} \big|\widehat{g}(x, y)\big| \, \text{Vol}\big(\H'_{(a,b,c)}\big) \cdot X^{6\delta_s} \left|\frac{\beta R_0^2}{a^2 H}\right|^{-1/2} \, dx \, dy \\
    &\ll X^{6\delta_s} \, \Vol\big(\H'_{(a,b,c)}\big) \left|\frac{\beta R_0^2}{a^2 H}\right|^{-1/2},
\end{align*}
which bounds this part of the integral.

\medskip
\noindent\textbf{Case 2:} $|y| \leq |\beta/2|$ and $|(\alpha - x)a^2 H| \gg |R_0 (\beta - y) X^{4\delta_s}|$.

In this region, we apply part (c) of Lemma \ref{h-Fourier-Coeff}, which is most effective when $|\alpha - x|$ is the dominant term in the bound. Specifically, part (c) of Lemma \ref{h-Fourier-Coeff} yields:
\begin{equation*}
\big|\widehat{h}(\alpha - x, \beta - y)\big| \ll \Vol\big(\H'_{(a,b,c)}\big)\cdot \frac{X^{2K \delta_s}(1 + |(\beta - y) R_0^2 / (a^2 H)|^K)}{(R_0 |\alpha - x|)^K}.
\end{equation*}
Combining this with $\widehat{g}(x, y)$, we have:
\begin{align*}
    &\int_{|y| \leq |\beta/2|} \int_{|(\alpha - x)a^2 H| \gg |R_0 (\beta - y) X^{4\delta_s}|} \big|\widehat{g}(x, y) \big| \cdot\Vol\big(\H'_{(a,b,c)}\big)\ \cdot \frac{X^{2K \delta_s}(1 + |(\beta - y) R_0^2 / (a^2 H)|^K)}{(R_0 |\alpha - x|)^K} \, dx \, dy.
\end{align*}
Since $|\beta - y| \geq |\beta|/2$ and $|\beta R_0^2| \gg |a^2 H|$, this bound becomes:
\begin{equation*}
\frac{|(\beta - y) R_0^2 / (a^2 H)|^K}{(R_0 |\alpha - x|)^K} = O(X^{-4\delta_s K}).
\end{equation*}
By choosing $K$ large enough, this part is as small as $O(X^{-3})$, making it negligible. This completes the bound for this portion of the integral and concludes the proof.
\end{proof}
\begin{cor}\label{beta-nonzero}
   Assume $\delta_e$ is sufficiently small and positive. Consider the sum in Equation \eqref{poissonerror} over $(\alpha, \beta) \in \mathbb{Z}^2 \backslash \{(0, 0)\}$ with $|\alpha| \leq aX^{4\delta_s + \delta_a + \delta_e}$ and non-zero $\beta$ satisfying $0 < |\beta| \leq X^{2\delta_s + \delta_e}$, we have \begin{equation*} \sum_{\substack{(\alpha, \beta) \in \Z^2 \backslash\{(0, 0)
\} \\ |\alpha| \leq aX^{4\delta_s + \delta_a + \delta_e} \\ 0 < |\beta| \leq X^{2\delta_s + \delta_e}}} \Big|  \widehat{\psi}_{\Lambda(a, b, c)}(\alpha, \beta) \, \widehat{\theta}_{(a, b, c)}\left(\frac{\alpha}{8a^2 \cdot 12a}, \frac{\beta}{12a}\right) \Big| = O(X^{1/4 + \delta^*}), \end{equation*}
where $\delta^* = 4\delta_a + \delta_H + 12\delta_s + 3\delta_e$.
\end{cor}
\begin{proof}
    To prove this, we use Lemma \ref{theta-fourier-coeff-zero}. Expanding the expression with this lemma, we have: \begin{align*} &\sum_{\substack{(\alpha, \beta) \in \mathbb{Z}^2 \backslash {(0, 0)} \\ |\alpha| \leq aX^{4\delta_s + \delta_a + \delta_e} \\ 0 < |\beta| \leq X^{2\delta_s + \delta_e}}} \Big| \frac{\epsilon_{(\alpha, \beta)}}{a^3} \, \widehat{\theta}_{(a, b, c)}\Big(\frac{\alpha}{8a^2 \cdot 12a}, \frac{\beta}{12a}\Big) \Big|\\  \ll\quad
&\sum_{\substack{(\alpha, \beta) \in \mathbb{Z}^2 \backslash {(0, 0)} \\|\alpha| \leq aX^{4\delta_s + \delta_a + \delta_e} \\ 0 < |\beta| \leq X^{2\delta_s + \delta_e}}} \Big|\frac{\epsilon_{(\alpha, \beta)}}{a^3}\Big| \Big( \frac{\Vol\big(\H'_{(a,b,c)}\big)}{(1 + \lambda^2 |\beta / a|)^M} + X^{6\delta_s} \, \Vol\big(\H'_{(a,b,c)}\big)\Big|\frac{\beta R_0^2}{a^3 H}\Big|^{-1/2} \Big). \end{align*}

Since $\lambda^2 / |a|$ is greater than $X^{1/4}$, we can choose $M$ sufficiently large to guarantee that the first term in the sum is of order $O(X^{-3})$. Thus, our main task is to bound the second term in the sum:
\begin{equation*} \sum_{\substack{(\alpha, \beta) \in \mathbb{Z}^2 \backslash {(0, 0)} \\ |\alpha| \leq aX^{4\delta_s + \delta_a + \delta_e} \\ 0 < |\beta| \leq X^{2\delta_s + \delta_e}}} \Big| \frac{\epsilon_{(\alpha, \beta)}}{a^3} \, X^{6\delta_s} \, \Vol\big(\H'_{(a,b,c)}\big) \cdot \Big|\frac{\beta R_0^2}{a^3 H}\Big|^{-1/2} \Big|. \end{equation*}
Using the fact that $\Vol\big(\H'_{(a,b,c)}\big)$ can be approximated by $O(|R_0a X / H|)$, we find that: 
\begin{equation*}  \frac{|a|^{-3} \Vol\big(\H'_{(a,b,c)}\big)|a|^{3/2} |H|^{1/2}}{R_0}  \leq X^{1/4 + \delta_a / 2 + \delta_H / 2}. \end{equation*}
Thus, we have:
\begin{equation*} \sum_{\substack{(\alpha, \beta) \in \mathbb{Z}^2 \backslash {(0, 0)} \\ |\alpha| \leq aX^{4\delta_s + \delta_a + \delta_e} \\ 0 < |\beta| \leq X^{2\delta_s + \delta_e}}} \Big| \frac{\epsilon_{(\alpha, \beta)}}{a^3} \, X^{6\delta_s} \, \Vol\big(\H'_{(a,b,c)}\big)\Big|\frac{\beta R_0^2}{a^3 H}\Big|^{-1/2} \Big| \ll 
X^{1/4 + \delta_a / 2 + \delta_H / 2 + 6\delta_s} \cdot \sum_{\substack{(\alpha, \beta) \in \mathbb{Z}^2 \backslash {(0, 0)} \\ |\alpha| \leq aX^{4\delta_s + \delta_a + \delta_e} \\ 0 < |\beta| \leq X^{2\delta_s + \delta_e}}} \epsilon_{(\alpha, \beta)}. \end{equation*}
For a fixed $\beta$, the term $\epsilon_{(\alpha, \beta)}$ is nonzero only when $\alpha$ is fixed modulo $12a$. Consequently, we obtain:
\begin{equation*} 
X^{6\delta_s + \delta_a / 2 + \delta_e / 2} \cdot \sum_{\substack{(\alpha, \beta) \in \mathbb{Z}^2 \backslash {(0, 0)} \\ |\alpha| \leq aX^{4\delta_s + \delta_a + \delta_e} \\ 0 < |\beta| \leq X^{2\delta_s + \delta_e}}} \epsilon_{(\alpha, \beta)} \ll X^{1/4 + \delta^*}. 
\end{equation*}
 where $\delta^* = 4\delta_a + \delta_H + 12\delta_s + 3\delta_e$, concluding this corollary. 
\end{proof}

\begin{lemma}\label{theta-fourir-beta=0} For any non-zero $\alpha$, we have: 
\begin{equation*} \widehat{\theta}_{(a, b, c)}(\alpha, 0) = O\left(\frac{\Vol\big(\H'_{(a,b,c)}\big)}{\left(1 + \frac{\lambda^3}{t^6} |\alpha|\right)^K} + \Vol\big(\H'_{(a,b,c)}\big) \cdot \frac{X^{3 \delta_s+2\delta_a}}{|R_0 \alpha|}\right). \end{equation*} 
\end{lemma}

\begin{proof}
 Similar to previous lemmas, we aim to bound the following integral: \begin{equation*} 
 \int \int |\widehat{g}(x, y) \, \widehat{h}(\alpha - x, -y)| \, dx \, dy. 
 \end{equation*} 
 First, consider the region where $|x| > |\alpha|/2$. In this region, we have: 
 \begin{equation*} 
 \int_{|x| > |\alpha|/2} \int |\widehat{g}(x, y) \, \widehat{h}(\alpha - x, -y)| \, dx \, dy \ll \int_{|x| > |\alpha|/2} \int \left| \frac{\lambda^5}{t^6} \frac{1}{\left(1 + \frac{\lambda^3}{t^6} |x|\right)^{K+2}} \frac{1}{(1 + \lambda^2 |y|)^2} \cdot \mathrm{Vol}(\H'_{(a,b,c)}) \right| dx \, dy. 
 \end{equation*}
This last expression is bounded by: \begin{equation*} O\left( \frac{\Vol\big(\H'_{(a,b,c)}\big)}{\big(1 + \frac{\lambda^3}{t^6} |\alpha|\big)^K} \right). \end{equation*}
Now we consider the region where $|x| \leq |\alpha|/2$ and $|y R_0^2| \gg X^{\delta_s + 2 \delta_a} |a^2 H|$. In this case, we have:
\begin{align*}
    &\int_{|x| \leq |\alpha|/2} \int_{|y R_0^2| \gg X^{\delta_s + 2 \delta_a} |a^2 H|} |\widehat{g}(x, y) \, \widehat{h}(\alpha - x, -y)| \, dx \, dy \\
    \ll \quad &\int_{|x| \leq |\alpha|/2} \int_{|y R_0^2| \gg X^{\delta_s + 2 \delta_a} |a^2 H|} \left| \frac{\lambda^5}{t^6} \frac{1}{\left(1 + \frac{\lambda^3}{t^6} |x|\right)^2} \frac{1}{(1 + \lambda^2 |y|)^{M+2}} \cdot \Vol\big(\H'_{(a,b,c)}\big) \right| dx \, dy \\
    \ll\quad&\Vol\big(\H'_{(a,b,c)}\big) \cdot \frac{1}{\left(1 + \lambda^2 |a^2 H X^{\delta_s + 2 \delta_a} / R_0^2|\right)^M}.
\end{align*}
Since $\lambda^2 |a^2 H X^{\delta_s + 2 \delta_a}| / R_0^2$ is greater than $X^{\delta_s}$, choosing $M$ large enough ensures that this part is bounded by $O(X^{-3})$.

Next, we consider the region where $|x| \leq |\alpha|/2$ and $|y R_0^2| \ll X^{\delta_s + 2 \delta_a} |a^2 H|$. Here, we can use part (c) of Lemma \ref{h-Fourier-Coeff} to bound the integral as follows:
\begin{align*}
    &\int_{|x| \leq |\alpha|/2} \int_{|y R_0^2| \ll X^{\delta_s + 2 \delta_a} |a^2 H|} |\widehat{g}(x, y) \, \widehat{h}(\alpha - x, -y)| \, dx \, dy \\
    \ll \quad &\int_{|x| \leq |\alpha|/2} \int_{|y R_0^2| \ll X^{\delta_s + 2 \delta_a} |a^2 H|} \left|\widehat{g}(x, y) \cdot \Vol\big(\H'_{(a,b,c)}\big) \cdot \frac{X^{3 \delta_s + 2 \delta_a}}{(\alpha - x) R_0}\right| \, dx \, dy \\
    \ll \quad &\Vol\big(\H'_{(a,b,c)}\big)\ \cdot \frac{X^{3 \delta_s + 2 \delta_a}}{|\alpha| R_0}.
\end{align*}
\end{proof}
\begin{cor}\label{beta-zero}
   Assume $\delta_e$ is sufficiently small and positive. Consider the sum in Equation \eqref{poissonerror} over $(\alpha, \beta) \in \mathbb{Z}^2 \backslash \{(0, 0)\}$ with $0<|\alpha| \leq aX^{4\delta_s + \delta_a + \delta_e}$, we have:
   \begin{equation*}
    \sum_{\substack{1 \leq |\alpha| \leq aX^{4\delta_s + \delta_a + \delta_e}}} 
    \Big| \widehat{\psi}_{\Lambda(a, b, c)}(\alpha, 0) \, \widehat{\theta}_{(a, b, c)}\left(\frac{\alpha}{8a^2 \cdot 12a}, 0\right) \Big|=O(X^{\delta^*}),
   \end{equation*}
where $\delta^* = 4\delta_a + \delta_H + 12\delta_s + 3\delta_e$.
\end{cor}
\begin{proof}
Consider the sum over all pairs $(\alpha, \beta)$ satisfying $|\alpha| \leq a X^{4\delta_s + \delta_a + \delta_e}$ and $\beta = 0$. Under this condition, $\epsilon_{(\alpha, 0)}$ is nonzero only when $a$ divides $\alpha$. Applying Lemma \ref{theta-fourir-beta=0}, we obtain:

\begin{align*}
    &\sum_{ 1 \leq |\alpha| \leq aX^{4\delta_s + \delta_a + \delta_e}} 
    \Big| \frac{\epsilon_{(\alpha, 0)}}{a^3} \, \widehat{\theta}_{(a, b, c)}\left(\frac{\alpha}{8a^2 \cdot 12a}, 0\right) \Big| \\ \ll
   \quad & \sum_{ 1 \leq |\alpha| \leq X^{4\delta_s + \delta_a + \delta_e}} 
    \Big| \frac{1}{a^3} \Big| \left( \frac{\Vol\big(\H'_{(a,b,c)}\big)}{\big(1 + | \lambda^3\alpha/(a^2t^6)|\big)^K} + \Vol\big(\H'_{(a,b,c)}\big) \cdot \frac{X^{3 \delta_s + 2 \delta_a}}{|R_0 \alpha/a^2|} \right).
\end{align*}
To handle the first term, we observe that $|\lambda^3\alpha/(a^2t^6)|$ is greater than $X^{1/4}$ for sufficiently small choices of $\delta$. By selecting $K$ large enough, the first term in the sum attains an order of $O(X^{-3})$.  Next, for the second term in the sum, we have:
\begin{equation*}
    \frac{1}{|a^3|}\cdot X^{4\delta_s + \delta_a + \delta_e} \cdot X^{3\delta_s+2\delta_a} \cdot \frac{\Vol\big(\H'_{(a,b,c)}\big)}{|R_0 \alpha /a^2|}.
\end{equation*}
Since $\Vol\big(\H'_{(a,b,c)}\big)$ can be approximated as $O(|R_0 aX/H|)$, this term simplifies to:
\begin{equation*}
   O( X^{4\delta_s + \delta_a + \delta_e} \cdot X^{3\delta_s+2\delta_a+\delta_H}),
\end{equation*}
which implies the sum is bounded by $O(X^{\delta^*})$. This completes the proof of the theorem.
\end{proof}

The combination of Corollaries \ref{beta}, \ref{alpha}, \ref{beta-nonzero}, and \ref{beta-zero} establishes Theorem \ref{Base-Poisson-Counting}.

%%%%%%%%%%%%%%%%%%%%%%%%%%%%%%%%%%%%%%%%%%%%%%%%%%%%%%%%%%%%%%%%%%%%%%%%%%%%%%%%%%%%%%%%%%%%%%%%%%%%%%%%%%%%%%%
\section{Asymptotic volume estimates }\label{sec4}
Davenport's lemma states that, for any bounded semialgebraic set in Euclidean space $\mathbb{R}^n$, the number of $\mathbb{Z}^n$ points contained in the set is approximately equal to its volume, with an error term that depends on the volumes of its projections. This lemma has been applied in previous work to analyze $\Bod$. However, we were unable to achieve the desired error bounds using this approach.  

Through a detailed study of the semi-invariant space, we found that the number of points in $\Bod$ within each fiber can be effectively approximated by the volume of the fiber while also satisfying the required error bounds. In this section, we complete the proof of Theorem \ref{The-Main-counting-Theo} by showing that summing the volumes of the fibers asymptotically recovers the volume of the entire space. Consequently, we approximate the number of points in $\Bod$ by its volume.
\subsection{Relating the Sum of $N_{(a,b,c)}$ to Volume}
Our goal in this subsection is to approximate $\Bodi$ using an integral. Specifically, we express $\Bodi$ as the integral of $(nt) \cdot \Psi(X^{-1/2} \cdot f) \cdot \h_{X}(f)$ over $V_{\mathbb{R}}$. We establish that the total volume is well approximated by summing the volumes of the fibers over $(a, b, c)$ and prove the following result:

\begin{theorem}\label{Goal-fiber-volume}
     \begin{equation*}
         \Bod = \int_{f \in V_{\delta} }(nt \cdot \Psi^{(i)})(X^{-1/2} \cdot f) \cdot \h_{X}(f) \, df + O(X^{2-\epsilon'}),
        \end{equation*} 
for some $\epsilon' > 0$.
\end{theorem}
Using the definitions of $\theta_1$ and $\theta_2$, and applying Lemma \ref{H-Is-Big}, we obtain:
\begin{equation*} 
\sum_{(a,b,c) \in B'_{X,\delta}} N_{(a,b,c)}(\theta_1, X) \leq \Bod + O_{\epsilon}(X^{2-2\delta+\epsilon}) \leq \sum_{(a,b,c) \in B'_{X,\delta}} N_{(a,b,c)}(\theta_2, X).
\end{equation*}  
To complete the proof of Theorem \ref{Goal-fiber-volume}, it suffices to establish the following proposition:

\begin{proposition}\label{Gluing-fiber-volumes} Let $\delta$ be sufficiently small, and set $\delta_s = 22\delta$. For any $\theta$ that is a good $\delta_s$-approximation, we have: 
\begin{equation} \sum_{(a,b,c)\in B_{X,\delta}'} N_{(a,b,c)}(\theta,X) = \int_{f\in V_{\delta}} (nt \cdot \Psi)(X^{-1/2} \cdot f) \cdot \h_{X}(f) \, df + O(X^{2-2\delta+\epsilon})+O(X^{7/4+500\delta}). \end{equation} 
\end{proposition}
From this point onward, we set $\delta_s = 22\delta$.  

We prove this proposition by applying the Mean Value Theorem along with the lemmas from Section 2, which allow us to initially restrict our attention to $B'_{X,\delta}$. Fix $(a,b,c) \in B'_{X,\delta}$. By Theorem \ref{Base-Poisson-Counting}, we approximate $N_{(a,b,c)}(\theta, X)$ as follows:
\begin{equation*}   
\sum_{(a,b,c)\in B'_{X,\delta}} N_{(a,b,c)}(\theta,X) = \sum_{(a,b,c)\in B'_{X,\delta}} \frac{1}{|8a^2 \cdot 12a|}\int \int \theta_{(a,b,c)}(R,I) \, dR \, dI + O\
\Big(\sum_{(a,b,c)\in B'_{X,\delta}} X^{1/4+50\delta+12\delta_s} \Big). \end{equation*} 
Since $a$, $b$, and $c$ in $B'_{X,\delta}$ are $O(X^{1/2})$, the error term simplifies to $O_{\epsilon}(X^{7/4 + 50\delta + 12\delta_s})$.

The following lemma allows us to neglect the smoothing effect of the height function induced by $\phi$.
 
\begin{lemma}\label{ignore-h}
For any $(a, b, c)$ in the set $B'_{X, \delta}$, and for any good $\delta_s$-approximation $\theta$, we have:

    \begin{equation*}
        \frac{1}{|a|^3} \int \int \theta_{(a, b, c)}(R, I) \, dR \, dI 
        =\frac{1}{|a|^3} \int \int 
        \psi_{n} \Big( 
            \frac{at^4}{\lambda}, 
            \frac{bt^2}{\lambda}, 
            \frac{c}{\lambda}, 
            \frac{Rt^6}{\lambda^3}, 
            \frac{I}{\lambda^2} 
        \Big) 
        \h_X'(a, b, c, R, I) \, dR \, dI 
       +O(X^{1/2 - 2\delta}),
    \end{equation*}
 where $\h_X'$ is defined as:
   \begin{equation}
       \h_X'(a,b,c,R,I) \vcentcolon=  \chi_{(-1,1)}\Big(\frac{R}{R_0(a,b,c)}\Big) \chi_{(-1,1)}\Big(\frac{3HI}{4aCX}-X^{-1}\Gamma(a,b,c,R)\Big).
   \end{equation} 
Here, $R_0(a,b,c)$ and $\Gamma(a,b,c,R)$ are defined similarly as before.  
\end{lemma}

\begin{proof}
Since the function $\psi_{n}$ has bounded support, it suffices to bound:
\begin{equation*} \frac{1}{|a|^3} \int \int \Big| \phi \Big( \frac{R}{R_0} \Big) \phi \Big( \frac{3HI}{4aCX} - X^{-1}\Gamma(a, b, c, R) \Big) - \chi_{(-1,1)} \Big( \frac{R}{R_0} \Big) \chi_{(-1,1)} \left( \frac{3HI}{4aCX} - X^{-1}\Gamma(a, b, c, R) \right)\Big| \, dR \, dI. \end{equation*}
By applying a change of variables, this reduces to bounding:
\begin{equation*} \Big| \frac{R_0 X}{a^2 H} \Big| \int \int \Big| \phi(R) \phi(I) - \chi_{(-1,1)}(R) \chi_{(-1,1)}(I) \Big| \, dR \, dI. \end{equation*}
This inequality follows from the fact that $\phi$ is a good $X^{-\delta_s}$-approximation of $\chi_{(-1,1)}$, satisfying  
\begin{equation*}  
\int \Big| \phi(x) - \chi_{(-1,1)}(x) \Big| \, dx = O(X^{-\delta_s}).  
\end{equation*}  
Moreover, we can bound $R_0$ by $X^{1/2+\delta_a} |a| |H|^{1/2}$ and apply known bounds on $a$ and $H$, since $(a,b,c) \in B'_{X,\delta}$. This yields the bound  
\begin{equation*}  
O(X^{1/2+2\delta_a+\delta_H/2-\delta_s}),  
\end{equation*}  
which completes the proof.
\end{proof}

We apply the Mean Value Theorem to justify approximating the summation over $B'_{X,\delta}$ by an integral.

\begin{lemma}\label{help fglueing lemma} 
Define $\SV'$ in the semi-invariant space by:
\begin{equation*}
\SV' \vcentcolon= \Big\{(a,b,c,R,I) : (\lfloor a \rfloor, \lfloor b \rfloor, \lfloor c \rfloor) \in B'_{X,\delta} \Big\}. 
\end{equation*} 
For any $(\lambda, t)$ and a given semi-form $f_s = (a, b, c, R, I) \in V_{\R}^s$, define the action $(\lambda, t) \cdot f_s$ by 
\begin{equation*} 
(\lambda, t) \cdot f_s \vcentcolon= \Big(\frac{at^4}{\lambda}, \frac{bt^2}{\lambda}, \frac{c}{\lambda}, \frac{Rt^6}{\lambda^3}, \frac{I}{\lambda^2}\Big).
\end{equation*}
Then we have
\begin{align*} 
&\sum_{(a,b,c) \in B'_{X,\delta}} \frac{1}{|8a^2 \cdot 12a|} \int_R \int_I \psi_{n} \left( \frac{at^4}{\lambda}, \frac{bt^2}{\lambda}, \frac{c}{\lambda}, \frac{Rt^6}{\lambda^3}, \frac{I}{\lambda^2} \right) \h_X'(a,b,c,R,I) \, dR \, dI 
= \\
&\quad\int_{f_s \in \SV'} \frac{1}{|8a^2 \cdot 12a|} , \psi_{n}\left((\lambda, t) \cdot f_s\right) \h_X'(f_s) \, df_s +
O(X^{3/2 +200\delta}) \end{align*} 
where $\h_X'(f_s)$ is defined as before.
\end{lemma}

\begin{proof}  Define the function $g(a, b, c)$ as  
\begin{equation*}
    g(a, b, c) \vcentcolon= \frac{R_0(a, b, c) X}{a^2 H},
\end{equation*}  
where $R_0(a,b,c)$ is given in equations \eqref{R0Hp} and \eqref{RoHn}.  
Next, define the function $F(a, b, c, R, I)$ as
\begin{equation*}
    F(a, b, c, R, I) \vcentcolon= \psi_{n} \Big( \frac{at^4}{\lambda}, \frac{bt^2}{\lambda}, \frac{c}{\lambda}, \frac{RR_0 t^6}{\lambda^3}, \frac{\frac{4aCXI}{3H} + \frac{4aC\Gamma}{3H}}{\lambda^2} \Big),
\end{equation*}
where $\Gamma$ is $\Gamma(a,b,c,R)$, as defined previously.

Then, after a suitable change of variables, the difference of the integrals in the lemma simplifies to
\begin{align*} 
\int_{I} \int_{R}  \chi_{(-1,1)}(R) \chi_{(-1,1)}(I) \Big( \sum_{(a,b,c)\in B'_{X,\delta}} \int_{a}^{a+1} \int_{b}^{b+1} \int_{c}^{c+1} & \Big( g(a,b,c) F(a, b, c, R, I) \\
&\quad - g(a',b',c') F(a', b', c', R, I) \Big) \, d\epsilon_a \, d\epsilon_b \, d\epsilon_c \Big)\, dR \, dI,
\end{align*}
where we define $a' = a + \epsilon_a$, $b' = b + \epsilon_b$, and $c' = c + \epsilon_c$ for small $\epsilon_a$, $\epsilon_b$, and $\epsilon_c$.
\noindent To complete the proof, we need to show that the difference 
\[
\left| F(a, b, c, R, I) g(a, b, c) - F(a', b', c', R, I) g(a', b', c') \right|
\]
is small. Thus, we need to bound both $|g(a, b, c) - g(a', b', c')|$ and $|g(a,b,c)\big(F(a, b, c, R, I) - F(a', b', c', R, I)\big)|$. For clarity, we will explicitly demonstrate the bound for $F(a, b, c, R, I)$, as the argument for $g(a, b, c)$ follows analogously.

Since $g(a, b, c)$ is bounded by $O(X^{1/2 + 2 \delta_a + \delta_H/2})$ for all $(a, b, c) \in B'_{X, \delta}$, it suffices to focus on bounding the difference between the values of $F(a, b, c, R, I)$ and $F(a', b', c', R, I)$, where:
\begin{equation*}
F(a, b, c, R, I) \vcentcolon= \psi_{n} \Big( \frac{a t^4}{\lambda}, \frac{b t^2}{\lambda}, \frac{c}{\lambda}, \frac{R R_0 t^6}{\lambda^3}, \frac{\frac{4aCXI}{3H} +  \frac{27 R_0^2R^2}{48 a^2 H} + \frac{H^2}{48 a^2}}{\lambda^2}\Big),
\end{equation*}
and
\begin{equation*}
F(a', b', c', R, I) \vcentcolon= \psi_{n} \Big( \frac{a' t^4}{\lambda}, \frac{b' t^2}{\lambda}, \frac{c'}{\lambda}, \frac{R R_0' t^6}{\lambda^3}, \frac{R R_0 t^6}{\lambda^3}, \frac{\frac{4a'CXI}{3H'} +  \frac{27 R_0'^2R^2}{48 (a')^2 H'} + \frac{(H')^2}{48 (a')^2}}{\lambda^2} \Big).
\end{equation*}
By the Mean Value Theorem, it suffices to show that each component of $\psi_{n}$ remains close when evaluated at $(a, b, c)$ and $(a', b', c')$, as $\psi_{n}$ is both bounded and smooth. This requires demonstrating that differences such as

\begin{equation*}
 \lambda^{-1}|a t^4 - a' t^4|, \quad \lambda^{-1}| b t^2 -b' t^2|, \quad \lambda^{-1}|c-c'|, \quad  \lambda^{-3}t^6|R_0-R_0' |,
\end{equation*}
and 
\begin{equation*}
\Big| \frac{ a }{H } - \frac{a' }{H' } \Big|, \quad 
\lambda^{-2}\left|\Big(\frac{27 R_0^2R^2}{48 a^2 H} + \frac{H^2}{48 a^2}\Big)-\Big(\frac{27 R_0'^2R^2}{48 (a')^2 H'} + \frac{(H')^2}{48 (a')^2}\Big) \right|,
\end{equation*}
are all bounded by $O(X^{-1/2 + \delta'})$ for a sufficiently small $\delta' > 0$. We analyze each term separately:

\begin{itemize}
    \item Bounding $\lambda^{-1}|a t^4 - a' t^4|$:
    \begin{equation*}
    \Big| \frac{a t^4}{\lambda} - \frac{a' t^4}{\lambda} \Big| = \frac{|\epsilon_a| t^4}{\lambda} \ll X^{-1/2 + 4\delta_t}.
    \end{equation*}
Similar argument applies to $\lambda^{-1}| b t^2 -b' t^2|$ and $\lambda^{-1}|c-c'|$, yielding $O(X^{-1/2 + 4\delta_t})$ for each.

\item Bounding $|R_0-R_0'|t^6\lambda^{-3}$ :
we aim to show that $R_0$ and $R_0'$ are sufficiently close for any $(a, b, c) \in B'_{X, \delta}$. Given the polynomial representation \( H(x,y,z) = 3y^2 - 8xz \), the Mean Value Theorem implies
\begin{equation*}
|H(a', b', c') - H(a, b, c)| = O(X^{1/2}).
\end{equation*}
Since $|H(a, b, c)|$ is greater than $X^{1 - \delta_H}$, this ensures that $H$ and $H'$ maintain the same sign.

To make the analysis concrete, we first consider the case where \( H(a, b, c) < 0 \); the case for \( H(a, b, c) > 0 \) will follow similarly. Based on the definition of \( R_0 \), we need to show that the difference
\begin{equation*}
\Big| |a| |H|^{1/2} \sqrt{X + \frac{H^2}{48a^2}} - |a'| |H'|^{1/2} \sqrt{X + \frac{(H')^2}{48(a')^2}} \Big|
\end{equation*}
is sufficiently small. Since we have upper and lower bounds on $a$ and $H$, we can decompose this expression into three independent terms and bound each difference separately.

Firstly, we consider the term
\begin{equation*}
\Big| |a| |H|^{1/2} \sqrt{X + \frac{H^2}{48a^2}} - |a'| |H|^{1/2} \sqrt{X + \frac{H^2}{48a^2}} \Big|.
\end{equation*}
Here, the difference is $O(X^{1/2} \cdot X^{1/2 + \delta_a})$ because both $H$ and $a$ are bounded above and below, with $|a - a'| = O(1)$ by assumption. This term contributes an error of order $O(X^{1 + \delta_a})$.

Secondly, we examine
\begin{equation*}
\Big| |a'| |H|^{1/2} \sqrt{X + \frac{H^2}{48a^2}} - |a'| |H'|^{1/2} \sqrt{X + \frac{H^2}{48a^2}} \Big|.
\end{equation*}
This term is $O(X^{1 + \delta_H/2 + 2\delta_a})$ because $H$ and $H'$ have the same sign, and we have the constraints $|H - H'| = O(X^{1/2})$ and $|H| > X^{1 - \delta_H}$. This bound follows directly from these constraints.

Finally, we consider the term
\begin{equation*}
\Big| |a'| |H'|^{1/2} \sqrt{X + \frac{H^2}{48a^2}} - |a'| |H'|^{1/2} \sqrt{X + \frac{(H')^2}{48(a')^2}} \Big|.
\end{equation*}
Our goal is to show that this difference is also small. Using the Mean Value Theorem, we know that $|(H/a)^2 - (H'/a')^2| = O(X^{1/2 + 3\delta_a})$, which implies that this term is bounded by $O(X^{1 + 3\delta_a})$.

In summary, each of the three terms is bounded, resulting in a total error of $O(X^{-1/2 + 3\delta_a+\delta_H+6\delta_t})$ for $|(R_0-R_0')t^6\lambda^{-3}|$.
\item Bounding $ \left| \frac{a}{H} - \frac{a'}{H'} \right|$: Since the difference between $H$ and $H'$ is $O(X^{1/2})$ and $H$ is bounded from below, we obtain the bound $O(X^{-1+2\delta_H})$ for this term.
    
\item Bounding $ \left| \frac{27 R_0^2}{48 a^2 H} - \frac{27 (R_0')^2}{48 (a')^2 H'} \right|$: We use the definition along with the fact that $a'$ and $H'$ have the same sign as $a$ and $H$. This reduces to bounding $ \Big| \Big(\frac{H}{a}\Big)^2 - \Big(\frac{H'}{a'}\Big)^2 \Big|$, which we have shown to be as small as $O(X^{1/2+3\delta_a})$. Consequently:
    \begin{equation*}
        \lambda^{-2} \Big| \Big( \frac{27 R_0^2 R^2}{48 a^2 H} + \frac{H^2}{48 a^2} \Big) - \Big( \frac{27 (R_0')^2 R^2}{48 (a')^2 H'} + \frac{(H')^2}{48 (a')^2} \Big) \Big| = O(X^{-1/2+3\delta_a}).
    \end{equation*}
\end{itemize}
Adding all the errors will complete the proof.
\end{proof}

\begin{customproof}[Proposition \ref{Gluing-fiber-volumes}] To proceed, we note that Lemma \ref{help fglueing lemma}, together with the preceding result, implies:

\begin{equation*} 
\sum_{(a,b,c) \in B_{X,\delta}'} N_{(a,b,c)}(\theta, X) = \int_{f_s \in \SV'} \frac{1}{|8a^2 \cdot 12a|} \, \psi_{n}\left((\lambda, t) \cdot f_s\right) \h_X'(f_s) \, df_s + O(X^{7/4 + 500\delta}) + O_{\epsilon}(X^{2-2\delta+\epsilon}).
\end{equation*} 
Thus, by Corollary \ref{SV'toSV}, we have:

\begin{equation*} 
\sum_{(a,b,c) \in B_{X,\delta}'} N_{(a,b,c)}(\theta, X) = \int_{f_s \in \SV} \frac{1}{|8a^2 \cdot 12a|} \, \psi_{n}\left((\lambda, t) \cdot f_s\right) \h_X(f_s) \, df_s + O(X^{7/4 + 500\delta}) + O_{\epsilon}(X^{2-2\delta+\epsilon}).
\end{equation*} 
Next, we rewrite the integral over the space $V_{\mathbb{R}}$. The Jacobian for the change of variables induced by the map $\upsilon$ is $8a^2 \cdot 12a$, and $\psi_{n}$ is given by $n \cdot \Psi$. Thus, we can express this integral in the space of quartic forms as follows:

\begin{equation*}
    \int_{f_s \in \SV} \frac{1}{|8a^2 \cdot 12a|} \, \psi_{n}\left((\lambda, t) \cdot f_s\right) \h_X(f_s) \, df_s = \int_{f \in V_{\delta}} ((nt) \cdot \Psi)(X^{-1/2} \cdot f) \cdot \h_X(f) \, df.
\end{equation*}
This completes the proof.
\end{customproof}

Theorem \ref{Goal-fiber-volume} implies the following expression:

\begin{align*}
 N_C(\omega^{(i)}, V_{\Z}^{(i)}; X) &= \frac{1}{n_i \, \Vol(\theta)} \int_{t = \frac{\sqrt[4]{3}}{\sqrt{2}}}^{X^{\delta_t}} \int_{N'(t)} \int_{K} 
 \Big( \int_{f \in V_{\delta}} \Big((nt) \cdot \Psi[\chi^{(i)}]\Big)(X^{-1/2} \cdot f) \cdot \h_{X}(f) \,df \Big) \, \dw \\ 
 &\quad + O(X^{7/4 + 500\delta}) + O_{\epsilon}(X^{2 - 2\delta + \epsilon}).
\end{align*}
Applying Corollary \ref{VtoV_R} and choosing $\delta$ sufficiently small, we rewrite this expression as:

\begin{equation*}
 N_C(\omega^{(i)}, V_{\Z}^{(i)}; X) = \frac{1}{n_i \, \Vol(\theta)} \int_{\gamma \in \mathcal{F}_{\PGL_2}} \Bigg( \int_{f \in V_{\R}} (\gamma \cdot \Psi[\chi^{(i)}])(X^{-1/2} \cdot f) \cdot \h_{X}(f) \, df \Bigg) \, d\gamma + O(X^{2 - 1/2000}).
\end{equation*}
At this stage, Corollary 5.3 from \cite{moment} and the computed constant from \cite{main} allow us to express the integral in terms of $\omega^{(i)}$ as follows:

\begin{equation*}
\frac{\Vol(\mathcal{F}_{\PGL_2})}{27n_i} \int_{I} \int_{J} \omega^{(i)}(X^{-1}I, X^{-3/2}J) \cdot \chi_{(-1,1)}\left(\frac{I}{X}\right) \cdot \chi_{(-1,1)}\left(\frac{J}{X}\right) \,dI \,dJ.
\end{equation*}
Therefore, to complete the proof of Theorem \ref{The-Main-counting-Theo}, it remains to evaluate this integral and approximate $\mathcal{J}^{(i)}$ in the limit. This will be carried out in the next section.
%%%%%%%%%%%%%%%%%%%%%%%%%%%%%%%%%%%%%%%%%%%%%%%%%%%%%%%%%%%%%%%%%%%%%%%%%%%%%%%%%%%%%%%%%%%%%%%%%%%%%%%%%%%%%%%%%%%%%%%%%%%%%%%%%%%%%%%%%%%%%%%%%%%
\subsection{Computation of the volume}
In the previous sections, we proved the following:

\begin{equation*}
N_C(\omega^{(i)}, V_{\Z}^{(i)}; X) = \frac{\Vol(\mathcal{F}_{\PGL_2})}{27n_i} \int_{I} \int_{J} \omega^{(i)}(X^{-1}I, X^{-3/2}J) \cdot \chi_{(-1,1)}\left(\frac{I}{X}\right) \cdot \chi_{(-1,1)}\left(\frac{J}{CX}\right) \, dI \, dJ.
\end{equation*}
Now, let $\omega^{(i)}_{\epsilon}$ be a smooth approximation of $\chi_{\mathcal{J}^{(i)}}$, the characteristic function of $\mathcal{J}^{(i)}$, such that for any fixed $J$:

\begin{equation*}
    \int_I \Big|\omega^{(i)}(I,J) - \chi_{\mathcal{J}^{(i)}}(I,J)\Big| \, dI = O(\epsilon).
\end{equation*}
Thus, we have:

\begin{align*}
    X^2 \int_{I} \int_{J} \omega^{(i)}(I, X^{-1/2}J) \cdot \chi_{(-1,1)}(I) \cdot \chi_{(-1,1)}(C^{-1}J) \, dI \, dJ &= X^{5/2} \int_{\mathcal{J}^{(i)}} \chi_{(-1,1)}(I) \cdot \chi_{(-1,1)}(C^{-1}X^{1/2}J) \, dI \, dJ \\
    &\quad + O(X^2 \epsilon).
\end{align*}
For calculating the volume with $i = 0, 2+, $ and $2-$, we find:

\begin{equation*}
    X^{5/2} \int_{I=0}^{1} \int_{-2I^{3/2}}^{2I^{3/2}} \chi_{(-1,1)}(C^{-1}X^{1/2}J) \, dJ \, dI = 2CX^2 + O(X^{2 - 1/3}).
\end{equation*}
This holds because, if $2I^{3/2} > CX^{-1/2}$, the term $\chi_{(-1,1)}(C^{-1}X^{1/2}J)$ controls the interval of $J$. Conversely, if $2I^{3/2} \leq CX^{-1/2}$, then $I$ itself is $O(X^{-1/3})$.

Similarly, for $i = 1$, we obtain:

\begin{equation*}
  X^{5/2} \int_{\mathcal{J}^{(1)}} \chi_{(-1,1)}(I) \cdot \chi_{(-1,1)}(C^{-1}X^{1/2}J) \, dI \, dJ = 2CX^2 + O(X^{2 - 1/3}).
\end{equation*}
Combining these results with the fact that $\Vol(\mathcal{F}_{\PGL_2}) = \Vol(\PGL_2(\mathbb{Z}) / \PGL_2(\mathbb{R})) = 2\zeta(2)$, we conclude that:

\begin{equation*}
N_C(\omega^{(i)}, V_{\Z}^{(i)}; X) = \frac{4C \zeta(2)X^2}{n_i} + O(X^{2 - 1/2000}) + O(\epsilon X^2).
\end{equation*}
To complete the proof of Theorem $\ref{binary quartic average}$, we first let $X \to \infty$, and then let $\epsilon \to 0$.
%%%%%%%%%%%%%%%%%%%%%%%%%%%%%%%%%%%%%%%%%%%%%%%%%%%%%%%%%%%%%%%%%%%%%%%%%%%%%%%%%%%%%%%%%%%%%%%%%%%%%%%%%%%%%%%%%%%%%%%%%%%%%%%%%%%%%%%%%%%%%%%%%%%

\subsection{Congruence conditions}
In this subsection, our goal is to establish a version of Theorem \ref{The-Main-counting-Theo} under the additional assumption that we count integral quartic forms subject to finitely many congruence conditions.

Let $S$ denote a subset of $V_{\mathbb{Z}}$ defined by finitely many congruence conditions. We can assume that $S \subset V_{\mathbb{Z}}$ is defined by congruence conditions modulo some integer $m$. Consequently, $S$ can be viewed as the union of $k$ translates $\mathcal{L}_1, \ldots, \mathcal{L}_k$ of the lattice $m \cdot V_{\mathbb{Z}}$.

Our aim is to compute $N(V_{\mathbb{Z}}^{(i)} \cap \mathcal{L}_i; X)$ using a similar method. To proceed, we need to understand the image of $\mathcal{L}_i$ under the map $\upsilon$, which will enable us to apply twisted Poisson summation.

\begin{lemma}
Let $m$ be a fixed integer, and let $\mathcal{L}$ be defined by translating $m \cdot V_{\mathbb{Z}}$ by $f_0 = (a_0, b_0, c_0, d_0, e_0)$. The image of $\mathcal{L}$ under $\upsilon$, denoted $\mathcal{L}_s$, is characterized by:
\begin{equation*}
\mathcal{L}_s \vcentcolon= \left\{(a,b,c,R,I):
\begin{aligned}
&(a,b,c) \equiv (a_0,b_0,c_0) \pmod m \\
&(R,I) \in \Lambda_{(a,b,c)}^{\mathcal{L}}
\end{aligned} \right\}
\end{equation*}
Where $\Lambda_{(a,b,c)}^{\mathcal{L}}$ is the union of sets $\Lambda_{(a,b,c),k}^{\mathcal{L}}$ for any $k = 0, \ldots, 12|a| - 1$ and
\begin{equation*}
\Lambda_{(a,b,c),k}^{\mathcal{L}} = \left\{ (R,I):
\begin{aligned}
& R\equiv d_0'+8a^2\cdot m \cdot k &\pmod{8a^2\cdot 12a \cdot m}\\
&I \equiv e_0'-3bmk &\pmod {12a\cdot m}
\end{aligned}\right\}
\end{equation*}
Where $d_0'=8a^2d_0+b^3-4abc$ and $e_0'=12ae_0-3bd_0+c^2$ are constants independent of $k$.
\end{lemma}
\begin{proof}
Fix $d$ modulo $12a \cdot m$. Since $d$ is determined modulo $m$, there exists a unique $k \in \{0, \ldots, 12|a| - 1\}$ such that
\begin{equation*}
d \equiv d_0+ mk \pmod{12a \cdot m}.
\end{equation*}
Then fixing $(a,b,c)$ and $d$ mod $12a\cdot m$ will fix $R$ modulo $8a^2\cdot 12 a \cdot m$. By putting $d_0'=b^3-4abc+8a^2d_0$, then  
\begin{equation*}
R \equiv d_0'+8a^2\cdot m\cdot k   \pmod{8a^2\cdot 12a \cdot m}.
\end{equation*}
Similarly, $I$ remains fixed modulo $12a \cdot m$ given that $d$ is fixed modulo $12a\cdot m$. Consequently, this implies that the image of $\mathcal{L}$ under the map $\upsilon$ lies within $\Lambda_{(a,b,c)}^{\mathcal{L}}$.
The converse situation is similar. 
\end{proof}

If we repeat the counting process here, but apply the twisted Poisson summation to $\Lambda_{(a,b,c)}^{\mathcal{L}}$ rather than $\Lambda_{(a,b,c)}$, then we need to compute the finite part of the twisted Poisson summation. This is handled in the following lemma.

\begin{lemma}
Consider $\mathcal{L}$ and $\Lambda_{(a,b,c)}^{\mathcal{L}}$ as in the previous lemma. Then:
\begin{equation*}
|\widehat{\psi_{\Lambda_{(a,b,c)}^{\mathcal{L}}}}(\alpha,\beta)| = \frac{\epsilon_{(\alpha,\beta)}}{m^2 \cdot |8a^2 \cdot 12a|},
\end{equation*}
where $
\psi_{\Lambda_{(a,b,c)}^{\mathcal{L}}}$ the characteristic function of $\Lambda_{(a,b,c)}^{\mathcal{L}}$ in $\R^2$.
\end{lemma}
\begin{proof}
By adapting the proof of Lemma \ref{Poisson} to this context, it suffices to understand the following:

\begin{equation*}
    m^2 \cdot 8a^2 \cdot 12a \cdot 12a \cdot \big|\widehat{\psi_{\Lambda_{(a,b,c)}^{\mathcal{L}}}}(\alpha,\beta)\big| = \Big|\sum_{k=0}^{12|a|-1} e^{2\pi i \alpha \xi'_{\alpha,k}(a, b, c)} e^{2\pi i \beta \xi'_{\beta,k}(a, b, c)}\Big|,
\end{equation*}
where $\xi'_{\alpha,k}(a, b, c)$ and $\xi'_{\beta,k}(a, b, c)$ are defined as:

\begin{equation*}
    \xi_{\alpha,k}(a, b, c) \vcentcolon= \frac{d_0' + 8a^2 \cdot m \cdot k}{8a^2 \cdot 12a \cdot m}, \quad \xi_{\beta,k}(a, b, c) \vcentcolon= \frac{e_0' - 3bmk}{12am}.
\end{equation*}
Here, $d_0'$ and $e_0'$ are defined as before and are independent of $k$. Then, we have:

\begin{equation*}
    \Big|\sum_{k=0}^{12|a|-1} e^{2\pi i \left( \frac{k\alpha}{12a} \right)} e^{2\pi i \left( \frac{-3bk\beta}{12a} \right)}\Big| = \Big|\sum_{k=0}^{12|a|-1} e^{2\pi i k \left( \frac{\alpha - 3b\beta}{12a} \right)}\Big|.
\end{equation*}
This completes the proof.
\end{proof}

Applying the same method, we obtain the following theorem.

\begin{theorem}
Suppose $S$ is a subset of $V_{\mathbb{Z}}$ defined by congruence conditions modulo of finitely many prime powers. Then we have the following result:
\begin{equation*}
N(S \cap V_{\Z}^{(i)};X) = N(V_{\Z}^{(i)};X) \prod \mu_p(S) + o(X^2)
\end{equation*}
Where $\mu_p(S)$ denotes the $p$-adic density of $S$ in $V_{\Z}$.
\end{theorem}
We also need to consider counting orbits with weights. The proof of the following theorem proceeds using the same approach as described above.

\begin{theorem}\label{conqgruent-counting}
    Let $p_1, \cdots, p_k$ be distinct prime numbers. For $j=1,\cdots,k$, let $\phi_{p_j}:V_{\mathbb{Z}} \rightarrow \mathbb{R}$ be a $\GL_2(\mathbb{Z})$-invariant function on $V_{\mathbb{Z}}$ such that $\phi_{p_j}(f)$ only depends on the class of $f$ modulo some powers $p_j^{a_j}$ of $p_j$. Let $N_{\phi}(V_{\mathbb{Z}}^{(i)};\phi)$ denote the number of irreducible $\GL_2(\mathbb{Z})$-orbits in $V_{\mathbb{Z}}^{(i)}$ with height bounded by $X$, where each orbit $\GL_2(\mathbb{Z}) \cdot f$ is counted with weight $\phi(f) \coloneqq \prod_{j=1}^{k}\phi_{p_j}(f)$. Then we have
    \begin{equation*}
        N_{\phi}(V_{\mathbb{Z}}^{(i)};X)=N(V_{\mathbb{Z}}^{(i)};X)\prod_{j=1}^{k}\int_{f\in V_{\mathbb{Z}_{p_j}}} \tilde{\phi}_{p_j}(f) df + o(X^2)
    \end{equation*}
    where $\tilde{\phi}_{p_j}$ is the natural extension of $\phi_{p_j}$ to $V_{\mathbb{Z}_p}$, $df$ denotes the additive measure on $V_{\mathbb{Z}_{p}}$, normalized so that $\int_{f\in V_{\mathbb{Z}_p}}df=1$.
\end{theorem}

For our purposes, we require a more general version of our theorem on counting integral orbits of binary quartic forms, one that accounts for infinitely many congruence conditions. We say that a function $\phi: V_{\mathbb{Z}} \to [0,1] \subset \mathbb{R}$ is \textit{defined by congruence conditions} if, for every prime $p$, there exist functions $\phi_p: V_{\mathbb{Z}} \to [0,1]$ such that:

\begin{enumerate}
    \item[(1)] For all $f\in V_{\mathbb{Z}}$, the product $\prod_{p}\phi_p(f)$ converges to $\phi(f)$.
    \item[(2)] For each prime $p$, the function $\phi_p$ is locally constant outside some closed set $S_p \subset V_{\mathbb{Z}_p}$ of measure zero.
\end{enumerate}
This definition is consistent with that in \cite{main}, and we will utilize their results for counting integral quartic forms with weight considerations. A function $\phi$ is said to be \textit{acceptable} if, for sufficiently large primes $p$, we have $\phi_p(f) = 1$ whenever $p^2 \nmid \Delta(f)$. The following theorem establishes a framework for bounding the count with weights defined by infinitely many congruence conditions.

\begin{theorem}\label{weight-counting-quartic}
    Let $\phi \rightarrow [0,1]$ be acceptable function that is defined by congruence conditions via the local functions $\phi_p: V_{Z_{p}}: \rightarrow[0,1]$. then, with notations as in Theorem \ref{conqgruent-counting}, we have 
    \begin{equation*}
        N_{\phi}(V_{\Z}^{(i)};X) \leq N(V_{\Z}^{(i)};X) \prod_{p}\int_{f \in V_{\Z_p}} \phi_p(f)df + o(X^2)
    \end{equation*}
\end{theorem}
\begin{proof}
The proof closely follows the initial part of the proof of Theorem 2.21 in \cite{main}.
\end{proof}

\begin{remark}\label{uniformity tail estimate}
The bound in Theorem \ref{weight-counting-quartic} is expected to be exact, meaning that the inequality $\leq$ should be replaced with equality. To establish this, we require a tail estimate showing that the number of integral binary quartic forms with height less than $X$ and discriminant divisible by $p^2$, where $p > Q$, is $O\left(\frac{X^2}{Q^\delta}\right) + o(X^2)$ for some small positive $\delta$. The main difficulty arises when $p$ is a prime number, as this imposes congruence conditions at size of $X^{3/2}$ on a region of size $X^{1/2}$ within our space, where fixing four variables results in one of the edges being as small as $O(1)$.
\end{remark}

We also show that Theorem~\ref{2-Selmer average} remains valid for broader families of elliptic curves defined by finitely many congruence conditions on the coefficients $A$ and $B$.
%%%%%%%%%%%%%%%%%%%%%%%%%%%%%%%%%%%%%%%%%%%%%%%%%%%%%%%%%%%%%%%%%%%%%%%%%%%%%%%%%%%%%%%%%%%%%%%%%%%%%%%%%%%%%%%%%%%%%%%%%%%%%%%%%%%%%%%%%%%%%%%%%%%%
\section{The average size of 2-Selmer groups of elliptic curves}\label{sec5}

Recall that every elliptic curve $E$ over $\mathbb{Q}$ can be uniquely expressed in the form:
\begin{equation}\label{elliptic-def}
E_{A,B}: y^2 = x^3 + Ax + B
\end{equation}
where $A, B \in \mathbb{Z}$, and  $p^4 \nmid A$ if $p^6 \mid B$. For any elliptic curve $E = E_{A,B}$ over $\mathbb{Q}$ given by equation \eqref{elliptic-def}, we define
\begin{equation*}
\begin{aligned}
&I(E) \vcentcolon=  - 3  A, \\
&J(E) \vcentcolon= -27  B.
\end{aligned}
\end{equation*}
Any curve $E_{A,B}$ can be expressed uniquely as $E^{I,J}$. The height of $E_{A,B}$ is defined as follows:
\begin{equation*}
h_e(E_{A,B}) \vcentcolon= h_e(I,J)  = \max \big\{|I|/3,|J|/27\big \}.
\end{equation*}
We define families of elliptic curves as in \cite{main}. For each prime $p$, let $\Sigma_p$ be a closed subset of $\mathbb{Z}_p^2 \setminus \{\Delta \neq 0\}$ with boundary measure zero. We then define the set $F_{\Sigma}$ of elliptic curves over $\mathbb{Q}$, where $E^{I,J}$ belongs to $F_{\Sigma}$ if and only if $(I, J) \in \Sigma_p$ for all $p$. This characterizes $F_{\Sigma}$ as a family of elliptic curves over $\mathbb{Q}$ \textit{defined by congruence conditions}.

We can also impose "congruence conditions at infinity" on $F_{\Sigma}$ by requiring that an elliptic curve $E^{I,J}$ is included in $F_{\Sigma}$ if and only if $(I, J)$ lies in the set $\Sigma_{\infty}$, where $\Sigma_{\infty}$ is one of the following: 
\begin{equation*}
\{(I, J) \in \mathbb{R}^2 : \Delta(I, J) > 0\}, \quad \{(I, J) \in \mathbb{R}^2 : \Delta(I, J) < 0\}, \quad \text{or} \quad \{(I, J) \in \mathbb{R}^2 : \Delta(I, J) \neq 0\}.
\end{equation*}
Let $F$ be any nonempty family of elliptic curves over $\mathbb{Q}$, defined by congruence conditions. We define $\mathrm{Inv}(F)$ as the set of pairs $(I(E), J(E))$ for all $E \in F$. The set $\mathrm{Inv}_p(F)$ consists of elements $(I, J)$ that lie in the $p$-adic closure of $\mathrm{Inv}(F) \subseteq \mathbb{Z}_p^2$, subject to the condition that $\Delta(I, J) \neq 0$. Similarly, we define $\mathrm{Inv}_{\infty}(F)$ as one of the following sets:
\begin{equation*}
\{(I, J) \in \mathbb{R}^2 : \Delta(I, J) > 0\}, \quad \{(I, J) \in \mathbb{R}^2 : \Delta(I, J) < 0\}, \quad \text{or} \quad \{(I, J) \in \mathbb{R}^2 : \Delta(I, J) \neq 0\},
\end{equation*}
depending on whether $F$ contains elliptic curves with positive discriminant, negative discriminant, or both.

A family $F$ of elliptic curves defined by congruence conditions is said to be \textit{large} if, for all but finitely many primes $p$, the set $\mathrm{Inv}_p(F)$ contains all pairs $(I, J) \in \mathbb{Z}_p \times \mathbb{Z}_p$ such that $p^2 \nmid \Delta(I, J)$.

The family of all elliptic curves, as well as those defined by finitely many congruence conditions on the coefficients of the elliptic curves, are both large families. We can prove the following stronger version of Theorem \ref{2-Selmer average}.

\begin{theorem}\label{2-selmer-bound}
    When all elliptic curves $E$ in any large family are ordered by height $h_e$, the average size of the $2$-Selmer group is at most 3.
\end{theorem}

\subsection{Counting elliptic curves with bounded height in a large family}
In this subsection, our objective is to understand the asymptotic behavior of the count of elliptic curves within a large family $F$, bounded by the height $h$.

For any subset $S \subset \mathbb{Z} \times \mathbb{Z}$, let $N(S; X)$ denote the number of pairs $(I, J) \in S$ such that $4I^3 - J^2 \neq 0$ and the height $h_e$ of the pair is less than $X$. Since an elliptic curve is uniquely determined by its invariants $I$ and $J$, our aim is to study the asymptotic behavior of $N(\text{Inv}(F); X)$ as $X$ tends to infinity. Here, $\text{Inv}(F)$ may be defined by infinitely many congruence conditions. To conduct this analysis, we introduce the following uniformity estimate.

\begin{theorem}\label{IJ-Sieve}
    For a prime $p$, let $\mathcal{W}_{p}$ represent the set of elliptic curves $E$ over $\Q$ such that $p^2$ divides the discriminant of $E$. For any $Q$, we have the following:
    \begin{equation*}
      \#\bigcup_{p>Q} \left\{E_{A,B} \in \mathcal{W}_p:h(E_{A,B})<X\right\} = O_{\epsilon}\left( \frac{X^{2+\epsilon}}{Q}+ X^{7/4+\epsilon}\right).
    \end{equation*}
\end{theorem}
Let $U_0(\Z)$ denote the set of monic cubic polynomials over $\Z$ with a zero $x^2$-coefficient. For any element $g \in U_0(\Z)$, we define $K_g$ as $\Q[x]/(g(x))$, $R_g$ as $\Z[x]/(g(x))$, and $O_g$ as the maximal order in $K_g$. The $Q$-invariant of $g$ is defined as the index of $R_g$ in $O_g$, and $D_g$ represents the discriminant of $K_g$.

Additionally, let $U_0(\mathbb{Z})_{X}^{\min}$ denote the subset of elements $g(x) = x^3 + Ax + B \in U_0(\mathbb{Z})$ satisfying the condition that $p^4 \nmid A$ whenever $p^6 \mid B$, where $|A|$ and $|B|$ are bounded by $X$. The following equations will be used to establish Theorem \ref{IJ-Sieve}, primarily because the discriminant $\Delta(g)$ satisfies the relation $\Delta(g) = D_g Q_g^2$.
\begin{align}
& \# \left\{g \in U_0(\Z)_{X}^{\min}: Q_g>Q\right\}=  O_{\epsilon}\left(\frac{X^{2+\epsilon}}{Q}+X^{7/4+\epsilon}\right) \label{weakly-divisible}\\
& \#\bigcup_{p>Q} \left\{g \in U_0(\Z)_{X}^{\min}: p^2 \mid D_g\right\}=O_{\epsilon}\left(\frac{X^{2}}{Q}+X^{1+\epsilon}\right) \label{strongly-divisible}.
\end{align}
By following the lemma, we can confine the analysis of equations \eqref{weakly-divisible} and \eqref{strongly-divisible} to polynomials in $U_0(\Z)_{X}^{\min}$ that are irreducible over $\Q$.

\begin{lemma}\label{reducible cubic resolvent count}
The number of reducible polynomials over $\mathbb{Q}$ in $U_0(\mathbb{Z})_{X}^{\min}$ is bounded by $O(X^{1+\epsilon})$.
\end{lemma}
\begin{proof}
Consider a polynomial $f \in U_0(\mathbb{Z})_X^{\text{min}}$ of the form $f(x) = x^3 + Ax + B$. Suppose $\frac{a}{b}$ is a root of $f$, where $a$ and $b$ are coprime integers with $b$ positive. In this case, $a$ must divide $B$, and $b$ must equal one. To count the reducible polynomials in $U_0(\mathbb{Z})_X^{\text{min}}$, we begin by fixing $B$, which offers $O(X)$ choices. The possible roots are factors of $B$, and fixing the root uniquely determines the polynomial. This completes the proof of the lemma.
\end{proof}

For any prime $p$ greater than $5$, if $p^2$ divides $D_g$ for an irreducible polynomial $g(x)=x^3+Ax+B$ over $\Q$, then it is equivalent to $p$ dividing both $A$ and $B$. Consequently, the bound in equation \eqref{strongly-divisible} is clear.
Let $U_0(\mathbb{Z})_{X,\text{irr}}^{\text{min}}$ denote the set of elements in $U_0(\mathbb{Z})_{X}^{\text{min}}$ that are additionally irreducible. Lemma $\ref{reducible cubic resolvent count}$ enables us to focus on proving equation $\eqref{weakly-divisible}$ by establishing a similar result exclusively for irreducible cubics within $U_0(\mathbb{Z})_{X, \irr}^{\text{min}}$.

The main challenge in proving Theorem \ref{IJ-Sieve} lies in establishing equation \eqref{weakly-divisible}. To tackle this, we embed $U_0(\mathbb{Z})_{X,\irr}^{\min}$ into the $\text{PGL}_2(\mathbb{Z})$-orbits of integral binary quartic forms that possess a unique root in $\mathbb{P}^1(\mathbb{Z})$. We denote the set of such integral binary quartic forms by $\mathcal{L}$.

This embedding follows a similar approach to that in \cite{consuctor}, with the key distinction that we focus exclusively on quartic forms with a unique root in $\mathbb{P}^1(\mathbb{Z})$, eliminating the need to retain information about the root itself.

Consider the polynomial $f(x) = x^3 + Ax + B$ in $U_0(\mathbb{Z})_{X, \irr}^{\min}$ with $Q_f = n$. According to Theorem 1.6 in \cite{consuctor}, there exists a unique integer $r$ modulo $n$ such that $f(x + r)$ can be written as

\begin{equation*}
f(x + r) = x^3 + ax^2 + bnx + cn^2.
\end{equation*}
The corresponding cubic ring $R_f$ is then contained within the cubic ring associated with the binary cubic form
\begin{equation*}
h(x, y) = nx^3 + ax^2y + bxy^2 + cy^3,
\end{equation*}
via the Delone–Faddeev correspondence, with index n. This induces the following map:
\begin{equation*}
\tilde{\sigma}: U_0(\mathbb{Z})_{X,\irr}^{\min} \rightarrow \mathcal{L},
\quad f(x) \mapsto yh(x, y).
\end{equation*}
We define the $Q$-invariant for an integral quartic form $f$ with a unique root $(\alpha, \beta)$ as $h(\alpha, \beta)$, where $h(x, y)$ is defined as $f(x, y)/(\beta x - \alpha y)$. Now, consider the following map:

\begin{equation*}
\sigma: U_0(\mathbb{Z})_{X,\irr}^{\min} \rightarrow \mathcal{L} \rightarrow \text{PGL}_2(\mathbb{Z}) \backslash \mathcal{L}.
\end{equation*}
Proposition 5.3 in \cite{consuctor} establishes that this is indeed an embedding, and we have the following identities:
\begin{equation*}
I(\sigma(f)) = -3A, \quad J(\sigma(f)) = -27B, \quad Q(f) = Q(\sigma(f)).
\end{equation*}
Thus, our objective is to bound the number of $\text{PGL}_2(\mathbb{Z})$-orbits of integral quartic forms with height less than $X$ and $Q$-invariant greater than $Q$, which possess a unique root in $\mathbb{P}^1(\mathbb{Z})$. For convenience, we assume, by a slight abuse of notation, that $\mathcal{L}$ consists only of elements for which the polynomial $X^3 - 3I(f) + J(f)$ is irreducible over $\Q$. This additional assumption is justified, as we are ultimately concerned with counting the set $U_0(\mathbb{Z})_{X,\irr}^{\min}$.

To achieve this, we employ the averaging method to count such orbits. However, as in the previous case, the height restriction confines us to a thin region. Consequently, counting via Davenport’s lemma does not yield precise results. The following lemma establishes a connection between the unique root of $f \in \mathcal{L}$ and its semi-invariants. This connection allows us to count these orbits by fibering over each root after applying the averaging method.

\begin{lemma}\label{J-single root lemma}
Consider $f = (a, b, c, d, e) \in \mathcal{L}$, and let its unique root be denoted as $(u', 4a) \in \mathbb{P}^1(\mathbb{Z})$. Assume that $H$ and $R$ are semi-invariants of $f$, and that both $R \neq 0$ and $a \neq 0$. Define $u = -u' - b$, and let $v$ satisfy $2v = u^2 + H$. Then, the invariant $J$ is given by:
\begin{equation}\label{J-single root}
(4a)^3J = 2H^3-9Hv^2+3(uH)^2-3(3R-uH)^2
\end{equation}
\end{lemma}
\begin{proof}
The proof of this lemma follows from the proof of Proposition 8 in \cite{cremona1}, along with the explicit computation of the relevant coefficients. Let $g(x) = x^3 - 3I(f)x + J(f)$ denote the cubic resolvent of $f$. Consider the field $K_g$, and let $\varphi$ be an element of $K_g$ satisfying $g(\varphi) = 0$. From equation \eqref{syzygy}, it follows that the characteristic polynomial of  
$z_g = \frac{4a\varphi - H}{3}$  
is given by  
\begin{equation*}
m(Z) = \Big(\frac{4a}{3}\Big)^3 g\Big(\frac{3Z+H}{4a}\Big) = Z^3 + HZ^2 + SZ - R^2,
\end{equation*}  
where $R$, $H$, and $S$ are semi-invariants of $f$. As stated in Lemma 3 of \cite{cremona1}, a quartic form has a linear factor over $\mathbb{Q}$ if and only if there exists an element $z \in K_g$ such that $z_g = z^2$. Consider the characteristic polynomial of $z$ (replacing $z$ with $-z$ if necessary)
\begin{equation*}
p(Z) = Z^3 + uZ^2 + vZ + R.
\end{equation*}  
Furthermore, $m(Z^2)$ is equal to $-p(Z)p(-Z)$. Comparing coefficients shows that $f$ has a root at $(-u - b, 4a)$. Since there is only one unique root, another coefficient comparison implies the semi-invariant relations.  

The final step involves calculating $m\big(-\frac{H}{3}\big)$, which equals $\big(\frac{4a}{3}\big)^3 J(f)$. Using the relations of $u$ and $v$ with the semi-invariants, we obtain the desired relation.
\end{proof}

\begin{lemma}\label{nonzero-a-bounding}
    Let $B$ be any bounded open set in $\mathbb{R}^5$, and let $t$ be a positive number. Then, for $\lambda = X^{1/2}$, we have:
    \begin{equation*}
      \# \big\{ f \in \mathcal{L} \cap (\lambda,t) \cdot B : a(f) \neq 0, \,|J(f)| \ll X \big\} = O_{\epsilon}(X^{7/4+\epsilon}).
    \end{equation*}
\end{lemma}
\begin{proof}
    We have the following bounds as $B$ is bounded:
\begin{equation*}  
a \ll \frac{X^{1/2}}{t^4}, \quad b \ll \frac{X^{1/2}}{t^2}, \quad c \ll X^{1/2}, \quad d \ll t^2 X^{1/2}, \quad e \ll t^4 X^{1/2}.
\end{equation*}
We begin by counting integral quartic forms where both $a$ and $R$ are nonzero. It can always be shown that a root of $f$ in $\mathbb{P}^{1}(\mathbb{Z})$ can be uniquely represented as $(u', 4a)$. By symmetry, we focus on counting integral quartic forms with the root $(u', 4a)$ such that $|u'| 
\leq |4a|$. 

As previously discussed, we consider the fiber over the root and the coefficients of $x^4$. This approach is advantageous as it allows us to incorporate the restrictions imposed by $J$ on each fiber by applying Lemma \ref{J-single root lemma}. Thus, it remains only to estimate the following:

\begin{equation}\label{}
\sum_{0<|a|\ll X^{1/2}} \sum_{\delta \mid 4a}\sum_{\substack{ |u'| \leq |4a|\\
\gcd(u',4a)=\delta }}
 \#\Big\{f=(a,b,c,d,e) \in \mathcal{L}\cap (\lambda,t)\cdot B : |J(f)|\ll X ,\, f(u',4a)=0 \}.
\end{equation}
For any pair $(u', 4a)$ where $\gcd(u', 4a) = \delta$, we define $\alpha = \frac{u'}{\delta}$ and $\beta = \frac{4a}{\delta}$, ensuring that $|\alpha|$ and $|\beta|$ are coprime. Given that the root $(\alpha, \beta)$ and the coefficient $a$ are fixed, we can factor the polynomial $f$ as follows:
\begin{equation}\label{factoring by root}
    \begin{split}
         f(x, y) &= (\beta x-\alpha y)(a_1x^3+a_2x^2y+a_3xy^2+a_4x^3)\\&= (\beta a_1) x^4 + (\beta a_2 - \alpha a_1) x^3 y + (\beta a_3 - \alpha a_2) x^2 y^2 + (\beta a_4 - \alpha a_3) x y^3 + (-\alpha a_4) y^4.
    \end{split}
\end{equation}
Upon fixing $a$, the coefficient $a_1$ is determined as $a_1 = \frac{a}{\beta}$, and this also fixes $b$ modulo $\beta$ as $ -\alpha a_1$. Subsequently, when $b$ is fixed, the coefficient $a_2$ is established using the relation $a_2 = \frac{b + \alpha a_1}{\beta}$. This setup directly affects $c$, fixing it modulo $\beta$ as $c \equiv -\alpha a_2 \pmod \beta$. This implies that:
\begin{align*}
& \sum_{0 < |a| \ll X^{1/2}} \sum_{\delta \mid 4a} \sum_{\substack{|u'| \leq |4a| \\ \gcd(u', 4a) = \delta}} \# \Big\{ f = (a, b, c, d, e) \in \mathcal{L} \cap (\lambda,t)\cdot B: |J(f)| \ll X, f(u', 4a) = 0 \}\\ \ll \quad
& \sum_{0 < |a| \ll X^{1/2}} \sum_{\delta \mid 4a} \sum_{\substack{|u'| \leq |4a| \\ \gcd(u', 4a) = \delta}} \sum_{\substack{|b| \ll X^{1/2} \\ b \equiv -\alpha a_1 \\ \pmod{\beta}}} \sum_{\substack{|c| \ll X^{1/2} \\ c \equiv -\alpha a_2 \\ \pmod{\beta}}} \# \Big\{ f \in \big(\mathcal{L}_{(\alpha,\beta)} \cap (\lambda,t)\cdot B\big)_{(a, b, c)} :  |J(f)| \ll X\},
\end{align*}
where $\mathcal{L}_{(\alpha, \beta)}$ denote the subset of quartic forms within $\mathcal{L}$ that uniquely possess the root $(\alpha, \beta)$. For any subset $W$ of $V_{\mathbb{R}}$, we define $W_{(a, b, c)}$ as the fiber over the point $(a, b, c)$, effectively isolating those elements of $W$ that correspond precisely to the fixed coefficients $a$, $b$, and $c$. 

Now, assuming we have fixed $a$, $b$, $c$, and the root $(u', 4a)$, this will also fix the semi-invariant $H$. Thus, we need to bound the number of integral quartic forms $f = (a, b, c, -, -) \in \mathcal{L} \cap (\lambda,t)\cdot B $ with $f(u', 4a) = 0$ and $J$-invariant less than $X$. By applying Lemma \ref{J-single root lemma}, the $R$-semi-invariant of such a form is bounded in at most two intervals of the size $O(|a|^{3/2}X^{1/2})$. This is true since fixing the root and $(a, b, c)$ will also fix $u$ and $v$ as defined in the lemma. This, with the fact that we are counting quartic forms with $J$-invariant less than $X$, implies that $(3R-uH)^2$ is within an interval of the size $O(|a|^3X)$. Again, as $uH$ is fixed, the possible values for $R$ are within two intervals of the size $O(|a|^{3/2}X^{1/2})$. Note that $R = b^3 - 4abc + 8a^2d$, and fixing $a$, $b$, $c$, and the root $(\alpha, \beta)$ additionally fixes $d$ modulo $\beta$ by looking at equation \ref{factoring by root}. Hence, our assumptions will restrict $R$ to be fixed modulo $8a^2 \cdot \beta$. By Fixing $a$, $b$, $c$, $R$, and the unique root will completely determine a form in $\mathcal{L}$. Thus, we have:

\begin{align*}
& \sum_{0 < |a| \ll X^{1/2}} \sum_{\delta \mid 4a} \sum_{\substack{|u'| \leq |4a| \\ \gcd(u', 4a) = \delta}} \sum_{\substack{|b| \ll X^{1/2} \\ b \equiv -\alpha a_1 \\ \pmod{\beta}}} \sum_{\substack{|c| \ll X^{1/2} \\ c \equiv -\alpha a_2 \\ \pmod{\beta}}} \# \Big\{ f \in \big(\mathcal{L}_{(\alpha,\beta)} \cap (\lambda,t)\cdot B)_{(a,b,c)} :  |J(f)| \ll X\} \\ \ll \quad
&
\sum_{0 < |a| \ll X^{1/2}} \sum_{\substack{\delta \mid 4a}} \sum_{\substack{|u'| \leq |4a| \\ \gcd(u', 4a) = \delta}} \sum_{\substack{|b| \ll X^{1/2} \\ b \equiv -\alpha a_1 \\ \pmod{\beta}}} \sum_{\substack{|c| \ll X^{1/2} \\ c \equiv -\alpha a_2 \\ \pmod{\beta}}}\left( \frac{|a|^{3/2} X^{1/2}}{a^2 \cdot |4a/\delta|} + 1 \right) \\ \ll  \quad
&\sum_{0 < |a| \ll X^{1/2}} \sum_{\substack{\delta \mid 4a}} \sum_{\substack{|u'| \leq |4a| \\ \gcd(u', 4a) = \delta\\
}}  \frac{X^{1/2}}{|4a/\delta|} \cdot \frac{X^{1/2}}{|4a/\delta|} \left( \frac{|a|^{3/2} X^{1/2}}{a^2 \cdot |4a/\delta|} + 1 \right).
\end{align*}
Under the given constraints on $u'$, the number of possible choices for it is limited with $O(|\frac{4a}{\delta}|)$. 
\begin{align*}
    &\sum_{0 < |a| \ll X^{1/2}} \sum_{\substack{\delta \mid 4a}} \sum_{\substack{|u'| \leq |4a| \\ \gcd(u', 4a) = \delta\\
}}  \frac{X^{1/2}}{|4a/\delta|} \cdot \frac{X^{1/2}}{|4a/\delta|} \left( \frac{|a|^{3/2} X^{1/2}}{a^2 \cdot |4a/\delta|} + 1 \right)  \\ \ll \quad
 &\sum_{0 < |a| \ll X^{1/2}} \sum_{\substack{\delta \mid 4a }}     |4a/\delta| \cdot\frac{X^{1/2}}{|4a/\delta|} \cdot \frac{X^{1/2}}{|4a/\delta|} \left( \frac{|a|^{3/2} X^{1/2}}{a^2 \cdot |4a/\delta|} + 1 \right) \\\ll_{\epsilon}\quad
 &X^{7/4+\epsilon}
\end{align*}
The last part of the proof uses the fact that $|\frac{4a}{\delta}| \geq 1$.
\end{proof}

\begin{customproof}[Theorem \ref{IJ-Sieve}]
As previously mentioned, our goal is to determine the number of $\text{PGL}_2(\mathbb{Z})$-orbits in $\mathcal{L}$ with a $Q$-invariant greater than $Q$ and a height less than $X$. According to the averaging method, the number of such orbits can be expressed as:
\begin{equation}\label{average-sieve}
\int_{(tnk)\in \mathcal{F}_{\text{PGL}_2}} \#\big\{f \in \mathcal{L} \cap (tn)G_0 \cdot R^{(i)}_{X^3} : Q(f) > Q,  |J(f)| \ll X \big\} t^{-2} d^{\times}t  dndk.
\end{equation}
As defined in \cite{2sel-e}, $L^{(i)}$ is absolutely bounded, and $R^{(i)}_{X^3}$ denotes the scaling of $L^{(i)}$ by $\mathbb{R}_{>0}$, where $\max(|I|^3, |J|^2)$ is $O(X^3)$. Additionally, $G_0$ is a compact, $K$-invariant subset of $\text{PGL}_2(\mathbb{R})$. The restriction on $J$ will be addressed separately to accommodate the new height setup.

As the element $\gamma$ ranges over $\mathcal{F}_{\text{PGL}_2}$, the set $\gamma G_0 \cdot R^{(i)}_{X^3}$ undergoes distortion. It is known from \cite{2sel-e} that the coefficients of $R_X^{(i)}$ are bounded by $O(X^{1/6})$, while this scaling ensures that the coefficients of forms in $R_{X^3}^{(i)}$ are bounded by $O(X^{1/2})$.

In particular, when $\gamma = tnk$ in Iwasawa coordinates, the five coefficients $a$, $b$, $c$, $d$, and $e$ of any element in $\gamma G_0 \cdot R^{(i)}_{X^3}$ satisfy:
\begin{equation} \label{box-bound}
a \ll \frac{X^{1/2}}{t^4}, \quad b \ll \frac{X^{1/2}}{t^2}, \quad c \ll X^{1/2}, \quad d \ll t^2 X^{1/2}, \quad e \ll t^4 X^{1/2}.
\end{equation}
Let $tnk$ be fixed. Lemma \ref{nonzero-a-bounding} implies:
\begin{equation*}
    \# \Big\{f=(a,b,c,d,e) \in  \mathcal{L}\cap (tn)G_0\cdot R^{(i)}_{X^3}: a(f)\neq 0, \, |J(f)| \ll X \Big\} = O_{\epsilon}(X^{7/4+\epsilon}).
\end{equation*}
We only need to try to bound those quartics  $f=(0,b,c,d,e) \in \mathcal{L}\cap (tn)G_0.R^{(i)}_{X}$ with $J$-invarient less than $X$ and $b>Q$ since $Q$-invariant in this case is $b$. Input $a$ equal to zero in equation \eqref{J}, this implies that 
\begin{equation}\label{J-for-a=0}
    J= 9bcd-27eb^2-2c^3.
\end{equation}
To find these quartic forms, let $b$ be fixed, and assume we are counting those of the form $(0, b, c, d, e)$ such that $\gcd(3b, c) = \eta$. Thus, we obtain:
\begin{align*}
    &\# \Big\{f=(0,b,c,d,e) \in  \mathcal{L}\cap (tn)G_0\cdot R^{(i)}_{X^3}: |b|>Q, \, |J(f)| \ll X \Big\}\\ \ll \quad
    &\sum_{ Q<|b|\ll X^{1/2}/t^2} \sum_{\eta | 3b} \sum_{\substack{|c| \ll X^{1/2}\\\gcd(c,3b) =\eta}} \#\Big\{(d,e)\in \big((tn)G_0\cdot R^{(i)}_{X^3}\big)_{(0,b,c)}: |9bcd-27eb^2-2c^3|\ll X \Big\}.
\end{align*}
Fixing $b$ and $c$ will determine $J$ modulo $9b\eta$, meaning that the number of possible values for $J$ is $O(X/|\eta b|)$. Specifically, we have $J \equiv -2c^3 \pmod{9b\eta}$. Next, we count the number of possible values for $d$ given that $b$, $c$, and $J$ are fixed. From equation \eqref{J-for-a=0}, we see that $d$ is determined modulo $3b/\eta$, while $e$ is fully determined by $b$, $c$, $d$, and $J$.  This shows that the number of quartic forms $f$ in $(0,b,c,-,-) \in \mathcal{L} \cap (tn)G_0\cdot R^{(i)}_{X^3}$ with a fixed $J$-invariant is bounded by $O(X^{1/2}t^2/|b/\eta|)$. Therefore, we have:

\begin{align*}
&\sum_{ Q<|b|\ll X^{1/2}/t^2} \sum_{\eta | 3b} \sum_{\substack{|c| \ll X^{1/2}\\\gcd(c,3b) =\eta}} \#\Big\{(d,e)\in \big((tn)G_0\cdot R^{(i)}_{X^3}\big)_{(0,b,c)}: |9bcd-27eb^2-2c^3|\ll X \Big\} \\ \ll \quad
&\sum_{ Q<|b|\ll X^{1/2}/t^2} \sum_{\eta | 3b} \sum_{\substack{|c| \ll X^{1/2}\\\gcd(c,3b) =\eta}}
\sum_{\substack{|J|\ll X\\
J \equiv -2c^3\\ \pmod {9b\eta}}} \#\Big\{(d,e) \big((tn)G_0\cdot R^{(i)}_{X^3}\big)_{(0,b,c)}: J(0,b,c,d,e)=J \Big\}
\\ \ll \quad
&\sum_{Q<|b|\ll X^{1/2}/t^2} \sum_{\eta|3b} \frac{X^{1/2}}{|\eta|}\cdot \frac{X}{|b\eta|}\cdot \frac{t^2X^{1/2}}{|b/\eta|} \ll_{\epsilon}
\frac{t^2X^{2+\epsilon}}{|Q|}.
\end{align*}
To complete the proof, we need to estimate equation \eqref{average-sieve}. Consider two different ranges for $t$. If $t \gg X^{1/8}$, then equation \eqref{box-bound} implies that $a$ must be zero. In this case, the $Q$-invariant equals $b$ and, by assumption, must be nonzero. This condition further implies that $t$ is bounded by $O(X^{1/4})$, as follows from equation \eqref{box-bound}. We express the integral in \eqref{average-sieve} as follows, since we assume the integral over $K$ has value 1.
\begin{equation*}
    \int_{t>c} \int_n \#\left\{f=(a,b,c,d,e) \in \mathcal{L}\cap (tn)G_0\cdot R^{(i)}_{X^3}: Q(f)>Q, \, |J(f)|<X\right\}  t^{-2}d^{\times}tdn,
\end{equation*}
where $c$ is a fixed constant.

\noindent
For $X^{1/8} \ll t \ll X^{1/4}$, we have:
\begin{align*}
&\int_{X^{1/8}\ll t\ll X^{1/4}} \int_n \#\left\{f=(0,b,c,d,e) \in \mathcal{L}\cap (tn)G_0\cdot R^{(i)}_{X^3}:|b|>Q, \, |J(f)|<X\right\}  t^{-2}d^{\times}tdn \\ \ll_{\epsilon} \quad
& \int_{X^{1/8}\ll t\ll X^{1/4}} \int_n  \frac{t^2X^{2+\epsilon}}{Q} t^{-2}d^{\times}tdn \\ \ll_{\epsilon} \quad
&\frac{X^{2+\epsilon}}{Q}  
\end{align*}
For $t\ll X^{1/8}$, we have:
\begin{align*}
&\int_{c<t\ll X^{1/8}} \int_n \#\Big\{f \in \mathcal{L}\cap (tn)G_0.R^{(i)}_{X}:Q(f)>Q, \, |J(f)|\ll X \Big\}  t^{-2}d^{\times}tdn\\ \ll_{\epsilon} \quad
& \int_{c<t\ll X^{1/8}} \int_n  \left(\frac{t^2X^{2+\epsilon}}{Q}+X^{7/4+\epsilon}\right) t^{-2}d^{\times}tdn \\ \ll_{\epsilon} \quad
& \frac{X^{2+\epsilon}}{Q}+X^{7/4+\epsilon}
\end{align*}
This completes the proof of the Theorem \ref{IJ-Sieve}. 
\end{customproof}
 
For large family of elliptic curves define the local mass $M_p(F)$ by
\begin{equation}\label{M_p(F)-def}
M_p(F) \vcentcolon= \int_{(I,J) \in \Inv_p(F)}dIdJ.
\end{equation}
Let us also define the following analogues at infinity of $M_p(F)$
\begin{equation*}
    M_{\infty}(F;X) \vcentcolon= \int_{\substack{\Inv_{\infty}(F)\\h_e(I,J)<X}} \,dI\,dJ
\end{equation*}
We have the following theorem, which follows from Theorem 3.17 in \cite{main}:
\begin{theorem}\label{elliptic-counting}
    Let $F$ be a large family of elliptic curves and let $N(F;X)$ denote the number of elliptic curves $E\in F$ such that $h_e(E)<X$. Then 
\begin{equation*}
        N(F;X) = M_{\infty}(F;X) \prod_p M_p(F) + o(X^2).    \end{equation*}
\end{theorem}

\subsection{Proof of the main theorem}
According to \cite{main}, building on results originally established in \cite{birch1963notes} and \cite{cremona2008algorithms}, the non-identity elements in the $2$-Selmer group of an elliptic curve $E^{I,J}$ over $\mathbb{Q}$ correspond one-to-one with $\text{PGL}_2(\mathbb{Q})$-equivalence classes of locally soluble integral binary quartic forms having invariants $2^4I$ and $2^6J$, provided these forms do not have a linear factor.

In Section 3, we computed the number of $\PGL_2(\mathbb{Z})$-orbits of generic integral quartic forms with bounded height. Furthermore, the proofs in Subsection \ref{reducibility subsection} indicate that the contribution of integral quartic forms with two irreducible factors is negligible.

Our goal is to count each $\text{PGL}_2(\mathbb{Z})$-orbit of integral binary quartic forms while assigning a specific weight to each orbit. Specifically, we assign a weight of $\frac{1}{n(f)}$ to each orbit, where $n(f)$ denotes the number of $\text{PGL}_2(\mathbb{Z})$-orbits contained within the $\text{PGL}_2(\mathbb{Q})$-equivalence class of the form $f$.

As discussed in \cite{main}, it is  sufficient to perform the count with a weight of $\frac{1}{m(f)}$, where  

\begin{equation*}
m(f) \vcentcolon= \sum_{f'\in B(f)}\frac{\#\text{Aut}_{\mathbb{Q}}(f)}{\#\text{Aut}_{\mathbb{Z}}(f')}.
\end{equation*}
Here, $B(f)$ denotes a set of representatives for the action of $\text{PGL}_2(\mathbb{Z})$ on the $\text{PGL}_2(\mathbb{Q})$-equivalence class of $f$ in $V_{\mathbb{Z}}$. The terms $\text{Aut}_{\mathbb{Q}}(f)$ and $\text{Aut}_{\mathbb{Z}}(f)$ represent the stabilizer of $f$ in $\text{PGL}_2(\mathbb{Q})$ and $\text{PGL}_2(\mathbb{Z})$, respectively. Since the vast majority of $\text{PGL}_2(\mathbb{Z})$-orbits have a trivial stabilizer, it follows that $m(f) = n(f)$ for generic forms.

The significance of the global weight $m(f)$ for a form $f$ with nonzero discriminant lies in its equality to the product $\prod_{p} m_p(f)$, where $m_p(f)$ is defined as follows:
\begin{equation*}
    m_p(f) \vcentcolon= \sum_{f' \in B_p(f)} \frac{\# \Aut_{\mathbb{Q}_p}(f)}{\# \Aut_{\mathbb{Z}_p}(f')}.
\end{equation*}
Here, $B_p(f)$ denotes a set of representatives for the action of $\text{PGL}_2(\mathbb{Z}_p)$ on the $\text{PGL}_2(\mathbb{Q}_p)$-equivalence class of $f$ in $V_{\mathbb{Z}_p}$. The terms $\Aut_{\mathbb{Q}_p}(f)$ and $\Aut_{\mathbb{Z}_p}(f)$ represent the stabilizers of $f$ in $\text{PGL}_2(\mathbb{Q}_p)$ and $\text{PGL}_2(\mathbb{Z}_p)$, respectively. The proof of this result can be found in Proposition 3.6 of \cite{main}. Consequently, the function $m(f)^{-1} \cdot \chi$, where $\chi$ denotes the characteristic function of locally soluble integral binary quartic forms, serves as a valid weight function, enabling us to apply Theorem \ref{weight-counting-quartic}.

Consider any large family of elliptic curves, and let $S(F)$ represent the set of locally soluble integral binary quartic forms characterized by the invariants $2^4I$ and $2^6J$, where $(I, J) \in \text{Inv}(F)$. By assigning each element $f \in S(F)$ a weight of $\frac{1}{m(f)}$, we find that the weighted count of irreducible $\text{PGL}_2(\mathbb{Z})$ orbits with height $h_{36}$ less than $3\cdot2^4X$ in $S(F)$ is asymptotically equivalent to the count of non-identity elements in the $2$-Selmer group for all elliptic curves with height $h_e$ less than $X$ in $F$. Let $S_p(F)$ denote the $p$-adic closure of $S(F)$ in $V_{\mathbb{Z}_p}$. Referring to Theorem \ref{weight-counting-quartic}, our objective is then reduced to the calculation of certain local densities. These local densities are addressed in Proposition 3.9 of \cite{main} and are computed as follows:

\begin{equation}\label{local-density-calc}
    \int_{S_p(F)}\frac{1}{m_p(F)}\, df = \left| \frac{2^{10}}{27} \right|_p \cdot \Vol\left( \text{PGL}_2(\mathbb{Z}_p)\right) \cdot M_p(V, F) ,
\end{equation}
where
\begin{equation*}
    M_p(V, F) \vcentcolon= \int_{(I, J) \in \Inv_p(F)} \frac{\#\left( E^{I,J}(\mathbb{Q}_p) / 2E^{I,J}(\mathbb{Q}_p) \right)}{\#E^{I,J}(\mathbb{Q}_p)[2]} \, dI \, dJ.
\end{equation*}

\noindent {\bf Proof of Theorem \ref{2-Selmer average}:} 
Let $S_{\infty}(F)$ denote the collection of all $\mathbb{R}$-solvable binary quartic forms in $V_{\mathbb{R}}$ whose invariants belong to $\text{Inv}_{\infty}(F)$. We define $M_{\infty}(V,F;X)$ analogously as  
\begin{equation*}
     M_{\infty}(V,F;X) \vcentcolon= \int_{\substack{(I,J)\in\Inv_{\infty}(F)\\h_{36}(I,J)<3\cdot2^4X}} \frac{\# \left(E^{I,J}(\R)/2E^{I,J}(\R)\right)}{\#E^{I,J}(\R)[2]}\,dI\,dJ.
\end{equation*}
As mentioned earlier,  
\begin{equation*}
        \sum_{\substack{E\in F\\h_e(E)<X}}\left(\#S_2(E)-1\right)
\end{equation*}
is asymptotically equal to the number of locally soluble $\PGL_2(\Z)$-orbits on $S(F)$ with height $h_{36}$ bounded by $3\cdot 2^4X$ and no rational linear factor, where each orbit $\PGL_2(\Z)\cdot f$ is counted with weight $1/m(f)$.  

Thus, by Theorem \ref{weight-counting-quartic} and the results of Propositions 3.6, 3.9, and 3.18 in \cite{main}, we obtain  
\begin{align*}
        \sum_{\substack{E\in F\\h_e(E)<X}}\left(\#S_2(E)-1\right) 
        &\leq  N_{36}(V_{\Z}\cap S_{\infty}(F);3\cdot2^4X)\prod_{p} \int_{S_{p}(F)}\frac{1}{m_p(f)}df+o(X^2)\\
         &=  \frac{1}{27}\Vol(\PGL_2(\Z)/\PGL_2(\R)) M_{\infty}(V,F;X) \prod_{p} \left|\frac{2^{10}}{27}\right|_p \Vol(\PGL_2(\Z_p)) M_p(V,F)+o(X^2)\\
         &=\frac{3^4\cdot 2^{10}}{2\cdot27}  \Vol(\PGL_2(\Z)/\PGL_2(\R))\prod_{p} \left|\frac{2^{10}}{27}\right|_p \Vol(\PGL_2(\Z_p)) M_p(V,F)+o(X^2).
\end{align*} 
Additionally, Theorem \ref{elliptic-counting} gives  
\begin{align*}
\sum_{\substack{E\in F\\h_e(E)<X}}1 &=M_{\infty}(F;X) \prod_p M_p(F) + o(X^2)\\
&= 3^4X^2 \prod_p M_p(F) + o(X^2).
\end{align*}
Considering the ratios $\frac{M_{p}(V, F)}{M_{p}(F)}$, which are computed in \cite{main} using results from \cite{quotient}, we have  
\begin{equation*}
    \frac{M_p(V,F)}{M_p(F)}
    = \frac{\int_{(I,J) \in \text{Inv}_p(F)} \frac{\#(E^{I,J}(\mathbb{Q}_p)/2E^{I,J}(\mathbb{Q}_p))}{\#E^{I,J}(\mathbb{Q}_p)[2]} dI dJ}
    {\int_{(I,J) \in \text{Inv}_p(F)} dI dJ}
    =
    \begin{cases}
        1 & \text{if } p \neq 2, \\
        2 & \text{if } p = 2.
    \end{cases}
\end{equation*}
Thus, we conclude  
\begin{equation*}
 \lim_{X \to \infty} \frac{\sum\limits_{\substack{E\in F\\h_{e}(E)<X}}\left(\#S_2(E)-1\right)}{\sum\limits_{\substack{E\in F\\h_{e}(E)<X}}1} \leq \Vol(\PGL_2(\Z)/\PGL_2(\R))\prod_{p} \Vol(\PGL_2(\Z_p)).
\end{equation*}
The constant on the right-hand side of the above expression is  
\begin{equation*}
2\zeta(2)\prod_{p}(1-p^{-2})=2,
\end{equation*}
which is the Tamagawa number of $\PGL_2(\Q)$.
\bibliographystyle{plain}  % Or another style like alpha, amsplain, etc.
\bibliography{main}  % Link to the .bib file (do not include the .bib extension)
\end{document}